\documentclass[11pt,oneside,english]{amsart}
\usepackage[T1]{fontenc}
\usepackage[latin9]{inputenc}
\usepackage{prettyref}
\usepackage{calc}
\usepackage{amsthm}
\usepackage{amstext}
\usepackage{amssymb}
\usepackage{graphicx}
\usepackage{esint}

\makeatletter
\numberwithin{equation}{section}
\numberwithin{figure}{section}
\theoremstyle{plain}
\newtheorem{thm}{\protect\theoremname}[section]
  \theoremstyle{plain}
  \newtheorem{lem}[thm]{\protect\lemmaname}
  \theoremstyle{remark}
  \newtheorem{rem}[thm]{\protect\remarkname}
  \theoremstyle{plain}
  \newtheorem{prop}[thm]{\protect\propositionname}
  \theoremstyle{plain}
  \newtheorem{cor}[thm]{\protect\corollaryname}
  \theoremstyle{remark}
  \newtheorem*{acknowledgement*}{\protect\acknowledgementname}

\usepackage{url}
\usepackage{color}
\usepackage{marginnote}

\newrefformat{prop}{Proposition \ref{#1}}
\newrefformat{cor}{Corollary \ref{#1}}
\newrefformat{sub}{Section \ref{#1}}
\newrefformat{rem}{Remark \ref{#1}}
\newrefformat{def}{Definition \ref{#1}}
\newtheorem*{conjecture*}{{\bf Conjecture}}

\usepackage[numbers,sort&compress,square]{natbib}

\makeatother

\usepackage{babel}
  \providecommand{\acknowledgementname}{Acknowledgement}
  \providecommand{\corollaryname}{Corollary}
  \providecommand{\lemmaname}{Lemma}
  \providecommand{\propositionname}{Proposition}
  \providecommand{\remarkname}{Remark}
\providecommand{\theoremname}{Theorem}

\begin{document}
\title{The subleading order of two dimensional cover times}
\begin{abstract}
The $\varepsilon$-cover time of the two dimensional torus by Brownian motion is the time it takes for the process to come within distance $\varepsilon>0$ from any point. Its leading order in the small $\varepsilon$-regime has been established by Dembo, Peres, Rosen and Zeitouni [{\it Ann. of Math.}, {\bf 160} (2004)]. In this work, the second order correction is identified. The approach relies on a multi-scale refinement of the second moment method, and draws on ideas from the study of the extremes of branching Brownian motion. \end{abstract}
\maketitle

\authors{David Belius\footnote{Centre de recherches math\'{e}matiques, Universit\'{e} de Montr\'{e}al, 2920 chemin de la Tour, Montr\'{e}al, QC H3T 1J4, Canada;  \email{david.belius@cantab.net}. Supported by the Swiss National Science Foundation, the Centre de Recherches Math\'{e}matiques and the Institut des Sciences Math\'{e}matiques.}, Nicola Kistler\footnote{
Department of Mathematics, College of Staten Island, City University of
New York, 2800 Victory Boulevard, Staten Island 10314 New York; \email{nicola.kistler@csi.cuny.edu}}}.

\section{Introduction}

A fundamental question one can ask about a Markov process concerns the time 
it takes to visit all of the state space. In this article
we study this question for Brownian motion in the two dimensional
Euclidean torus $\mathbb{T}=\left(\mathbb{R}/\mathbb{Z}\right)^{2}$, i.e.
the box $[0,1)^{2}$ with periodic boundary. More precisely, we study the time it takes for the process to come within distance $\varepsilon>0$ of every point, in the small $\varepsilon$-regime. This time is referred to as the \emph{$\varepsilon$-cover time}, and is denoted by $C_{\varepsilon}$.

The $\varepsilon$-cover time (and its discrete version, the cover time) has been extensively studied over the past decades.
For the two dimensional torus upper and lower bounds on the expected cover time were proven by Matthews \cite{MatthewsCoveringProblemsForBM} and Lawler \cite{LawlerOnTheCovering}. The gap between these bounds was closed by Dembo, Peres, Rosen and Zeitouni \cite{DemboPeresEtAl-CoverTimesforBMandRWin2D}, who proved the law of large numbers,
\begin{equation}
\frac{C_{\varepsilon}}{\frac{1}{\pi}\log\varepsilon^{-1}}=\big(1+o\left(1\right)\big)\log\varepsilon^{-2}.
\label{eq:DPRZResult}
\end{equation}

The question of lower order corrections, and, in general, fluctuations, was left open.
By analogy to related models (for instance the two dimensional Gaussian free field) one may expect the presence of a ``$\log \log$-correction term'', see \cite{Ding-CoverTimefor2DLattice,BramsonZeitouniTightnessoftheRecenteredMaxOf2dGFF}. No suggestion for the exact form of this conjectured term (i.e. including multiplicative constant) appears in the literature.
In this work we settle this issue by establishing the following asymptotics, 
\begin{equation}
\frac{C_{\varepsilon}}{\frac{1}{\pi}\log\varepsilon^{-1}}=\log\varepsilon^{-2}-
\big(1+o(1)\big)\log\log\varepsilon^{-1},
\label{eq:MainResultIntro}
\end{equation}
in probability, as $\varepsilon \downarrow 0$.

The law of large numbers \prettyref{eq:DPRZResult} is somewhat surprising. In fact, $C_{\varepsilon}$ is the maximum of all hitting times of balls in the torus of radius $\varepsilon$. To first approximation, these hitting times are exponentially distributed, with mean given by the denominator of \eqref{eq:DPRZResult}. Now, since hitting times of highly overlapping balls should be roughly the same, one may take the maximum over a ``packing'' of $\sim \varepsilon^{-2}$ balls of radius $\varepsilon$ which do not overlap too much: assuming that these exponentials are independent, one indeed recovers \eqref{eq:DPRZResult}.
In other words, despite (what turn out to be) long-range correlations between hitting times of disjoint balls, the \emph{leading order} of the maximum behaves as in the independent setting.

On the other hand, since the maximum of  independent exponentials does {\it not} exhibit any correction, \eqref{eq:MainResultIntro}
is testament to the presence of these correlations.
As it turns out, the field of hitting times is {\it log-correlated}, i.e. correlations decay (roughly) with the logarithm of the distance.
The prototypical example of such a random field is branching Brownian motion, BBM for short. Our proof of \eqref{eq:MainResultIntro} goes via a multi-scale analysis which is much inspired by the picture which has emerged in the study of the extremes of BBM \cite{KistlerDerridasRandomEnergyModelsBeyondSpinGlasses,BramsonMaxDisplacementofBBM,ABKGenealogy}.
The correction term in \eqref{eq:MainResultIntro} corresponds to the well-known $3/2$-correction first identified by Bramson \cite{BramsonMaxDisplacementofBBM} for the
maximum of BBM (see the end of the introducton).

The above can be constrasted with the situation for discrete torii of higher dimensions, about which much more is known, see  
\cite{BeliusCTDT, denHollanderGoodmanExtremalGeometryofBrownianPorousMedium}. For the cover time of the discrete torus \cite{BeliusCTDT} proves that in $d\geq 3$ there is no correction term to the leading order, and that the fluctuations follow the Gumbel distribution\footnote{ Although it does not appear in the literature, it is expected that the behaviour of the $\varepsilon$-cover time of the Euclidean torus
in $d\geq 3$ is the same as in the discrete setting.}, just as for the maximum of independent exponentials. The reason for this behavior is the local transience of Brownian motion in $d\geq 3$, which leads to weak correlations among hitting times: weak enough for the extremes of the field to behave like the extremes of a field of independent random variables, even at the level of fluctuations.

In $d=2$, the local recurrence of Brownian motion leads to intricate long-range correlations among hitting times, and these are responsible for a radically different process of covering.
%
%
Perhaps more important than the numerical value of the subleading order identified in \prettyref{eq:MainResultIntro} is the description that our proof provides of this covering process: In short, at each scale the torus can be thought of as being tiled by neighbourhoods, where the scale corresponds to the neighbourhoods' size. Because of ergodicity Brownian motion has a tendency to spend a similar amount of time in most neighbourhoods at each scale (the effect becomes weaker at smaller scales). But to leave an $\varepsilon$-ball unvisited until very late Brownian motion needs to spend atypically little time in that ball's neighbourhoods (this effect becomes stronger at smaller scales). This ``conflict'' makes it harder to ``miss'' a small ball, thus making the cover time happen a little bit faster and giving rise to the subleading correction. Furthermore, the strategy\footnote{We do not prove that this is the \emph{only} strategy, but \cite{ABKGenealogy} proves the analogous statement for Branching Brownian Motion, and this seems very likely to carry over to our setting.} needed to avoid a small ball up until right before the $\varepsilon$-cover time turns out to be to spend \emph{relatively more} time in the intermediate scales. These phenomena can be considered instances of \emph{entropic repulsion}.


As in  \cite{DemboPeresEtAl-CoverTimesforBMandRWin2D}, we control hitting times via excursions between concentric circles at different scales, relying on an implicit tree structure. 
 Our main contribution is the identification of the mechanism by which this approximate tree structure gives rise to the covering behaviour described above, and the discovery of a concrete analogy to branching Brownian motion. Armed with this analogy we are able to apply methods from the study of branching Brownian motion to prove \prettyref{eq:MainResultIntro}.

\begin{figure}
\includegraphics[scale=0.8,clip,trim=585 650 600 0]{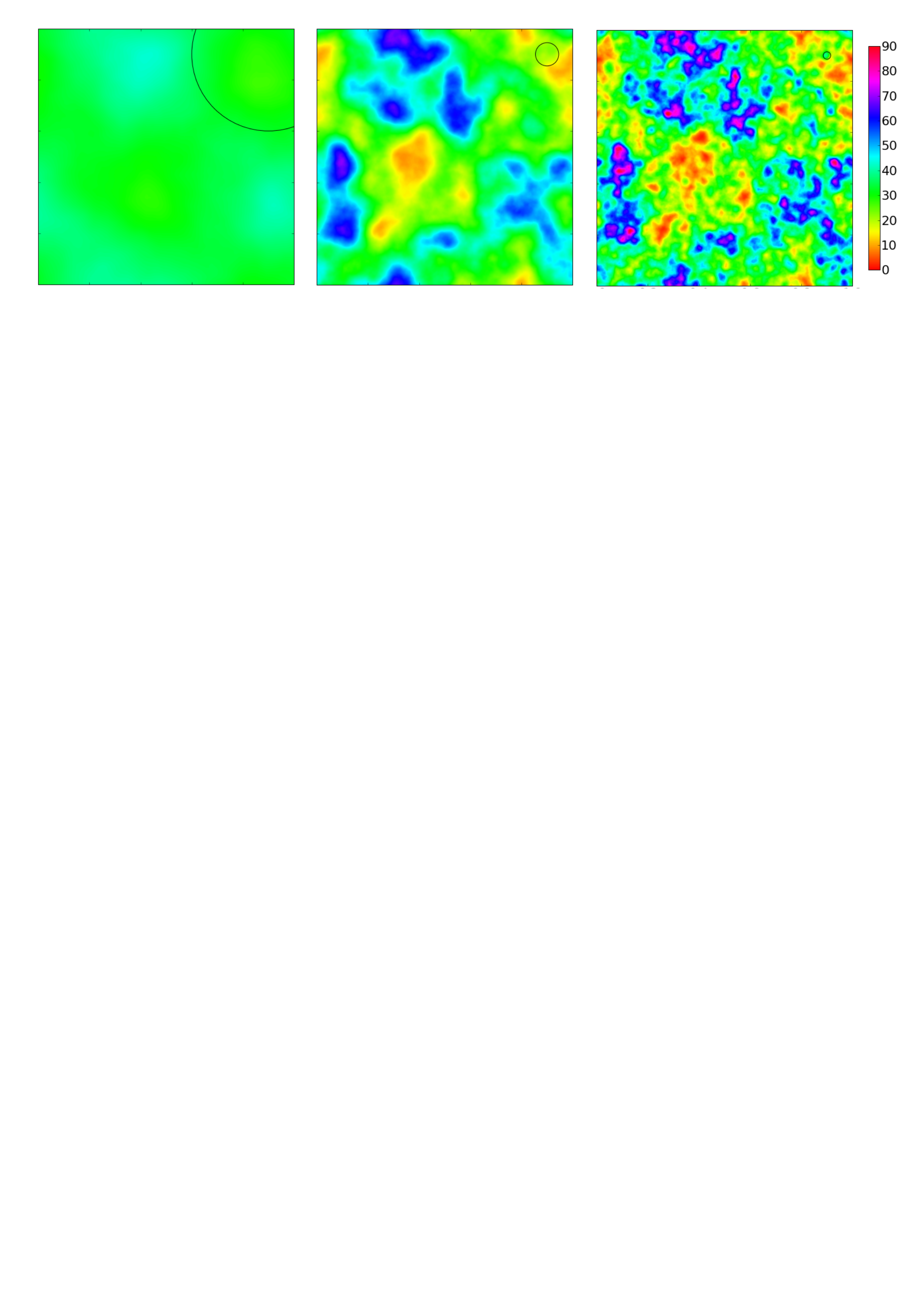}

\framebox{\begin{minipage}[t]{1\columnwidth}%
This simulation shows the occupation times of balls in the torus at three
different scales. Brownian motion is run up to time $10$ and the
intensity of each pixel is given by the time spent in a ball centered
at that pixel. The radius is indicated by the ball in the upper-right
corner, and the occupation times are rescaled by a factor proportional
to its area. The traversal process of a point (see \eqref{eq:bla}) can be thought of as
a proxy for the occupation times of balls around that point.\label{fig:Simulation}
The picture hints at the approximate hierarchical structure: at the
first scale all occupation times are essentially the same. At the
second scale the torus has been split into regions of high, moderate
and low occupation time. At the third scale these regions have been further
subdivided.\end{minipage}}

\caption{Effect of approximate hierarchical structure in simulation of occupation times.}
\end{figure}

\subsection{A sketch of the proof}
In the following, $F$ denotes a set of $\varepsilon^{-2}$ points in the torus, scattered in such a 
way that the balls of radius $\varepsilon$ centered at these points do not overlap too much.  
\subsubsection{{\bf (Failure of) vanilla second moment method}}
A classical approach for the study of extremes of random fields goes via the so-called {\it second moment method}, i.e. the comparative study of first and second moment of a suitably chosen quantity. In the case of cover times,  a natural candidate is
\begin{equation}
Z(m) = \mbox{number of }y\in F\mbox{ that such that }B\left(y,\varepsilon\right)\mbox{ is hit after time }m.\label{eq:ZDefIntro}
\end{equation}
Assuming that hitting times are approximately exponential with mean $\left(1/\pi\right)\log\varepsilon^{-1}$, we have: 
\begin{equation}
\mathbb{E}\left[Z\left(m\right)\right]\approx \varepsilon^{-2} \exp\left({ - \frac{m}{\frac{1}{\pi} \log \varepsilon^{-1}} }\right).\label{eq:ZExpIntro}
\end{equation}
Note that this is vanishing for $\varepsilon \downarrow 0$ if
\begin{equation}
m > \frac{1}{\pi}(1+\delta) 2 \left(\log \varepsilon^{-1} \right)^2 \mbox{ and } \delta>0.\label{eq:mforfirstmoment}
\end{equation}
By the Markov inequality, one immediately obtains an upper bound on the {\it leading} order of the $\varepsilon$-cover time (under the exponentiality assumption). In hindsight, this bound is tight. The analysis of the second moment is however inconclusive: it does not yield a matching lower bound, due to strong correlations of hitting times. To overcome this obstacle one needs a more sophisticated multi-scale analysis \cite{DemboPeresEtAl-CoverTimesforBMandRWin2D}. At the level of the subleading order the situation is even more delicate, since already the analysis of the first moment is inconclusive. In fact, if we let
\begin{equation}
m\left(s\right) = \frac{1}{\pi}\log\varepsilon^{-1}\left\{ \log\varepsilon^{-2}-s\log\log\varepsilon^{-1}\right\} , \quad s\in \mathbb{R},\label{eq:mwithcorrection}
\end{equation}
one has (cf. \eqref{eq:ZExpIntro} and \eqref{eq:mforfirstmoment})
\begin{equation}
\mathbb{E}\left[Z\left(m\left(s\right)\right)\right]\approx\left(\log\varepsilon^{-1}\right)^{s},\label{eq:Expofzms}
\end{equation}
which explodes if $s>0$, although \eqref{eq:MainResultIntro}  claims that $Z\left(m\left(s\right)\right)=0$ with high probability also for $0 < s < 1$. The source of the problem is 
easily identified: by linearity of the expectation, we are completely dismissing the  correlations of the field of hitting times, but these are severe enough to have an impact at the level of the subleading order.

To get around this $Z(m)$ should be replaced by a truncated version $\widetilde Z(m)$, whose first moment already encodes information about the correlation structure. This approach can be used to derive the subleading order of branching Browian motion (see \cite{KistlerDerridasRandomEnergyModelsBeyondSpinGlasses}). The challenge is identifying the right truncation procedure for the cover time of the torus.
\subsubsection{{\bf Traversal processes}}
As it turns out, a suitable truncation is formulated in terms of a \emph{traversal process} associated to each of the $\varepsilon^{-2}$ points in $F$. This process captures the amount of time Brownian motion spends in the neighbourhood of the point $y$ at the different scales in an associated tower of concentric balls (see \prettyref{fig:Simulation}). The scales are represented by $L\approx\log\varepsilon^{-1}$ geometrically
growing concentric balls,
\begin{equation}
B\left(y,\varepsilon\right)=B\left(y,r_{L}\right)\subset B\left(y,r_{L-1}\right)\subset\ldots\subset B\left(y,r_{1}\right)\subset B\left(y,r_{0}\right)\label{eq:scales},
\end{equation}
around each $y$, where $r_{l} = e \times r_{l+1} = e^{L-l}\varepsilon$
is the ``size'' of the $l-$th scale. We measure time spent in a ball $B\left(y,r_{l}\right)$ by the number of \emph{traversals} made from scale $l$ to scale $l+1$, that is the number of times that Brownian motion moves from the exterior of $B\left(y,r_{l}\right)$
to the interior of $B\left(y,r_{l+1}\right)$. More precisely, we count the number
of such traverals that take place during the first $t$ excursions from $\partial B\left(y,r_{1}\right)$ to  $\partial B\left(y,r_{0}\right)$. 
We call this number  $T_{l}^{y,t}$ and view this as a process in $l$, thus obtaining for each $y\in\mathbb{T}$ and initial excursion count $t\ge1$,
\begin{equation}
\begin{array}{c}
\mbox{the \emph{traversal process} }\left(T_{l}^{y,t}\right)_{l\ge0},\mbox{ counting}\\
\mbox{the number of traversals from }\partial B\left(y,r_{l}\right)\mbox{ to }\partial B\left(y,r_{l+1}\right).\label{eq:bla}
\end{array}
\end{equation}
Note that,
\begin{equation}
T_{L-1}^{t,y}=0 \iff B\left(y,\varepsilon\right)\mbox{ is not hit during  }t\mbox{ excursions}\label{eq:iff}
\end{equation}
from $\partial B(y,r_1)$ to $\partial B(y,r_0)$, thus providing a connection between traversal processes, hitting times of balls, and ultimately the $\varepsilon$-cover time.

Since we have one traversal process for each $y \in \mathbb{T}$ there is no explicit hierarchical structure in our construction. However,
the correlations of the processes have a crucial \emph{approximate} hierarchical structure, which underlies the whole approach.
If $y$ and $z$ are at distance of about $r_{k}$, then for $l$ slightly smaller than $k$ the balls of radius $r_{l} \gg r_k$ around $y$ and $z$ will have a very large overlap: they will be almost the same ball. Therefore, one would expect that the number of traversals around $y$ and $z$ at such scales essentially coincide, that is $T^{y,t}_{l} \approx T^{z,t}_{l}$ for $l \ll k$. On the other hand for $l$ larger than $k$ the balls of radius $r_{l} \ll r_k$ around $y$ and $z$ will be disjoint. By the strong Markov property the excursions of Brownian motion in disjoint balls are conditionally independent, and we may therefore expect that the traversal processes of $y$ and $z$ evolve essentially independently at such scales, conditionally on the number of traversals at scale $r_k$ (which should be roughly the same for both). This picture leads one to imagine a tree\footnote{Or more accurately a forest of $\sim r_0^{-2}$ trees, the latter being the number of balls that can be ``packed'' into the highest scale.} of depth $L$ where $y,z \in \mathbb{T}$ at distance of about $r_k$ roughly correspond to leaves whose most recent ancestor is in level $k$ of the tree (see \prettyref{fig:PosOfCirclesTWB}).

The advantage of defining the traversal process in terms of excursions from $\partial B\left(y,r_{1}\right)$ to  $\partial B\left(y,r_{0}\right)$ is that it then becomes a critical \emph{Galton-Watson process} with geometric offspring distribution, due to the ``spacing'' of the scales and a well-known result on the exit distribution of Brownian motion in two dimensions from an annulus (see \prettyref{fig:UpAScaleDownAScale}). A concentration argument for excursion times shows that the time needed to make $t$ excursions is very close to $\frac{1}{\pi}t$, thus providing a way to ``translate'' between excursion counts and the actual time of Brownian motion.
\subsubsection{{\bf Upper bound - barrier}}
The key idea, which eventually leads to the truncation, is based on the insight that a traversal process {\it cannot die out too quickly}. To formalize, consider the following ``barrier'' for the square root of the traversal count,
\begin{equation}
\alpha=\alpha\left(l\right)=\left(1-\frac{l}{L}\right)\sqrt{t}-\left(\log L\right)^2\mbox{\,\ for }l=0,\ldots, L.\label{eq:IntroBarrier}
\end{equation}
This barrier is the linear function interpolating between $\sqrt{t}$ and $0$, shifted downwards slightly (see \prettyref{fig:Illustration-of-barriers}).
It turns out that with high probability,
\begin{equation}
\mbox{no traversal process }T_{l}^{y,t}\mbox{ falls below }\alpha\left(l\right)^2,\mbox{ for }l=0,\ldots,L-1.\label{eq:IntroSmartMarkov}
\end{equation}
%
%
We prove this claim in two steps: first we reduce the ``combinatorial complexity'' at each scale by means of a  {\it packing argument}, and, second, we use a {\it Markov inequality} over the scales (``multi-scale Markov'', cf. \cite{KistlerDerridasRandomEnergyModelsBeyondSpinGlasses}). Roughly $r_l^{-2}$ balls of radius $r_l$ and at mutual distance roughly $r_l$ can be ``packed'' into the torus. By the above intuition that $T_{l}^{y,t}\approx T_{l}^{z,t}$ for $y$ and $z$ at distance smaller than $r_l$ (see also \prettyref{fig:AnnulusPacking}), we expect the minimum of $T^{y,t}_l$ over all $y \in F$ to be essentially the mininum over this packing.
A union bound then shows that the probability that the minimum of $T_{l}^{y,t}$ over all $y\in F$ drops below $\alpha(l)^{2}$ should be at most
\begin{equation}
cr_{l}^{-2}\mathbb{P}\left[T_{l}^{y,t}\le\alpha(l)^{2} \right]\label{eq:OneLevelUnionBound}.
\end{equation}
We derive a large deviation control on $T_{l}^{y,t}$, which allows us to prove that with our choice of $\alpha(l)$ the quantity in \prettyref{eq:OneLevelUnionBound} tends to zero, and what's more, the sum over $l$ of \prettyref{eq:OneLevelUnionBound} tends to zero. Therefore by a union bound over the scales $l=0,\ldots,L-1$ we will be able
to show that no traversal process falls below $\alpha(l)^2$, i.e. derive \prettyref{eq:IntroSmartMarkov}. For the upper bound on the cover time, this ``multi-scale'' use of the Markov inequality is the only place where the correlation structure of the traversal processes is used.

\subsubsection{{\bf Upper bound - matching}}
Now \eqref{eq:IntroSmartMarkov} suggests the following truncated version of the counting random variable $Z(m)$, which also counts balls which are not hit (cf. \eqref{eq:iff}) but furthermore requires the traversal process to stays above $\alpha(l)^2$:
\begin{equation}
\tilde{Z}\left(t\right)=\begin{array}{c}
\mbox{number of }y\in F\mbox{\,\ such that }T_{L-1}^{y,t}=0\\
\mbox{ and }\sqrt{T_{l}^{y,t}}\mbox{ never falls below }\alpha\left(l\right).\end{array}\label{eq:ZDefIntro-1-1}
\end{equation}
When written in terms of the number of scales $L \approx (\log \varepsilon)^{-1}$, the number of excursions that typically take place up to the time $m(s)$  from \eqref{eq:mwithcorrection} turns out to be roughly
\begin{equation}
t\left(s\right) = L\left\{ 2L - s\log L\right\} , \quad s\in \mathbb{R}.\label{eq:introts}
\end{equation}
Therefore to obtain the upper bound in \eqref{eq:MainResultIntro} one has to show that $\tilde{Z}(t(s))=0$ with high probability for $s<1$.

The expectation of $\tilde{Z}\left(t\right)$ can be written as
\begin{equation}
E[\tilde{Z}(t)]=|F| \cdot \mathbb{P}\left[T_{L-1}^{y,t}=0\right] \cdot \mathbb{P}\left[\sqrt{T_{l}^{y,t}}\ge\alpha(l)\mbox{ for }l=0,\ldots,L-1\Big|T_{L-1}^{y,t}=0\right].\label{eq:truncmean}
\end{equation}

Note that when $t=t(s)$ the product of the first two terms is essentially the expectation of the untruncated counting variable $Z(m(s))$, and as such will be equal to $L^{s}$, cf. \eqref{eq:Expofzms} and recall that $L\approx \log \varepsilon^{-1}$. The gain comes from the conditional probability, which plays a fundamental role.
It will become apparent that the mean of the process $\sqrt{T_l^{y,t}}$ when conditioned on $T^{y,t}_{L-1}=0$ 
is well approximated by $\left(1-\frac{l}{L}\right)\sqrt{t}$. Furthermore, we will see that the fluctuations of $\sqrt{T_{l}^{y,t}}$ around this mean behave roughly like a Brownian bridge on $[0,L]$ starting and ending in zero. Therefore the conditional probability in \prettyref{eq:truncmean} roughly behaves as
\begin{equation}
\mathbb{P}\left[X_{l}\ge-\left(\log L\right)^2\mbox{ for }l=0,\ldots,L\right],\label{eq:BBridgeTorusSketch}
\end{equation}
where $X_{l}$ is a Brownian bridge. It is well-known (e.g. by the Ballot
theorem) that this probability is of order $L^{-1}$ (ignoring unimportant $log$-terms). 
Coming back to \eqref{eq:truncmean}, this line of reasoning will lead to
\begin{equation} \label{denominator}
\mathbb{E}\left[\widetilde{Z}\left(t(s)\right)\right] \approx \underset{ L^s }{\underbrace{e^{2L} \cdot e^{-t(s)/L}}} \cdot L^{-1}
\longrightarrow \begin{cases} \infty\mbox{ if }s>1, \\
0\mbox{ if }s<1,
\end{cases}
\end{equation}
where the first term arises because $F$ has $\varepsilon^{-2} \approx e^{2L}$ elements, and the second because the number of excursions until the ball $B(y,\varepsilon)$ is hit turns out to be essentially exponentially distributed with mean $L$.
{\it Matching to unity}, we see that $\mathbb{E}[\widetilde{Z}\left(t\right)]$
is of order one for $t$ close to $t\left(1\right)$. In other words, the $L^{-1} \approx 1/\log\varepsilon^{-1}$ at the end of \eqref{denominator} gives rise to the $\log\log$-correction of the cover time. For $s<1$ the expectation tends to zero, giving the upper bound of \eqref{eq:MainResultIntro}. This will be formalized in Section \ref{sec:UpperBound}.
\subsubsection{{\bf Lower bound}}
As for a  tight lower bound, the approach relies on a key idea related to \eqref{eq:BBridgeTorusSketch}. In fact, it can be proven that a Brownian bridge which is required to stay above the line $-(\log L)^2$ 
 for $l=0, \dots, L$ stays \emph{well} above that line, a phenomenon which is reminiscent of the entropic repulsion appearing in the statistical mechanics of random surfaces, see \cite{ABKGenealogy}. More precisely, it can be shown that with overwhelming probability such a Brownian bridge will typically lie  higher than curves of the form $\min\{l^\delta, (L-l)^{\delta}\}$ for any $0 < \delta < 1/2$. Reformulating back in terms of 
the traversal process, this suggests that we do not lose any information 
by considering  
\begin{eqnarray}
 & \text{those}\; y\in F\;\text{for which the associated square root traversal process}\nonumber \\
 & \sqrt{T_{l}^{y,t}}\;\text{stays above}\;\left(1-\frac{l}{L}\right)\sqrt{t}+\min\{l^{0.49},(L-l)^{0.49}\},\;\text{for}\;0\ll l\ll L,\label{eq:notlosing}
\end{eqnarray}
(see \prettyref{fig:Illustration-of-barriers}).
We count the number of balls which have not been hit during $t$ excursions, and whose associated traversal process satisfies the constraint in \prettyref{eq:notlosing}:
\begin{equation}
\widehat{Z}\left(t\right)=
\mbox{number of }y\in F\mbox{\,\ such that }T_{L-1}^{y,t}=0\mbox{ and }\sqrt{T_{l}^{y,t}}\mbox{ satisfies \eqref{eq:notlosing} }.\label{eq:ZHatIntro}
\end{equation}
The expectation of $\widehat{Z}(t(s))$ turns out to be essentially that of the counting random variable $\tilde{Z}(t(s))$ used for the upper bound, and in particular it tends to infinity for $s>1$ (see \eqref{denominator}). Furthermore the truncation turns out to reduce correlations sufficiently for the second moment to be asymptotically equivalent to the first moment squared. An application of the Payley-Zygmund inequality will therefore establish that, when $s>1$, there will with high probability exist a $y \in F$ whose ball $B(y, \varepsilon)$ is not hit in $t(s)$ excursions. By the aforementioned concentration of excursion times this will provide the lower bound on the $\varepsilon$-cover time from \eqref{eq:MainResultIntro}. This is formalized in Section \ref{sec:Lowerbound}.

Without the truncation the second moment explodes with respect to the first moment squared. When writing the second moment as a sum over $y,z \in F$ of the probability that the traversal processes assosciated with both $y$ and $z$ satisfy the condition in \eqref{eq:ZHatIntro}, one sees that the source of the problem are those pairs of balls which lie at ``mesoscopic'' distance (smaller than $r_0$ but larger than $\varepsilon$). 
The truncation helps by penalising such pairs, since the ``bump'' on top of the line $(1-l/L)\sqrt{t}$ forces the square root traversal processes of $y$ and $z$ at distance roughly $r_k$ to \emph{each} make an atypically extreme jump from $(1-k/L)\sqrt{t}  + \min\{k^{0.49}, (L-k)^{0.49}\}$ to $0$ between scales $k$ and $L-1$. One such jump turns out to be achievable, but two jumps in the same neighbourhood turn out to be too costly for such pairs to contribute significantly to the truncated second moement.
Here it is crucial that the traversal processes decorrelate at scales $l \gg k$, so that one really needs to make two essentially \emph{independent} jumps if $y$ and $z$ are both to satisfy \eqref{eq:ZHatIntro}.
Finally, the number of pairs at distance close to $r_L$ is too small to contribute much to the second moment. Therefore the main contribution to the truncated second moment comes from pairs that are at distance at least $r_0$, and for these pairs the events of satisfying the condition in \eqref{eq:ZHatIntro} turn out to be independent. We choose $r_0$ tending to zero slowly, which means that the overwhelming majority of the pairs are independent. This causes the second moment of $\widehat{Z}$ to be asymptotic to the first moment squared.

The rigorous implementation of this decoupling for scales $l \gg k$ is  arguably the most delicate and technically demanding step  in our approach,
 and will be formalized in \prettyref{sec:TwoPointBound} (see also the statement \prettyref{prop:MainTwoProfileBound} of the main bound and \prettyref{rem:2ndmomentboundremark}).

\subsubsection{{\bf Barrier estimates and excursion time concentration}}

The above sketch rests on being able to control the probability that the traversal process $T^{y,t}_l$ avoids certain barriers. Proving these rigorously turns out to
be delicate, and is carried out in \prettyref{sec:BoundaryCrossingProofs} via a comparision of both a conditioned Galton-Watson process and a Brownian bridge to a Bessel bridge.

Furthermore, we have assumed that we are able to control the time needed to make $t$ excursions. As the subleading correction term we are trying to establish is very small compared to the leading order, we need a very precise bound. The basic recipe for such bounds, used e.g. in \cite{DemboPeresEtAl-CoverTimesforBMandRWin2D} (a large deviation bound on excursion times obtained by estimating their exponential moments using Khasminskii's lemma/Kac moment formula, together with a union bound), turns out to be insufficent. We must complement it with a packing argument to reduce combinatorial complexity in the union bound, and in our large deviation bound we need to exploit the Markovian structure of the excursions of Brownian motion. This is carried out in Section \ref{sec:Concentration}.

\subsection{Relation to branching Brownian motion, or: ``$3/2=1$''}
The heuristics described above rests on the approximate hierarchical
structure: it is absolutely fundamental that the traversal processes of two points at distance
of about $r_{l}$ for a given scale $l$ essentially {\it agree}
at higher scales, and {\it decorrelate} at lower scales. In other words, intuitively one starts with a small collection of traversal processes at the first scale which then branch in each subsequent scale, producing several offspring which evolve as (essentially)
independent traversal processes. The situation is thus reminiscent of 
branching Brownian motion, one of the simplest models with an exact hierarchical structure. A similar procedure of truncation and matching to unity can be applied to establish the level of the maximum of BBM, see \cite{KistlerDerridasRandomEnergyModelsBeyondSpinGlasses}.
The key insight is that with suitably chosen scales (cf. \eqref{eq:scales}) the \emph{square roots} of traversal processes correspond to the trajectories (or ``profiles'') of particles in BBM, and that the truncations applied to profiles in BBM can succesfully be applied to these processes.

A fundamental difference is that in case of BBM, the field consists of correlated {\it Gaussian} random variables, whereas in our case hitting times are (approximately) {\it exponentially} distributed. In particular, 
the {\it tail} of a Gaussian distribution has a polynomial term which exponentials do not 
 have. This has a considerable impact when ``matching to unity'': in the BBM version\footnote{To be precise, a version of BBM with branching at discrete integer times and average branching factor $e^2$, run up to time $L$.} of \eqref{eq:truncmean} and  \eqref{denominator} the middle term corresponds to the probability that a trajectory ends up close to the level of maximum, and has not only has the main part of the Gaussian tail $e^{-t^2/(2L)}$, but also the polynomial term $L^{-1/2}$ (when $t$ is chosen to be of order $L$, as it must be for the exponential part of the tail to match the combinatorial complexity $e^{2L}$). The last term in \eqref{eq:truncmean} and \eqref{denominator} is essentially the same barrier crossing probability also for BBM, but applied to the trajectory of a particle up to time $L$, and it also has order $\sim L^{-1}$, giving a total contribution from polynomial terms of $L^{-3/2}$. This gives rise to well-known BBM correction involving $3/2$ when ``matching to unity''.

Thus the subleading correction for BBM (and by extension the Gaussian free field on the two dimensional torus \cite{BramsonZeitouniTightnessoftheRecenteredMaxOf2dGFF}) correctly ``predicts'' the correction term of the $\varepsilon$-cover time of the torus, once this small difference in the tail is taken into account. The subleading order for the cover time of the tree, which was established in \cite{DingZeitouni-SharpEstimateForCTOnTree}, can also be ``predicted'' in the same manner; in this case the subleading correction coincides numerically\footnote{To verify this one must rearrange (0.1) of \cite{DingZeitouni-SharpEstimateForCTOnTree} appropriately, so that the cover time is rescaled by the expected hitting time of a leaf.} with our main result \eqref{eq:MainResultIntro}, since the tail of hitting times of leaves also lacks a polynomial term.

In short, the subleading correction in all of these models encodes the very same physical principle of entropic repulsion.

\subsection{Open problems}
Our main result is a necessary step towards the identification of the weak limit of the (suitably rescaled)
$\varepsilon$-cover time. Based on the analogy with branching Brownian motion it is natural to expect 
the limiting law to be described by a mixture of Gumbel distributions, see e.g. \cite{ABKGenealogy}. Even  more challenging would be the full description of the process of covering; also in this case, 
the analogy with BBM suggests that regions which are missed the longest form a Poisson {\it cluster} process of random intensity \cite{ABKExtremalProcOFBBM, AidekonBerestickyBruneyShiBBMSeenFromItsTip}.

The extension of our main result to the {\it discrete} setting is also of interest. Here $C_{N}$ is the time it takes for (discrete or continuous time)
random walk to visit every vertex of the two dimensional discrete torus graph $\left(\mathbb{Z}/N\mathbb{Z}\right)^{2}$,
in the large $N$-limit. Dembo \emph{et. al.} \cite{DemboPeresRosen-BrownianMotionOnCompactManifolds} were able to {\it deduce} from \eqref{eq:DPRZResult} 
the corresponding law of large numbers for the discrete torus, namely:
\begin{equation}
\frac{C_{N}}{\frac{2}{\pi}N^{2}\log N}=\big(1+o\left(1\right)\big)\log N^{2}, \label{eq:DPRZDiscrete}
\end{equation}
in probability, for $N$ large. The deduction uses a strong ``Hungarian'' coupling of random walk and Brownian motion. As it turns out, this coupling is too coarse to deduce from our main result \prettyref{eq:MainResultIntro} the subleading order for $C_{N}$.
Nevertheless, the heuristic underlying the proof of \prettyref{eq:MainResultIntro}
can be applied to the discrete setting as well, and leads to the following conjecture (see also \prettyref{rem:EndRemark}).
\begin{conjecture*}
\begin{equation}
\frac{C_{N}}{\frac{2}{\pi}N^{2}\log N}=\log N^{2}-\big(1+o\left(1\right)\big)\log\log N,\label{eq:DiscreteCorrection}
\end{equation}
in probability, as $N\to\infty$.
\end{conjecture*}

\newpage
\tableofcontents

\section{Notation and definitions}

In this section we collect some important notations and definitions
used throughout the article.

We write $\mathbb{R}$ for the real numbers and define $\mathbb{R}_{+}=[0,\infty)$.
For any real $t\in\mathbb{R}_{+}$, we let $\lfloor t\rfloor$ denote
the integer part of $t$, and let $\lceil t\rceil$ denote the smallest
integer at least as large as $t$. For any set $A$ in a topological
space, $A^{\circ}$ denotes the interior of $A$ and $\bar{A}$ the
closure. The boundary $\partial A$ of the set $A$ is defined by
$\bar{A}\backslash A^{\circ}$. For any $a,b\in\mathbb{R}$ we denote
the minimum of $a$ and $b$ by $a\wedge b$. For any sequences $a_{t},b_{t},$
depending on some parameter $t$ the notation
\[
a_{t}\asymp b_{t},
\]
means that there exist constants $c$ and $C$ (not depending on $t$)
such that
\[
cb_{t}\le a_{t}\le Cb_{t}\mbox{ for all }t\ge0.
\]

Let $\left\{ 0,1,\ldots\right\} ^{\infty}$ be the space of integer
sequences and let $\left(T_{l}\right)_{l\ge0}$ be the canonical process
on this space. We let $\mathbb{G}_{n}$ denote the law on $\left\{ 0,1,\ldots\right\} ^{\infty}$
of a critical Galton Watson process with geometric offspring distribution
and initial population $n\in\left\{ 0,1\ldots\right\} $. This offspring
distribution has parameter $\frac{1}{2}$ and is supported on $\left\{ 0,1,\ldots\right\} $.
For real $t\ge0$, we take $\mathbb{G}_{t}$ to mean $\mathbb{G}_{\lfloor t\rfloor}$.

We write $\mathbb{T}=\left(\mathbb{R}/\mathbb{Z}\right)^{2}$ for
the two dimensional Euclidean torus. The map $\pi$ is the natural
projection of $\mathbb{R}^{2}$ onto $\mathbb{T}$, and $\pi^{-1}\left(x\right)$
the point in $[0,1)^{2}\subset\mathbb{R}^{2}$ which maps to $x$
under $\pi$. The Euclidean metric on $\mathbb{R}^{2}$ induces a
metric on $\text{\ensuremath{\mathbb{T}}}$ which we denote by $d(x,y),x,y\in\mathbb{T}$.
The closed ball of radius $r>0$ in $\mathbb{T}$ or $\mathbb{R}^{2}$
centered at $x$ is denoted by $B\left(x,r\right)$.

For any interval $I\subset\mathbb{R}$ and $D=\mathbb{R},\mathbb{T}$
or $\mathbb{R}^{2}$ we write $C\left(I,D\right)$ for the space of
continuous functions from $I$ to $D$ with the topology of uniform
convergence. We let $\mathcal{B}\left(I,D\right)$ denote the Borel
sigma algebra on this space. The canonical process on $C\left(I,\mathbb{R}\right)$
is denoted by $X_{t},t\ge0$, and the canonical processes on $C\left(I,\mathbb{T}\right)$
and $C\left(I,\mathbb{R}^{2}\right)$ are denoted by $W_{t},t\ge0$.
We denote by $\mathcal{F}_{t}=\mathcal{F}_{t}\left(I,D\right)$ the
natural filtration of the canonical process. The canonical shift on
$C\left(\mathbb{R}_{+},D\right)$ is denoted by $\theta_{t},t\ge0$.
To indicate ``chunks'' of the canonical process we write e.g. $W_{S+\cdot}$
for the path $t\to W_{S+t}$ of $W_{t}$ after time $S$, where $S$
is a random time. If $S_{1}$ and $S_{2}$ are two random times $W_{\left(S_{1}+\cdot\right)\wedge S_{2}}$
denotes the path $t\to W_{\left(S_{1}+t\right)\wedge S_{2}}$ of $W_{t}$
between times $S_{1}$ and $S_{2}$.

We let $P_{x}^{\mathbb{R}^{2}},x\in\mathbb{R}^{2},$ be the law on
$\left(C\left(\mathbb{R}_{+},\mathbb{R}^{2}\right),\mathcal{B}\left(\mathbb{R}_{+},\mathbb{R}^{2}\right)\right)$
which turns $W_{t},t\ge0,$ into a standard Brownian motion starting
at $x\in\mathbb{R}^{2}$. We let $P_{x},x\in\mathbb{T}$, be the law
on $\left(C\left(\mathbb{R}_{+},\mathbb{T}\right),\mathcal{B}\left(\mathbb{R}_{+},\mathbb{T}\right)\right)$
of $\left(\pi\left(W_{t}\right)\right)_{t\ge0}$ under $P_{\pi^{-1}\left(x\right)}^{\mathbb{R}^{2}}$;
that is, $P_{x}$ is the law of standard Brownian motion in $\mathbb{T}$
started at $x$. For any measure $\nu$ on $\mathbb{T}$ we let $P_{\nu}\left[\cdot\right]=\int_{\mathbb{T}}P_{x}\left[\cdot\right]\nu\left(dx\right)$;
if $\nu$ is a probability measure then $P_{\nu}$ is the law of Brownian
motion with starting point distributed according to $\nu$.

For any measurable set $A\subset\mathbb{T}$ or $\mathbb{R}^{2}$
we define the hitting time $H_{A}$ of $A$ by
\[
H_{A}=\inf\left\{ t\ge0:W_{t}\in A\right\} ,
\]
and for $A\subset\mathbb{R}$ we define $H_{A}$ similarly but with
$W_{t}$ replaced by $X_{t}$. For a singleton $a\in\mathbb{R}$ we
abbreviate $H_{a}=H_{\left\{ a\right\} }$. We write $T_{A}$ for
the exit time from $A\subset\mathbb{T}$ or $\mathbb{R}^{2}$, that
is 
\[
T_{A}=\inf\left\{ t\ge0:W_{t}\notin A\right\} .
\]

Note that any set $A\subset(0,1)^{2}\subset\mathbb{R}^{2}$ can be
identified with $\pi\left(A\right)\subset\mathbb{T}$, and that the
law of Brownian motion in $\pi\left(A\right)$ and $A$ coincide:
formally speaking, for any $a\in A$,
\begin{equation}
\mbox{the }P_{\pi\left(x\right)}-\mbox{law of }\pi^{-1}\left(W_{\cdot\wedge T_{\pi\left(A\right)}}\right)\mbox{ is the }P_{x}^{\mathbb{R}^{2}}-\mbox{law of }W_{\cdot\wedge T_{A}}.\label{eq:LawOfBMInTorusAndR2Same}
\end{equation}
In particular, when $R<\frac{1}{2}$ the ball $B\left(y,R\right)\subset\mathbb{T}$
can be identified with $B\left(\mathbb{\pi}^{-1}\left(y\right),R\right)$,
and the laws of Brownian motion in these two balls coincide. Brownian
motion in $\mathbb{R}^{2}$ is rotationally invariant, in the sense
that for any rotation $\rho:\mathbb{R}^{2}\to\mathbb{R}^{2}$ of the
space $\mathbb{R}^{2}$ around a point $y\in\mathbb{R}^{2}$ we have
that
\begin{equation}
\mbox{the }P_{\rho\left(z\right)}^{\mathbb{R}^{2}}-\mbox{law of }\left(\rho^{-1}\left(W_{t}\right)\right)_{t\ge0}\mbox{ is }P_{z}^{\mathbb{R}^{2}}.\label{eq:RotationinR2}
\end{equation}
The law of Brownian motion in a ball $B\left(y,R\right)\subset\mathbb{T}$
in the torus is also rotationally invariant if $R<\frac{1}{2}$, since
if $\rho:\mathbb{T}\to\mathbb{T}$ is a rotation of the ball $B\left(y,R\right)$
around $y$ (leaving $B\left(y,R\right)^{c}$ invariant) then \eqref{eq:LawOfBMInTorusAndR2Same}
and \eqref{eq:RotationinR2} imply that
\begin{equation}
\mbox{the }P_{\rho\left(z\right)}-\mbox{law of }\rho^{-1}\left(W_{\cdot\wedge T_{B\left(y,R\right)}}\right)\mbox{ is the }P_{\rho}-\mbox{law of }W_{\cdot\wedge T_{B\left(y,R\right)}}\label{eq:RotationalInvarianceInTorusBall}
\end{equation}

It is a standard fact that for any $0<r<R$ the exit distribution
of Brownian motion from the annulus $B\left(y,R\right)\backslash B\left(y,r\right)\subset\mathbb{R}^{2}$
satisfies
\[
P_{v}^{\mathbb{R}^{2}}\left[T_{B\left(y,R\right)}<H_{B\left(y,r\right)}\right]=\frac{\log\left(\left|v-y\right|/r\right)}{\log\left(R/r\right)}\mbox{ for all }v\in B\left(y,R\right)\backslash B\left(y,r\right)\subset\mathbb{R}^{2},
\]
(see Theorem 3.17 \cite{PeresMoertersBrownianMotion}). By \prettyref{eq:LawOfBMInTorusAndR2Same}
the same also holds in the torus: for any $0<r<R<\frac{1}{2}$
\begin{equation}
P_{v}\left[T_{B\left(y,R\right)}<H_{B\left(y,r\right)}\right]=\frac{\log\left(d\left(v,y\right)/r\right)}{\log\left(R/r\right)}\mbox{ for all }v\in B\left(y,R\right)\backslash B\left(y,r\right)\subset\mathbb{T}.\label{eq:AnnulusExitProb}
\end{equation}

In this article we will make heavy use of departure and return times
from concentric circles. For $0<r<R<\frac{1}{2}$ and $y\in\mathbb{T}$
the succesive return times to $B\left(y,r\right)$ are denoted by
$R_{n}\left(y,R,r\right),n\ge1,$ and the succesive departure times
from $B\left(y,R\right)$ are denoted by $D_{n}\left(y,R,r\right),n\ge1$.
Formally,
\begin{equation}
\begin{array}{lcl}
R_{1}\left(y,R,r\right) & = & H_{B\left(y,r\right)},\\
R_{n}\left(y,R,r\right) & = & H_{B\left(y,r\right)}\circ\theta_{D_{n-1}\left(y,R,r\right)}+D_{n-1}\left(y,R,r\right),n\ge1,\\
D_{n}\left(y,R,r\right) & = & T_{B\left(y,R\right)}\circ\theta_{R_{n}\left(y,R,r\right)}+R_{n}\left(y,R,r\right),n\ge1.
\end{array}\label{eq:DefOfExcursionTimes}
\end{equation}
Note that
\[
0\le R_{1}\left(y,R,r\right)<D_{1}\left(y,R,r\right)<R_{2}\left(y,R,r\right)<D_{2}\left(y,R,r\right)<\ldots.
\]
We will often refer to $W_{\left(R_{n}\left(y,R,r\right)+\cdot\right)\wedge D_{n}\left(y,R,r\right)}$
as the $n$-th excursion or $n-$th traversal from $\partial B\left(y,r\right)$
to $\partial B\left(y,R\right)$; these (and other excursions) will
play an important role in the proofs.

Lastly a note on constants. The letter $c$ represents a constant
that is positive and does not depend on any other parametes. It may
represent different constants in different formulas, and even within
the same formula. Dependence on e.g. a parameters $s$ is denoted
by $c\left(s\right)$.

\section{\label{sec:Reduction}Statement of main theorem, construction of
traversal processes and first reduction}

In this section we formally state the main theorem. We also start
the proof by constructing the traversal processes around each point
$y\in\mathbb{T}$ which where mentioned in the introduction, and deriving
their basic properties. We reduce the proof of the main theorem to
three main propositions, which will be proven in the subsequent sections.
The first two of these deal with the traversal processes. More precisely,
\prettyref{prop:UpperBound} essentially proves the upper bound of
\prettyref{eq:MainResultIntro}, and \prettyref{prop:LowerBound}
essentially proves the lower bound. However, they do this in terms
of excursions, i.e. they determine how many excursions around each
point are needed to cover the torus. The third of the main proposition,
\prettyref{prop:Concentration}, relates this number of excursions
to the actual time of Brownian motion, thus allowing us to deduce
the main result \prettyref{eq:MainResultIntro} from propositions
\ref{prop:UpperBound} and \ref{prop:LowerBound}.

To formally state our main result we define the $\varepsilon-$cover
time $C_{\varepsilon}$ as 
\begin{equation}
C_{\varepsilon}=\sup_{y\in\mathbb{T}}H_{B\left(y,\varepsilon\right)},\label{eq:FormalDefOfCoverTime}
\end{equation}
and (deviating slightly from the formulation in the introduction,
cf. \eqref{eq:mwithcorrection}) let 
\begin{equation}
m\left(\varepsilon,s\right)=\frac{1}{\pi}\log\varepsilon^{-1}\left(2\log\varepsilon^{-1}-\left(1-s\right)\log\log\varepsilon^{-1}\right).\label{eq:DefinitionOfm}
\end{equation}
The formal statement of \eqref{eq:MainResultIntro} is the following.
\begin{thm}
\label{thm:MainResult}For all $s>0$ and all $x\in\mathbb{T}$
\begin{equation}
\lim_{\varepsilon\to0}P_{x}\left[m\left(\varepsilon,-s\right)\le C_{\varepsilon}\le m\left(\varepsilon,s\right)\right]=1.\label{eq:MainResult}
\end{equation}

\end{thm}
The proof of \prettyref{thm:MainResult} (or rather, its reduction
to propositions \ref{prop:UpperBound}-\ref{prop:Concentration})
will be given at the end of this section.

We now construct the traversal processes that are the cornerstone
of \prettyref{thm:MainResult}'s proof. The construction will depend
on an integer parameter $L\ge1$, which represents the number of scales
that we consider. Let
\begin{equation}
r_{l}=r_{l}\left(L\right)=e^{-\frac{3}{4}\log\log L-l},l=0,1,\ldots,\label{eq:DefOfRadii}
\end{equation}
denote a sequence of radii, corresponding to the scales in the multiscale
analysis described in the introduction. Note that 
\[
r_{0}\to0,\mbox{ as }L\to\infty,
\]
which is important for the proof of the lower bound \prettyref{prop:LowerBound}
(since it means that an overwhelming majority of all pairs of traversal
processes depend on disjoint regions and will therefore be completely
independent, which helps in bounding the second moment of the truncated
counting random variable, see \prettyref{eq:FirstTermInDecomp}; if
we were only proving the upper bound of \prettyref{thm:MainResult}
we could have removed $\frac{3}{4}\log\log L$ from \prettyref{eq:DefOfRadii}).
We will show \prettyref{thm:MainResult} along the sequence $\varepsilon=r_{L}$
as $L\to\infty$. We will see that this easily implies \prettyref{thm:MainResult}
in full generality.

The proof is based on tracking Brownian motion as it moves between
the scales given by the $r_{l}$. For this it is useful to note that
since the $r_{l}$ shrink geometrically we have from \prettyref{eq:AnnulusExitProb}
that for any $y\in\mathbb{T}$ and $0<l<L$ 
\begin{equation}
P_{v}\left[H_{B\left(y,r_{l+1}\right)}<T_{B\left(y,r_{l-1}\right)}\right]=\frac{1}{2},\mbox{ for }v\in\partial B\left(y,r_{l}\right).\label{eq:BallExitProbForRadii}
\end{equation}
\begin{figure}
\includegraphics[scale=0.8]{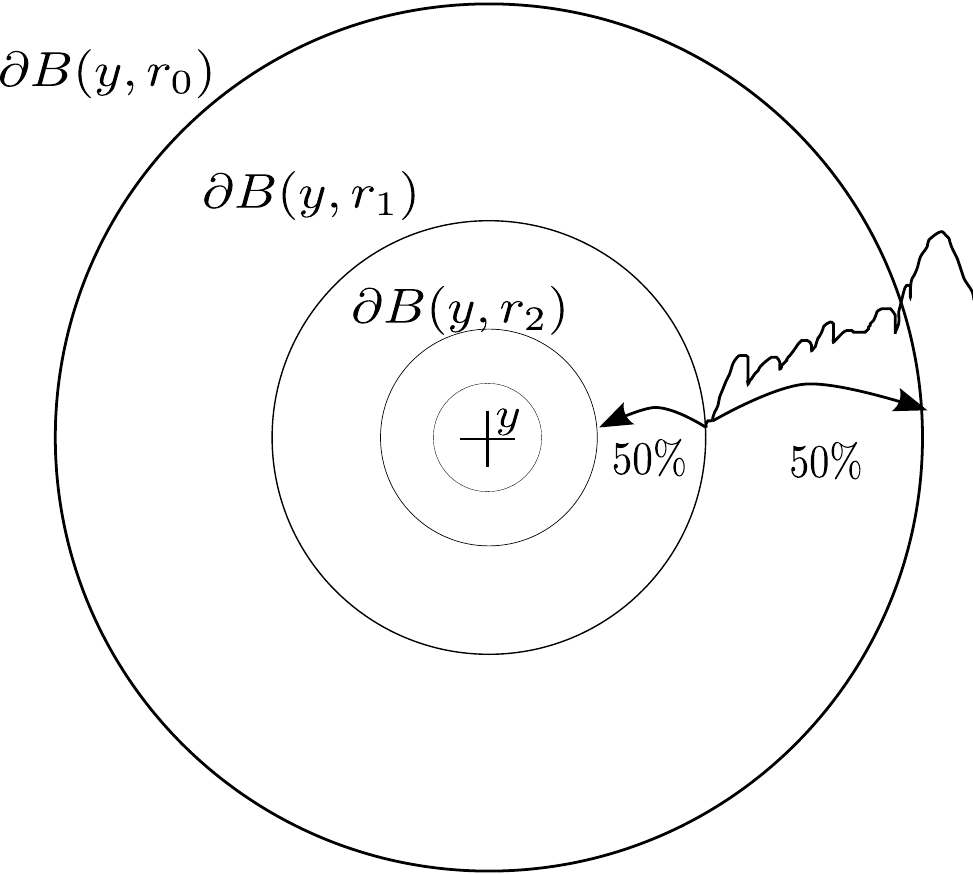}\caption{\label{fig:UpAScaleDownAScale}There is a $50\%$ probability that
Brownian motion goes ``up a scale'' and $50\%$ probability that
it goes ``down a scale''}
\end{figure}
That is Brownian motion essentially speaking flips an unbiased coin
to decide wether to move to a higher scale or a lower scale (see \prettyref{fig:UpAScaleDownAScale}).
More generally
\begin{equation}
P_{v}\left[H_{B\left(y,r_{l_{3}}\right)}<T_{B\left(y,r_{l_{1}}\right)}\right]=\frac{l_{2}-l_{1}}{l_{3}-l_{1}},\mbox{ for }v\in\partial B\left(y,r_{l_{2}}\right),l_{1}<l_{2}<l_{3}.\label{eq:HitLProbStartingFromk}
\end{equation}

It will be useful to introduce the following abbreviations for the
time $R_{n}\left(y,r_{l},r_{l+1}\right)$ that the $n-$th traversal
from scale $l$ and $l+1$ is completed and for the time $D_{n}\left(y,r_{l},r_{l+1}\right)$
that the $n-$th traversal from scale $l+1$ to $l$ is completed
(recall \prettyref{eq:DefOfExcursionTimes})
\begin{equation}
\begin{array}{l}
R_{n}^{y,l}=R_{n}^{y,l}\left(L\right)=R_{n}\left(y,r_{l},r_{l+1}\right)\mbox{ and }\\
D_{n}^{y,l}=D_{n}^{y,l}\left(L\right)=D_{n}\left(y,r_{l},r_{l+1}\right)\mbox{ for }n\ge0,l\ge0,y\in\mathbb{T}.
\end{array}\label{eq:ExcursionTimeAbreviation}
\end{equation}
If $t\in\mathbb{R}_{+}$ we let $R_{t}^{y,l}=R_{\lfloor t\rfloor}^{y,l}$
and $D_{t}^{y,l}=D_{\lfloor t\rfloor}^{y,l}$. For $y\in\mathbb{T}$
and $t\in\mathbb{R}_{+}$ we can now formally define the process of
traversals $T_{l}^{y,t},l\ge0,$ by
\begin{equation}
\begin{array}{ccccl}
T_{l}^{y,t} & = & T_{l}^{y,t}\left(L\right) & = & \sup\left\{ n\ge0:R_{n}^{y,l}\le D_{t}^{y,0}\right\} ,l\ge0,\end{array}\label{eq:TraversalsDef}
\end{equation}
where we understand $R_{0}^{y,l}=0$.
\begin{figure}
\begin{minipage}[t][0.25\paperheight][b]{0.65\columnwidth}%
\includegraphics[scale=0.7]{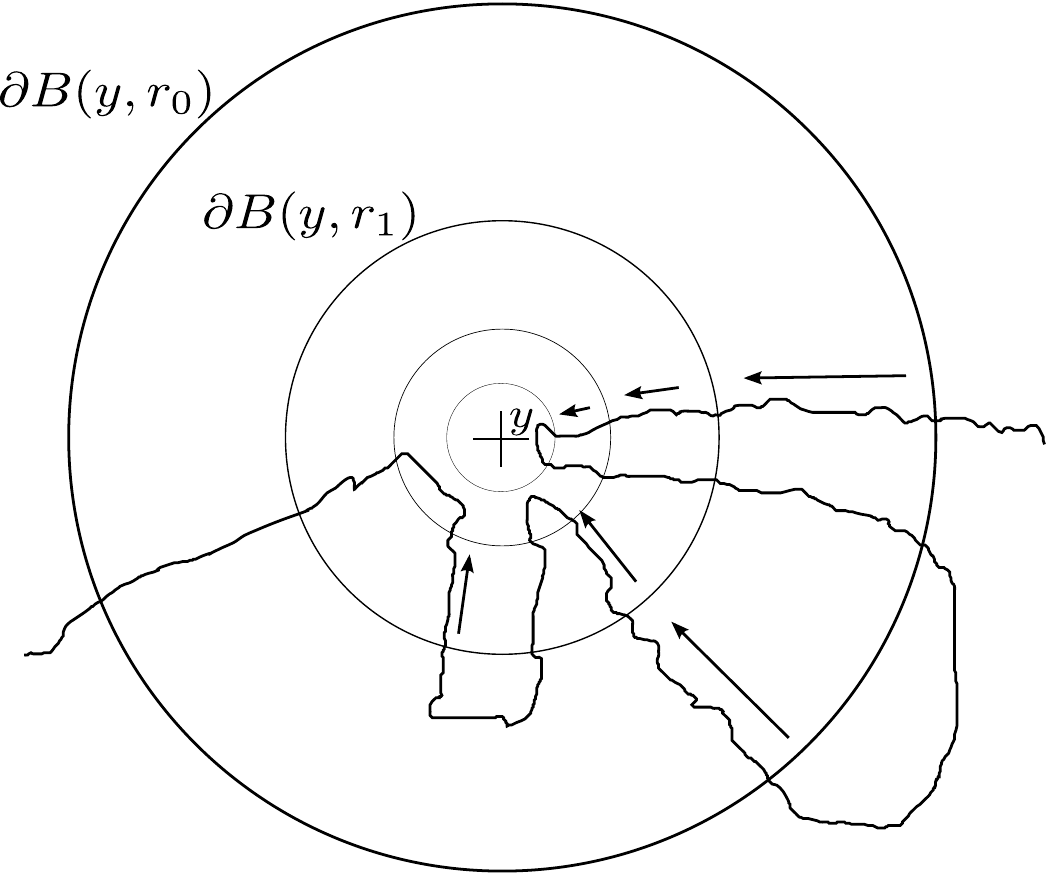}%
\end{minipage}%
\begin{minipage}[t][0.25\paperheight][c]{0.1\columnwidth}%
\[
\begin{array}{c}
T_{0}^{y,2}=2,\\
T_{1}^{y,2}=3,\\
T_{2}^{y,2}=1.
\end{array}
\]
\end{minipage}

\framebox{\begin{minipage}[t]{1\columnwidth}%
This illustration shows Brownian motion moving ``in the first three
scales'' around a point $y$. The arrows indicate completed traversals.
From this picture we can read off $T_{l}^{y,2},l=0,1,2.$ The values
are shown to the right.%
\end{minipage}}\caption{\label{fig:Illustration-of-process}Illustration of traversal process}
\end{figure}
Note that $T_{l}^{y,t}$ is the number of traversals from $B\left(y,r_{l}\right)^{c}$
to $B\left(y,r_{l+1}\right)$ made by Brownian motion during the first
$\lfloor t\rfloor$ excursions from $\partial B\left(y,r_{1}\right)$
to $\partial B\left(y,r_{0}\right)$ (see \prettyref{fig:Illustration-of-process}
for an illustration). This means that the process $T_{l}^{y,t}$ contains
information about whether $B\left(y,r_{L}\right)$ has been hit by
time $D_{t}^{y,0}$ or not, since (see \eqref{eq:ExcursionTimeAbreviation}
and \eqref{eq:TraversalsDef}, and cf. \eqref{eq:iff})
\begin{equation}
T_{L-1}^{y,t}=0\iff H_{B\left(y,r_{L}\right)}>D_{t}^{y,0},\mbox{ for }y\in\mathbb{T}.\label{eq:ExtinctionImpliesNotHit}
\end{equation}
This gives a link between the $r_{L}-$cover time and the collection
of process $T_{l}^{y,t}$. The traversal processes have the following
simple characterisation.
\begin{lem}
\label{lem:BranchingProc}For all $y\in\mathbb{T}$ and $v\notin B\left(y,r_{1}\right)^{\circ}$
the $P_{v}-$law of $\left(T_{l}^{y,t}\right)_{l\ge0}$ is the $\mathbb{G}_{t}-$law
of $\left(T_{l}\right)_{l\ge0}$, i.e. it is a critical Galton-Watson
process with geometric offspring distribution.\end{lem}
\begin{proof}
Fix $y\in\mathbb{T}$. Consider the indicator functions 
\[
I_{l,n}^{y}=\left\{ T_{B\left(y,r_{l-1}\right)}\circ\theta_{D_{n}^{y,l}}<H_{B\left(y,r_{l+1}\right)}\circ\theta_{D_{n}^{y,l}}\right\} ,n\ge0,l\ge0,
\]
which are one if Brownian motion next visits $\partial B\left(y,r_{l-1}\right)$
after making a traversal $l+1\to l$ and zero if it next visits $\partial B\left(y,r_{l+1}\right)$.
By \prettyref{eq:BallExitProbForRadii} and the strong Markov property
they are unbiased i.i.d. Bernoulli ``coin flips'' by (see also \prettyref{fig:UpAScaleDownAScale}).
We can reconstruct the traversal process from the $I_{l,n}^{y}$ recursively
by setting $T_{0}^{y,t}=\lfloor t\rfloor$ and
\[
T_{l+1}^{y,t}=\begin{array}{c}
\mbox{the number of zeros among }I_{l+1,m}^{y},m\le n\\
\mbox{ where }n\mbox{ is such that }I_{l+1,1}^{y}+\ldots+I_{l+1,n}^{y}=T_{l}^{y,t}
\end{array},\mbox{for }l\ge0.
\]
Thus $T_{l+1}^{y,t}$ has a negative binomial distribution conditioned
on $T_{l^{'}}^{y,t},l^{'}\le l,$ and since this distribution is also
the sum of $T_{l}^{y,t}$ independent geometrics with support $\left\{ 0,1,\ldots\right\} $
and mean $1$ the claim follows.\end{proof}
\begin{rem}
This can be seen as a discrete Ray-Knight theorem for the directed
edge local times of a simple random walk $\left(Z_{n}\right)_{n\ge0}$
on $\left\{ 0,1,\ldots\right\} $. By \prettyref{eq:BallExitProbForRadii}
the process $Z_{n}$ can be constructed by letting it be the index
of the successive scales around $y$ that $W_{t}$ visits, i.e. $Z_{0}=0,Z_{1}=1$
and $Z_{2}=0$ or $2$ depending on if $W_{t}$ visits $\partial B\left(y,r_{0}\right)$
or $\partial B\left(y,r_{2}\right)$ first after $R_{1}^{y,0}$, and
so on. With this construction $T_{l}^{y,t}$ is $Z_{n}$'s edge local
time at $l\to l+1$ after $t$ of $Z_{n}$'s excursions from $0$.
Also note that \eqref{eq:HitLProbStartingFromk} can be seen as the
standard result about the exit distribution of the simple random walk
$Z_{n}$ from the interval $\left[l_{1},l_{3}\right]$. 
\end{rem}
As mentioned in the introduction, there are several instances when
we will use ``packings'' of balls in the torus at different scales.
The $l-$th scale packing will consist of balls centered at points
of the following grid of ``spacing'' $r_{l}/1000$:
\begin{equation}
F_{l}=\left\{ \left(i\frac{r_{l}}{1000},j\frac{r_{l}}{1000}\right):i,j\in\left\{ 0,1,\ldots,\lfloor\frac{1000}{r_{l}}\rfloor\right\} \right\} \subset\mathbb{T},l\ge0.\label{eq:GridDefinition}
\end{equation}
It will turn out that $C_{r_{L}}$ is close to the maximum of the
hitting times of balls centered in $F_{L}$. We record for future
use that (see \eqref{eq:DefOfRadii}) 
\begin{equation}
\left|F_{l}\right|\asymp cr_{l}^{-2}=c\left(\log L\right)^{3/2}e^{2l},l\ge0.\label{eq:SizeOFfl}
\end{equation}
Note that when comparing our method to the study of branching Brownian
motion (or rather a version of BBM with branching at integer times;
equivalently Gaussian Free Field on a tree) $F_{l}$ corresponds to
the vertices at distance $l$ from the root. With this point of view
we see that essentially speaking we have a ``forest'' of $c\left(\log L\right)^{3/2}$
``pseudo-tree'' with branching factor $e^{2}$.

We now state a second simple but crucial property of the traversal
process, which essentially gives bounds on the probability the it
``dies out'' by generation $L-1$ (using the Galton-Watson terminology),
or equivalently that the ball $B\left(y,r_{L}\right)$ does not get
hit. For this we consider a number of traversals
\begin{equation}
t_{s}=t_{s}\left(L\right)=L\left(2L-\left(1-s\right)\log L\right)\mbox{\,\ for }s\in\mathbb{R},\label{eq:defofts}
\end{equation}
from scale $0$ to scale $1$, which we will see is roughly the number
of traversals that take place up to time $m\left(r_{L},s\right)$
(cf. \eqref{eq:introts}; note the slight difference).  The bound
is the following.
\begin{lem}
\label{lem:NotHitByrL}For all $y\in\mathbb{T}$, $x\notin B\left(y,r_{1}\right)^{\circ}$
, $s\in\mathbb{R}$ and $L\ge c\left(s\right)$
\begin{equation}
P_{x}\left[T_{L-1}^{y,t_{s}}=0\right]\asymp e^{-2L}L^{1-s}.\label{eq:NotHitByrL}
\end{equation}
\end{lem}
\begin{proof}
The event $\left\{ T_{L-1}^{y,t_{s}}=0\right\} $ is the event that
$B\left(y,r_{L}\right)$ is not hit in $\lfloor t_{s}\rfloor$ excursions
from $\partial B\left(y,r_{1}\right)$ to $\partial B\left(y,r_{0}\right)\cup\partial B\left(y,r_{L}\right)$.
By \prettyref{eq:HitLProbStartingFromk} one such excursion hits $B\left(y,r_{L}\right)$
with probability $1/L$, regardless of where in $\partial B\left(y,r_{1}\right)$
it starts. Thus using the strong Markov property 
\[
P_{x}\left[T_{L-1}^{y,t_{s}}=0\right]=\left(1-\frac{1}{L}\right)^{\lfloor t_{s}\rfloor}=e^{-\frac{t_{s}}{L}}\left(1+O\left(\frac{t_{s}}{L^{2}}+\frac{1}{L}\right)\right).
\]
Thus \prettyref{eq:NotHitByrL} follows since (recall \eqref{eq:defofts})
\begin{equation}
\frac{t_{s}}{L}=2L-\left(1-s\right)\log L.\label{eq:TsDivdedByL}
\end{equation}

\end{proof}
Using this bound and \prettyref{eq:SizeOFfl} one can roughly speaking
compute the expectation of the simple untruncated random variable
counting balls of radius $r_{L}$ centered in $F_{L}$ that are not
hit in $t_{s}$ excursions:
\begin{equation}
"E_{x}\left[\sum_{y\in F_{L}}1_{\left\{ T_{L-1}^{y,t_{s}}=0\right\} }\right]\asymp\left(\log L\right)^{3/2}L^{1-s}",\label{eq:Sec3CountingRV}
\end{equation}
cf. \prettyref{eq:ZDefIntro} and \eqref{eq:Expofzms} (here $L$
corresponds to $\log\varepsilon^{-1}$ and the $\log L$ factor is
an artifact of defining the $r_{l}$ so that $r_{0}\downarrow0$).
This essentially proves that for $s>1$ there are no balls with center
in $y\in F_{L}$ that avoid being hit in $t_{s}$ excursions from
$\partial B\left(y,r_{1}\right)$ to $\partial B\left(y,r_{0}\right)$.
For $s\le1$ we see that the expected number of balls that manage
this tends to infinity. This does not correctly capture the actual
number, as shown by our first main proposition which we now state.
It essentially says that also for $0<s\le1$ there will be no balls
which avoid being hit in $t_{s}$ of ``its'' excursions from scale
$1$ to scale $0$.
\begin{prop}
\label{prop:UpperBound}($x\in\mathbb{T}$)
\begin{equation}
\lim_{L\to\infty}P_{x}\left[T_{L-1}^{y,t_{s}}=0\mbox{ for some }y\in F_{L}\right]=0,\mbox{ for all }s>0.\label{eq:UppberBound}
\end{equation}

\end{prop}
This roughly speaking gives the upper bound of \prettyref{thm:MainResult},
``in terms of excursions at the highest scale''. \prettyref{prop:UpperBound}
will be proven in \prettyref{sec:UpperBound} using a truncated first
moment bound. We now state \prettyref{prop:LowerBound}, which essentially
says that for $s>0$ there is (with high probability) some $y\in F_{L}$
which is not hit in $t_{-s}$ of ``its'' excursions. This roughly
speaking gives the lower bound of \prettyref{thm:MainResult}, ``in
terms of excursions at the highest scale''.
\begin{prop}
\label{prop:LowerBound}($x\in\mathbb{T}$)
\begin{equation}
\lim_{L\to\infty}P_{x}\left[T_{L-1}^{y,t_{-s}}=0\mbox{ for some }y\in F_{L}\right]=1\mbox{, for all }s>0.\label{eq:LowerBoundInRedSec}
\end{equation}

\end{prop}
\prettyref{prop:LowerBound} will be proved in \prettyref{sec:Lowerbound}
using a truncated second moment method. Finally we state the concentration
result \prettyref{prop:Concentration} which essentially speaking
says that at time $\frac{1}{\pi}t_{s}$ there will have been roughly
$t_{s}$ excursions from scale $0$ to scale $1$ for all $y\in F_{L}$.
This will allow us to deduce the main result \prettyref{eq:MainResult}
from the above two propositions.
\begin{prop}
\label{prop:Concentration}($x\in\mathbb{T}$) For all $s>0$
\begin{eqnarray}
\lim_{L\to\infty}P_{x}\left[D_{t_{s}}^{y,0}>\frac{1}{\pi}t_{2s}\mbox{ for some }y\in F_{L}\right] & = & 0\mbox{ and},\label{eq:UpperBoundConcentration}\\
\lim_{L\to\infty}P_{x}\left[D_{t_{-s}}^{y,0}<\frac{1}{\pi}t_{-\frac{1}{2}s}\mbox{ for some }y\in F_{L}\right] & = & 0.\label{eq:LowerBoundConcentration}
\end{eqnarray}

\end{prop}
\prettyref{prop:Concentration} will be proven in \prettyref{sec:Concentration}
using a packing argument and a large deviation bound for $D_{t}^{y,0}$.
We now derive \prettyref{thm:MainResult} from propositions \ref{prop:UpperBound}-\ref{prop:Concentration}. 
\begin{proof}[Proof of \prettyref{thm:MainResult}]
We first reduce the proof of the convergence in \prettyref{eq:MainResult}
to convergence along the subsequence $\varepsilon=r_{L}$. Assume
we have shown that for all $s>0$
\begin{equation}
\lim_{L\to0}P_{x}\left[C_{r_{L}}>m\left(r_{L},s\right)\right]=0\mbox{ and }\label{eq:RedUpBound}
\end{equation}
\begin{equation}
\lim_{L\to0}P_{x}\left[C_{r_{L}}<m\left(r_{L},-s\right)\right]=0.\label{eq:RedLowerBound}
\end{equation}
Then for $\varepsilon>0$ we may set
\begin{equation}
L_{\pm}=\log\varepsilon^{-1}-\frac{3}{4}\log\log\log\varepsilon^{-1}\pm1000,\label{eq:DefOfLpm}
\end{equation}
so that $c\varepsilon\le r_{L_{+}}\le\varepsilon\le r_{L_{-}}\le c\varepsilon$
(see \prettyref{eq:DefOfRadii}). This in turn gives that $C_{r_{L_{-}}}\le C_{\varepsilon}\le C_{r_{L_{+}}}$
and $m\left(r_{L_{\pm}},s\right)=m\left(\varepsilon,s\right)\left(1+O\left(1/\log\varepsilon^{-1}\right)\right)$
(see \prettyref{eq:DefinitionOfm}), so that $m\left(\varepsilon,-s\right)\le m\left(r_{L_{-}},-s/2\right)$
and $m\left(r_{L_{+}},s/2\right)\le m\left(\varepsilon,s\right)$
for small enough $\varepsilon$. Therefore \prettyref{eq:MainResult}
follows from \prettyref{eq:RedUpBound} with $L_{+}$ in place of
$L$ and \prettyref{eq:RedLowerBound} with $L_{-}$ in place of $L$,
and $s/2$ in place of $s$ in both instances.

We thus turn our attention to \prettyref{eq:RedUpBound} and \prettyref{eq:RedLowerBound}.
For \prettyref{eq:RedUpBound} we first reduce the a bound where the
supremum in $C_{r_{L}}$ (recall \prettyref{eq:FormalDefOfCoverTime})
is taken over $F_{L}$ and not all of $\mathbb{T}$. We have that
\begin{equation}
\left\{ C_{r_{L}}>m\left(r_{L},s\right)\right\} =\left\{ H_{B\left(y,r_{L}\right)}>m\left(r_{L},s\right)\mbox{ for some }y\in\mathbb{T}\right\} .\label{eq:CTtoHT}
\end{equation}
Since $r_{L+1}/1000+r_{L+1}\le2r_{L+1}\le r_{L}$ each ball of radius
$r_{L}$ in $\mathbb{T}$ contains a ball of radius $r_{L+1}$ centered
at some $z\in F_{L+1}$ (recall \prettyref{eq:GridDefinition}). Also
similarly to above $m\left(r_{L},s\right)\ge m\left(r_{L+1},s/2\right)$
for $L\ge c$. Therefore the probability in \prettyref{eq:RedUpBound}
is bounded above by 
\[
P_{x}\left[H_{B\left(z,r_{L+1}\right)}>m\left(r_{L+1},s/2\right)\mbox{ for some }z\in F_{L+1}\right],
\]
so that to show \prettyref{eq:RedUpBound} it suffices to prove that
for all $s>0$
\begin{equation}
\lim_{L\to\infty}P_{x}\left[H_{B\left(z,r_{L}\right)}>m\left(r_{L},s\right)\mbox{ for some }z\in F_{L}\right]=0.\label{eq:rlFiniteSet}
\end{equation}

We have that (see \eqref{eq:DefinitionOfm} and \eqref{eq:defofts})
\begin{equation}
m\left(r_{L},s\right)=\frac{1}{\pi}t_{s}\left(1-O\left(\log\log L/L\right)\right)\ge\frac{1}{\pi}t_{s/2},\label{eq:msandthalfsrelation}
\end{equation}
for $L$ large enough. Thus the probability in \prettyref{eq:rlFiniteSet}
is bounded above by
\[
P_{x}\left[H_{B\left(z,r_{L}\right)}>\frac{1}{\pi}t_{s/2}\mbox{ for some }z\in F_{L}\right],
\]
which in turn is bounded above by
\begin{equation}
P_{x}\left[H_{B\left(z,r_{L}\right)}>D_{t_{s/4}}^{z,0}\mbox{ for some }z\in F_{L}\right]+P_{x}\left[D_{t_{s/4}}^{z,0}>\frac{1}{\pi}t_{s/2}\mbox{ for some }z\in F_{L}\right].\label{eq:MainProofUnionBound}
\end{equation}
Now since 
\[
\left\{ H_{B\left(z,r_{L}\right)}>D_{t_{s/4}}^{z,0}\mbox{ for some }z\in F_{L}\right\} \overset{\eqref{eq:ExtinctionImpliesNotHit}}{=}\left\{ T_{L-1}^{z,t_{s/4}}=0\mbox{\,\ for some }z\in F_{L}\right\} ,
\]
the two probabilities in \eqref{eq:MainProofUnionBound} tend to zero
when $L\uparrow\infty$, by \prettyref{prop:UpperBound} and \prettyref{eq:UpperBoundConcentration}.
This proves \prettyref{eq:rlFiniteSet}, and therefore also \prettyref{eq:RedUpBound}
and the upper bound of \prettyref{eq:MainResult}.

We now turn our attention to \prettyref{eq:RedLowerBound}. We have
\begin{equation}
\left\{ C_{r_{L}}<m\left(r_{L},s\right)\right\} \overset{\eqref{eq:FormalDefOfCoverTime}}{=}\left\{ H_{B\left(y,r_{L}\right)}<m\left(r_{L},s\right)\mbox{ for all }y\in\mathbb{T}\right\} .\label{eq:LBCTtoHT}
\end{equation}
For $L$ large enough we have that (cf \prettyref{eq:msandthalfsrelation})
 $m\left(r_{L},-s\right)\le\frac{1}{\pi}t_{-\frac{1}{2}s}$. Thus
for such $L$, \prettyref{eq:LBCTtoHT} is included in 
\[
\left\{ H_{B\left(y,r_{L}\right)}<\frac{1}{\pi}t_{-\frac{1}{2}s}\mbox{ for all }y\in F_{L}\right\} ,
\]
which in turn is included in
\[
\left\{ H_{B\left(y,r_{L}\right)}<D_{t_{-s}}^{y,0}\mbox{ for all }y\in F_{L}\right\} \cup\left\{ D_{t_{-s}}^{y,0}<\frac{1}{\pi}t_{-\frac{1}{2}s}\mbox{ for some }y\in F_{L}\right\} .
\]
But
\[
\left\{ H_{B\left(y,r_{L}\right)}<D_{t_{-s}}^{y,0}\mbox{ for all }y\in F_{L}\right\} \overset{\eqref{eq:ExtinctionImpliesNotHit}}{=}\left\{ T_{L-1}^{y,t_{-s}}>0\mbox{\,\ for all }y\in F_{L}\right\} ,
\]
so that we obtain for $L$ large enough
\[
\begin{array}{ccl}
P_{x}\left[C_{r_{L}}<m\left(r_{L},s\right)\right] & \le & P_{x}\left[T_{L-1}^{y,t_{-s}}>0\mbox{\,\ for all }y\in F_{L}\right]\\
 &  & +P_{x}\left[D_{t_{-s}}^{y,0}<\frac{1}{\pi}t_{-\frac{1}{2}s}\mbox{ for some }y\in F_{L}\right].
\end{array}
\]
Taking the limit $L\uparrow\infty$ we see that \prettyref{eq:RedLowerBound}
follows from \prettyref{prop:LowerBound} and \prettyref{eq:LowerBoundConcentration},
so the lower bound of \eqref{eq:MainResult} follows.
\end{proof}
In this section we have reduced the proof of the main result \prettyref{thm:MainResult}
to the three main propositions \ref{prop:UpperBound}-\ref{prop:Concentration}.
The rest of the article is devoted to their derivation.

\section{\label{sec:UpperBound}Upper bound on cover time in terms of excursions}

In this section we prove \prettyref{prop:UpperBound}, which is the
first of the three main propositions used to prove the main result
\prettyref{thm:MainResult}, and which gives the upper bound of that
result ``in terms of excursions from scale $1$ to scale $0$''.
More precisely, recall the claim \eqref{eq:UppberBound} of \prettyref{prop:UpperBound}
that 
\[
\lim_{L\to\infty}P_{x}\left[T_{L-1}^{y,t_{s}}=0\mbox{ for some }y\in F_{L}\right]=0\mbox{ for all }s>0,\tag{\ref{eq:UppberBound}'}
\]
(recall also the definitions from \eqref{eq:TraversalsDef}, \eqref{eq:GridDefinition}
and \eqref{eq:defofts}).

For technical reasons, we will consider rather than $T_{l}^{y,t_{s}},l\ge0,$
a modified traversal process $\tilde{T}_{l}^{y,t},l\ge0,$ which counts
only traversals that take place after leaving $B\left(y,r_{0}\right)$
for the first time (if the starting point of the Brownian motion lies
inside $B\left(y,r_{1}\right)$ then this modified traversal process
and the original traversal process may not coincide). Formally we
let (cf. \eqref{eq:TraversalsDef})
\begin{equation}
\tilde{T}_{l}^{y,t_{s}}=\sup\left\{ n\ge0:R_{n}^{y,l}\circ\theta_{T_{B\left(y,r_{0}\right)}}\le D_{t}^{y,0}\circ\theta_{T_{B\left(y,r_{0}\right)}}\right\} ,y\in F_{L},l\ge0,t\ge0.\label{eq:ModifiedTraversals}
\end{equation}
We will prove that
\begin{equation}
\lim_{L\to\infty}P_{x}\left[\tilde{T}_{L-1}^{y,t_{s}}=0\mbox{ for some }y\in F_{L}\right]=0\mbox{ for all }s>0,\label{eq:UBInUBSectionIgnoring}
\end{equation}
which is a slightly stronger statement than \prettyref{eq:UppberBound},
because $\tilde{T}_{L-1}^{y,t_{s}}\le T_{L-1}^{y,t_{s}}$ almost surely
for all $y$. The modifed traversal process is used because \prettyref{lem:BranchingProc}
and the strong Markov property imply that
\begin{equation}
\mbox{the }P_{x}-\mbox{law of }\left(\tilde{T}_{l}^{y,t_{s}}\right)_{l\ge0}\mbox{ is }\mathbb{G}_{t_{s}}\mbox{ for all }x,y\in\mathbb{T},\label{eq:LawOfTTilde}
\end{equation}
(this is not exactly true for $T_{l}^{y,t_{s}},l\ge0,$ such that
$x\in B\left(y,r_{1}\right)^{\circ}$).

A previously discussed, a natural approach to proving \prettyref{eq:UBInUBSectionIgnoring}
is the simple first moment upper bound using the counting random variable
$\sum_{y\in F_{L}}1_{\left\{ \tilde{T}_{L-1}^{y,t_{s}}=0\right\} }$,
but this however would yield \prettyref{eq:UBInUBSectionIgnoring}
only for $s<-1$ (cf. \prettyref{eq:Sec3CountingRV}). We therefore
introduce a truncation which is given in terms of the barrier
\begin{equation}
\alpha\left(l\right)=\alpha\left(l,L,s\right)=\beta\left(l\right)-\left(\log L\right)^{2},\label{eq:AlphaBarrierDef}
\end{equation}
where $\beta\left(l\right)$ given by
\begin{equation}
\beta\left(l\right)=\beta\left(l,L,s\right)=\left(1-\frac{l}{L}\right)\sqrt{t_{s}}\mbox{\,\ for }l\in\left[0,L\right].\label{eq:BetaDef}
\end{equation}
The line $\beta\left(l\right)$ turns out to essentially be the mean
of the process $\sqrt{T_{l}^{y,t_{s}}}$ conditioned on $T_{L-1}^{y,t_{s}}=0$.
See \prettyref{fig:Illustration-of-barriers}. We consider the truncated
counting random variable which imposes an additional ``barrier condition''
\begin{equation}
\sum_{y\in F_{L}}1_{\left\{ \tilde{T}_{L-1}^{y,t_{s}}=0\right\} \cap\left\{ \sqrt{\tilde{T}_{l}^{y,t_{s}}}\ge\alpha\left(l\right)\mbox{\,\ for }l=0,\ldots,L\right\} }.\label{eq:UBSecTruncatedZ}
\end{equation}
Our claim \prettyref{eq:UBInUBSectionIgnoring} will follow from two
main propositions: \prettyref{prop:UpperBoundOneProfile} and \prettyref{prop:SmartMarkov}
below. \prettyref{prop:UpperBoundOneProfile} will show that the expectation
of \prettyref{eq:UBSecTruncatedZ} goes to zero for all $s>0$. \prettyref{prop:SmartMarkov}
will show that with high probability there are no $y\in F_{L}$ such
that $\sqrt{\tilde{T}_{l}^{y,t_{s}}}$ violates the barrier condition.

The key to bounding the expectation in \prettyref{eq:UBSecTruncatedZ}
is bounding the conditional probability 
\begin{equation}
P_{x}\left[\sqrt{\tilde{T}_{l}^{y,t_{s}}}\ge\alpha\left(l\right)\mbox{\,\ for }l=0,\ldots,L-1|\tilde{T}_{L-1}^{t_{s}}=0\right].\label{eq:UBCondProb}
\end{equation}
By \prettyref{eq:LawOfTTilde} this amounts to an estimate purely
in terms of the Galton-Watson law $\mathbb{G}_{t_{s}}$. For this
law, \prettyref{lem:GWBarrierAlpha} gives a bound on the probability
of the form $\left(\log L\right)^{4}/L$ (it is this extra factor
that gives rise to the subleading correction). \prettyref{lem:GWBarrierAlpha}
will be proven together with other barrier estimates in \prettyref{sec:BoundaryCrossingProofs}.

To prove in \prettyref{prop:SmartMarkov} that no traversal process
violates the barrier condition we will use a union bound over the
scales: we aim to bound
\begin{equation}
\sum_{l=1}^{L-1}P_{x}\left[\sqrt{\tilde{T}_{l}^{y,t_{s}}}<\alpha\left(l\right)\mbox{ for some }y\in F_{L}\right].\label{eq:UBSecMarkovOverScales}
\end{equation}
We then use a packing argument that defines a further modifed traversal
counter $\hat{T}_{l}^{y,t_{s}}$ for each $y\in F_{l+\log L}$ where
the radii $r_{0},r_{1},r_{l}$ and $r_{l+1}$ have been slightly modfied
to ensure that if $y\in F_{L}$ and $y'\in F_{l+\log L}$ is the point
in $F_{l+\log L}$ closest to $y$ then, roughly speaking,
\begin{equation}
\hat{T}_{l}^{y^{'},t_{s}}\le\tilde{T}_{l}^{y,t_{s}}\mbox{ almost surely},\label{eq:Domination}
\end{equation}
(see \prettyref{fig:AnnulusPacking}). The only slightly modified
radii will mean that $\hat{T}_{l}^{y,t_{s}}$ has almost the same
law as $\tilde{T}_{l}^{y,t_{s}}$, and in particular in \prettyref{lem:LDForBinSumOfGeo}
we will derive a large deviation bound for $\hat{T}_{l}^{y,t_{s}}$
which is almost the same as the corresponding bound for $\tilde{T}_{l}^{y,t_{s}}$
(see \prettyref{rem:AlsoAppliesToTHat}; essentially, both $\hat{T}_{l}^{y,t_{s}}$
and $\tilde{T}_{l}^{y,t_{s}}$ turn out to be compound binomial random
variables with geometric compounding, so deriving a large deviation
bound is straightforward). The domination in \prettyref{eq:Domination}
will allow us to bound \prettyref{eq:UBSecMarkovOverScales} by 
\[
c\sum_{l=1}^{L-1}\left|F_{l+\log L}\right|P_{x}\left[\sqrt{\hat{T}_{l}^{y,t_{s}}}<\alpha\left(l\right)\right]
\]
and the aforementioned large deviation bound on $\hat{T}_{l}^{y,t_{s}}$
will show that $P_{x}\left[\sqrt{\hat{T}_{l}^{y,t_{s}}}<\alpha\left(l\right)\right]=o\left(L^{-1}\left|F_{l+\log L}\right|^{-1}\right)$,
so that we will be able to conclude that \prettyref{eq:UBSecMarkovOverScales}
is $o\left(1\right)$. Note that without the packing argument we would
be bounding $\sum_{l=1}^{L-1}\left|F_{L}\right|P_{x}\left[\sqrt{\tilde{T}_{l}^{y,t_{s}}}<\alpha\left(l\right)\right]$,
a quantity that can be shown to tend to infinity.
\begin{figure}
\includegraphics[scale=0.5]{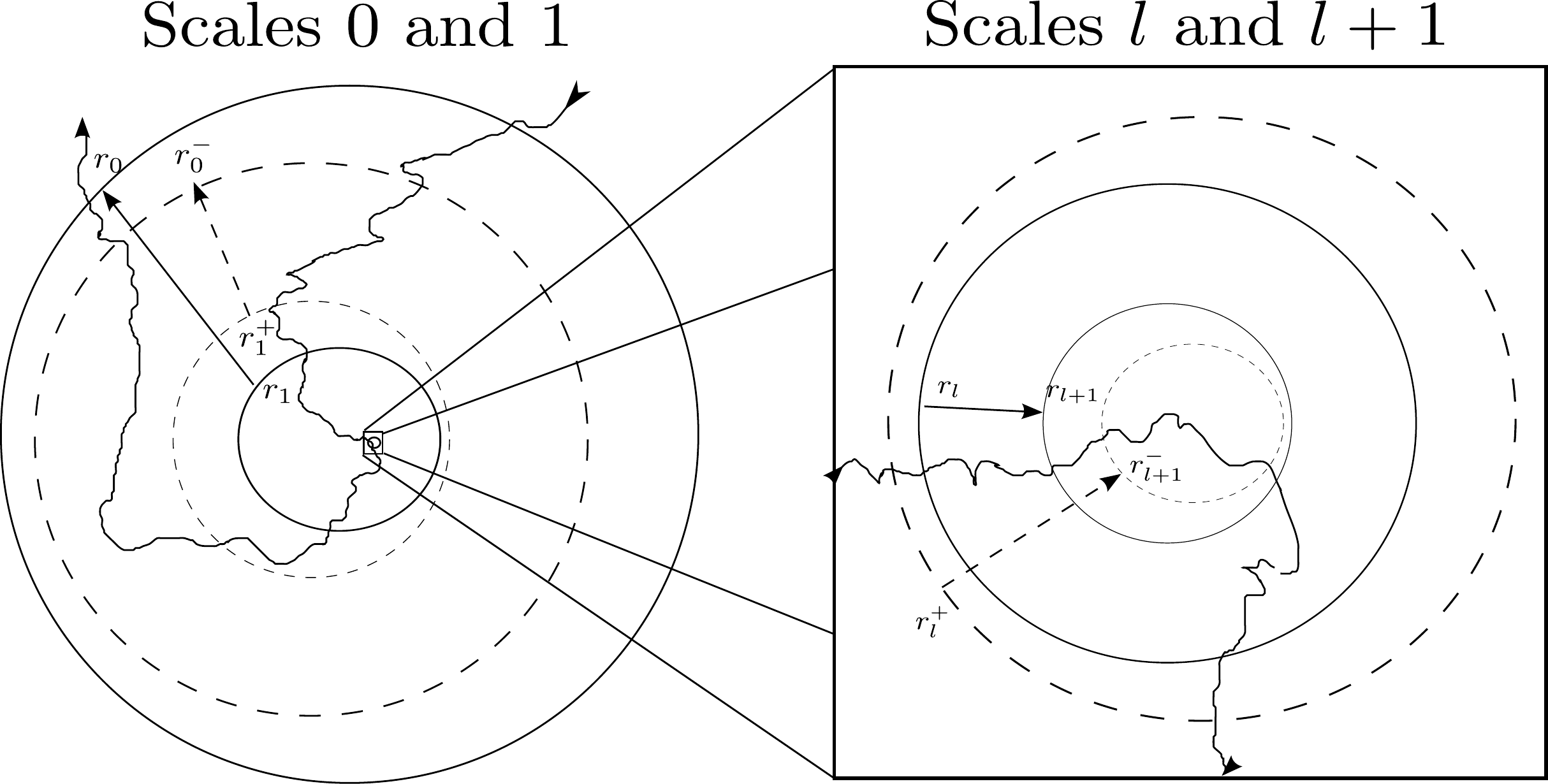}\caption{\label{fig:AnnulusPacking}An illustration of the packing used for
the proof of \prettyref{prop:SmartMarkov}. Each traversal counted
by $\hat{T}_{l}^{y,t_{s}}$ (that is a traversal between circles of
radii $r_{l}^{+}$ and $r_{l+1}^{-}$ before the $t_{s}-$th traversal
from the circle of radius $r_{1}^{+}$ to the circle of radius $r_{0}^{-}$,
dashed in the picture) gives rise to at least one traversal counted
by $\tilde{T}_{l}^{y,t_{s}}$ (that is a traversal between circles
of radii $r_{l}$ and $r_{l+1}$ before the $t_{s}-$th traversal
from the circle of radius $r_{1}$ to the circle of radius $r_{0}$,
solid in the picture). Therefore $\hat{T}_{l}^{y,t_{s}}\le T_{l}^{y,t_{s}}$.}
\end{figure}

We now state the barrier crossing bound for the Galton-Watson law
$\mathbb{G}_{t}$ (see \prettyref{fig:Illustration-of-barriers}).
\begin{lem}
\label{lem:GWBarrierAlpha}For any $s\in\left(-100,100\right)$ we
have that
\begin{equation}
\mathbb{G}_{t_{s}}\left[\sqrt{T_{l}}\ge\alpha\left(l\right)\mbox{ for }l\in\left\{ 0,\ldots,L-1\right\} |T_{L-1}=0\right]\le c\frac{\left(\log L\right)^{4}}{L}.\label{eq:GWBarrierAlpha}
\end{equation}

\end{lem}
\prettyref{lem:GWBarrierAlpha} will be proven in \prettyref{sec:BoundaryCrossingProofs},
together with further barrier crossing bounds that will be needed
in the proof of the lower bound in sections \ref{sec:Lowerbound}-\ref{sec:BoundaryCrossingProofs}.
We can now state and prove \prettyref{prop:UpperBoundOneProfile}
(the first main ingredient in the proof of \prettyref{prop:UpperBound}),
which bounds the expectation of the counting random variable in \prettyref{eq:UBSecTruncatedZ}.
Note that the bound goes to zero for all $s>0$.
\begin{prop}
\label{prop:UpperBoundOneProfile}$\left(x\in\mathbb{T}\right)$ For
all $y\in\mathbb{T}$ , $s\in\left(-100,100\right)$ and $L\ge1$
\begin{equation}
E_{x}\left[\sum_{y\in F_{L}}1_{\left\{ \tilde{T}_{L-1}^{y,t_{s}}=0\right\} \cap\left\{ \sqrt{\tilde{T}_{l}^{y,t_{s}}}\ge\alpha\left(l\right)\mbox{\,\ for }l=0,\ldots,L\right\} }\right]\le c\left(\log L\right)^{11/2}L^{-s}.\label{eq:UpperBoundOneProfile}
\end{equation}
\end{prop}
\begin{proof}
The expectation in \prettyref{eq:UpperBoundOneProfile} equals (cf.
\eqref{eq:truncmean})
\[
\left|F_{L}\right|\cdot P_{x}\left[\tilde{T}_{L-1}^{y,t_{s}}=0\right]\cdot P_{x}\left[\sqrt{\tilde{T}_{l}^{y,t_{s}}}\ge\alpha\left(l\right)\mbox{ for }l\in\left\{ 0,\ldots,L-1\right\} |\tilde{T}_{L-1}^{y,t_{s}}=0\right],
\]
for some arbitary $y\in\mathbb{T}$, where we have used that $\tilde{T}_{l}^{y,t_{s}},l\ge0,$
has the same law for all $y$ (see \eqref{eq:LawOfTTilde}). By \prettyref{eq:SizeOFfl}
and \prettyref{lem:NotHitByrL} the first two quantities are bounded
by
\[
c\left(\log L\right)^{3/2}e^{2L}\times e^{-2L}L^{1-s}\le c\left(\log L\right)^{3/2}L^{1-s}
\]
(for the latter we use the strong Markov property at time $T_{B\left(y,r_{1}\right)}$
when $y$ is such that $x\in B\left(y,r_{1}\right)^{\circ}$), cf.
\prettyref{eq:Sec3CountingRV}. The last probability equals, by \prettyref{eq:LawOfTTilde},
\[
\mathbb{G}_{t_{s}}\left[\sqrt{T_{l}}\ge\alpha\left(l\right)\mbox{ for }l\in\left\{ 0,\ldots,L-1\right\} |T_{L-1}=0\right].
\]
Thus by \prettyref{lem:GWBarrierAlpha} the expectation in \prettyref{eq:UpperBoundOneProfile}
is bounded above by
\[
c\left(\log L\right)^{3/2}L^{1-s}\times\frac{\left(\log L\right)^{4}}{L}=c\left(\log L\right)^{11/2}L^{-s}.
\]

\end{proof}
The next major step of this section is to prove \prettyref{prop:SmartMarkov},
exluding the possiblity that some traversal process violates the barrier
condition. As mentioned at the start of the section, we use a packing
argument. To this end define modified radii $r_{l}^{\pm}$ by
\begin{equation}
r_{l}^{-}=\left(1-\frac{100}{L}\right)r_{l}\mbox{ and }r_{l}^{+}=\left(1-\frac{100}{L}\right)^{-1}r_{l}\mbox{ for }l\ge0,\label{eq:ModifiedRadiiDef}
\end{equation}
and count for each $y\in F_{l+\log L}$ the number of traversals from
$\partial B\left(y,r_{l+1}^{-}\right)$ to $\partial B\left(y,r_{l}^{+}\right)$
during the first $t$ excursions from $\partial B\left(y,r_{1}^{+}\right)$
to $\partial B\left(y,r_{0}^{-}\right)$ as follows (cf. \prettyref{eq:ModifiedTraversals})
\begin{equation}
\hat{T}_{l}^{y,t}=\sup\left\{ n\ge0:R_{n}\left(y,r_{l}^{+},r_{l+1}^{-}\right)\circ\theta_{T_{B\left(y,r_{0}\right)}}\le D_{\lfloor t\rfloor}\left(y,r_{0}^{-},r_{1}^{+}\right)\circ\theta_{T_{B\left(y,r_{0}\right)}}\right\} .\label{eq:defofthat}
\end{equation}
For each $y\in\mathbb{T}$ let $y_{l}$ denote the point in $F_{l}$
closest to $y$ (breaking ties in some arbitrary way), and define
\[
\hat{T}_{l}^{y,t}=\hat{T}_{l}^{y_{l+\log L},t},\mbox{\,\ for }y\in\mathbb{T}\backslash F_{l+\log L}\mbox{ for all }t\ge0,l\ge1.
\]
With this construction $\tilde{T}_{l}^{y,t}$ dominates $\hat{T}_{l}^{y,t}$.
\begin{lem}
\label{lem:TraversalPackingDomination}For all $y\in\mathbb{T},l\ge1$
and $t\ge0$ we have that
\begin{equation}
\hat{T}_{l}^{y,t}\le\tilde{T}_{l}^{y,t}.\label{eq:TraversalPackingDomination}
\end{equation}
\end{lem}
\begin{proof}
By the definition \prettyref{eq:GridDefinition} of $F_{l+\log L}$
we have
\begin{equation}
d\left(y,y_{l+\log L}\right)\le r_{l+\log L}\overset{\eqref{eq:DefOfRadii}}{=}\frac{r_{l}}{L}=e\frac{r_{l+1}}{L}.\label{eq:Distance}
\end{equation}
Now because of the latter equality and the definition \eqref{eq:ModifiedRadiiDef}
of $r_{l}^{\pm}$ we have for $L$ large enough
\[
r_{l+1}^{-}+r_{l+\log L}\le r_{l+1}\mbox{ and }r_{l}+r_{l+\log L}\le r_{l}^{+},
\]
so that for all $y\in\mathbb{T}$
\begin{equation}
B\left(y_{l+\log L},r_{l+1}^{-}\right)\subset B\left(y,r_{l+1}\right)\subset B\left(y,r_{l}\right)\subset B\left(y_{l+\log L},r_{l}^{+}\right),\label{eq:UpperBoundSandwichA}
\end{equation}
and thus (recall \prettyref{eq:DefOfExcursionTimes} and \prettyref{eq:ExcursionTimeAbreviation})
\begin{equation}
R_{n}^{y,0}\circ\theta_{T_{B\left(y,r_{0}\right)}}\le R_{n}\left(y_{l+\log L},r_{l}^{+},r_{l+1}^{-}\right)\circ\theta_{T_{B\left(y,r_{0}\right)}}\mbox{\,\ for all }n\ge0,\label{eq:Rdomination}
\end{equation}
(see \prettyref{fig:AnnulusPacking}). Also \eqref{eq:ModifiedRadiiDef}
implies 
\[
r_{1}+r_{l+\log L}\le r_{1}^{+}\mbox{ and }r_{0}^{-}+r_{l+\log L}\le r_{0},
\]
so that
\begin{equation}
B\left(y,r_{1}\right)\subset B\left(y_{l+\log L},r_{1}^{+}\right)\subset B\left(y_{l+\log L},r_{0}^{-}\right)\subset B\left(y,r_{0}\right),\label{eq:UpperBoundSandwichB}
\end{equation}
and therefore (similarly to \prettyref{eq:Rdomination})
\begin{equation}
D_{n}\left(y_{l+\log L},r_{0}^{-},r_{1}^{+}\right)\circ\theta_{T_{B\left(y,r_{0}\right)}}\le D_{n}^{y,0}\circ\theta_{T_{B\left(y,r_{0}\right)}}\mbox{\,\ for all }n\ge0,\label{eq:Ddomination}
\end{equation}
Now by the definitions of $\tilde{T}_{t}^{y,l}$ and $\hat{T}_{l}^{y,t}$
(see \prettyref{eq:ModifiedTraversals} and \prettyref{eq:defofthat})
the claim \prettyref{eq:TraversalPackingDomination} now follows from
\prettyref{eq:Rdomination} and \prettyref{eq:Ddomination}.
\end{proof}
We now show that $\hat{T}_{l}^{y,t}$ has a binomial-geometric compound
distribution.
\begin{lem}
\label{lem:ModifedTraversalsAreCompoundGeo}Let
\begin{equation}
p=\frac{\log\left(r_{l}^{+}/r_{l+1}^{-}\right)}{\log\left(r_{0}^{-}/r_{l+1}^{-}\right)}\mbox{ and }q=\frac{\log\left(r_{0}^{-}/r_{1}^{+}\right)}{\log\left(r_{0}^{-}/r_{l+1}^{-}\right)}.\label{eq:pqdef}
\end{equation}
Let $G_{1},G_{2},\ldots$ be geometric random variables with support
$\left\{ 1,2,\ldots\right\} $ and success probability $p$ and let
$J_{1},J_{2},\ldots$ be Bernoulli random variables with success probability
$q$, all mutually independent. We have for all $y\in F_{l+\log L},t\ge0$
and $l\ge1$ that
\begin{equation}
\hat{T}_{l}^{y,t}\overset{\mbox{law}}{=}\sum_{i=1}^{\lfloor t\rfloor}J_{i}G_{i}.\label{eq:ModifedTraversalsAreCompoundGeo}
\end{equation}
\end{lem}
\begin{proof}
By \prettyref{eq:AnnulusExitProb} each excursion of Brownian motion
from scale $1$ to scale $0$ (from $\partial B\left(y,r_{1}^{+}\right)$
to $\partial B\left(y,r_{0}^{-}\right)$) has probability $q$ of
giving rise to at least one traversal from scale $l$ to $l+1$ (i.e.
of hitting $B\left(y,r_{l+1}^{-}\right)$ before leaving $\partial B\left(y,r_{0}^{-}\right)$).
After hitting $B\left(y,r_{l+1}^{-}\right)$ Brownian motion returns
to $\partial B\left(y,r_{l}^{+}\right)$, and from there it has a
probability $p$ to escape to $\partial B\left(y,r_{0}^{-}\right)$
(again by \prettyref{eq:AnnulusExitProb}) and end the traversal from
scale $1$ to scale $0$; otherwise it returns to $B\left(y,r_{l+1}^{-}\right)$
which gives rise to another traversal from scale $l$ to $l+1$. Each
of these ``coin flips'' are independent by the strong Markov property,
and thus \prettyref{eq:ModifedTraversalsAreCompoundGeo} follows.\end{proof}
\begin{rem}
\label{rem:AlsoAppliesToTHat}Note that by \prettyref{eq:DefOfRadii}
and \prettyref{eq:ModifiedRadiiDef}
\begin{equation}
p,q=\frac{1}{l+1}+O\left(L^{-1}\right)\mbox{\,\ for }p,q\mbox{ as in }\eqref{eq:pqdef},\label{eq:pqvalue}
\end{equation}
and that the argument giving \prettyref{eq:ModifedTraversalsAreCompoundGeo}
applies equally well to $\tilde{T}_{l}^{y,t}$ but with modifed $p$
and $q$ given by
\[
p=\frac{\log\left(r_{l}/r_{l+1}\right)}{\log\left(r_{0}/r_{l+1}\right)}\overset{\eqref{eq:DefOfRadii}}{=}\frac{1}{l+1}\overset{\eqref{eq:DefOfRadii}}{=}\frac{\log\left(r_{0}/r_{1}\right)}{\log\left(r_{0}/r_{l+1}\right)}=q.\qed
\]

\end{rem}

We will need a lemma on the large deviations of $\hat{T}_{l}^{y,t}$.
We state it for a general geometric distribution with binomial compounding,
and postpone the proof until the appendix.
\begin{lem}
\label{lem:LDForBinSumOfGeo}Let $G_{1},G_{2},\ldots,J_{1},J_{2},\ldots$
be as in \prettyref{lem:ModifedTraversalsAreCompoundGeo} for $p\in\left(0,1\right)$
and $q\in\left(0,1\right)$. Then for all integers $n\ge1$ and $\theta\le n\frac{q}{p}$
\begin{equation}
\mathbb{P}\left[\sum_{i=1}^{n}J_{i}G_{i}\le\theta\right]\le\exp\left(-\left(\sqrt{q\theta}-\sqrt{pn}\right)^{2}\right).\label{eq:LDForBinSumOfGeoLowerBound}
\end{equation}

\end{lem}
We can now use \prettyref{lem:TraversalPackingDomination}, \prettyref{lem:ModifedTraversalsAreCompoundGeo}
and \prettyref{lem:LDForBinSumOfGeo} to deduce \prettyref{prop:SmartMarkov},
proving that no traversal process violates the barrier condition.
\begin{prop}
\label{prop:SmartMarkov}$\left(x\in\mathbb{T}\right)$ For all $s\in\left(-100,100\right)$
\begin{equation}
\lim_{L\to\infty}P_{x}\left[\exists y\in F_{L}\mbox{ s.t. }\sqrt{\tilde{T}_{l}^{y,t_{s}}}\le\alpha\left(l\right)\mbox{ for }l\in\left\{ 0,\ldots,L-1\right\} \right]=0.\label{eq:SmartMarkov}
\end{equation}
\end{prop}
\begin{proof}
By \prettyref{lem:TraversalPackingDomination} and the definition
\prettyref{eq:defofthat} of $\hat{T}_{l}^{y,t}$ the probability
in \prettyref{eq:SmartMarkov} is bounded above by
\begin{equation}
\begin{array}{ccl}
{\displaystyle \sum_{l=1}^{L-1}}{\displaystyle \sum_{y\in F_{l+\log L}}}P_{x}\left[\sqrt{\hat{T}_{l}^{y,t_{s}}}\le\alpha\left(l\right)\right] & = & {\displaystyle \sum_{l=1}^{L-1}}{\displaystyle \left|F_{l+\log L}\right|}P_{x}\left[\sqrt{\hat{T}_{l}^{y,t_{s}}}\le\alpha\left(l\right)\right]\\
 & \overset{\eqref{eq:SizeOFfl}}{\le} & c\left(\log L\right)^{3/2}L^{2}{\displaystyle \sum_{l=1}^{L-1}}e^{2l}P_{x}\left[\sqrt{\hat{T}_{l}^{y,t_{s}}}\le\alpha\left(l\right)\right],
\end{array}\label{eq:FirstBound}
\end{equation}
for some arbitrary $y\in F_{l+\log L}$. Using \prettyref{lem:ModifedTraversalsAreCompoundGeo},
\prettyref{lem:LDForBinSumOfGeo} and \eqref{eq:pqvalue} it follows
that $P_{x}\left[\sqrt{\hat{T}_{l}^{y,t_{s}}}\le\alpha\left(l\right)\right]$
is bounded above by
\[
\begin{array}{cclcl}
ce^{-\frac{\left(\alpha\left(l\right)-\sqrt{t_{s}}+O\left(1\right)\right)^{2}}{l+1}} & \overset{\eqref{eq:AlphaBarrierDef},\eqref{eq:BetaDef}}{\le} & ce^{-\frac{\left(\frac{l}{L}\sqrt{t_{s}}+\left(\log L\right)^{2}-c\right)^{2}}{l+1}} & \le & ce^{-\frac{\left(\frac{l}{L}\right)^{2}t_{s}}{l+1}-c\left(\log L\right)^{2}}\\
 &  &  & \le & ce^{-2l-c\left(\log L\right)^{2}},
\end{array}
\]
where we have used that $t_{s}=2L^{2}\left(1+O\left(\log L/L\right)\right)$.
Thus the probability in \eqref{eq:SmartMarkov} is bounded above by
\[
c\left(\log L\right)^{3/2}L^{2}\sum_{l=1}^{L-1}e^{2l}\times e^{-2l-c\left(\log L\right)^{2}}\le c\left(\log L\right)^{3/2}L^{3}e^{-c\left(\log L\right)^{2}}=o\left(1\right).
\]

\end{proof}
We can now wrap up this section by proving \prettyref{prop:UpperBound}.
\begin{proof}[Proof of \prettyref{prop:UpperBound}]
Since the probability in \prettyref{eq:UppberBound} is decreasing
in $s$ it suffices to consider $s\in\left(0,100\right)$. For such
$s$ the statement \prettyref{eq:UBInUBSectionIgnoring} (and therefore
the claim \prettyref{eq:UppberBound}) follows immediately from the
Markov inequality, \prettyref{prop:UpperBoundOneProfile} (the right-hand
side tends to zero for $s>0$) and \prettyref{prop:SmartMarkov}.
\end{proof}
To complete the proof of our main result \prettyref{thm:MainResult}
it remains to show the lower bound in terms of excursions \prettyref{prop:LowerBound}
and the concentration of excursion times \prettyref{prop:Concentration},
in addition to the barrier estimate \prettyref{lem:GWBarrierAlpha}
and the large deviation bound \prettyref{lem:LDForBinSumOfGeo} used
in this section.

\section{\label{sec:Lowerbound}Lower bound on cover time in terms of excursions}

In this section we prove \prettyref{prop:LowerBound}, which gives
the lower bound of the main result \prettyref{thm:MainResult} ``in
terms of excursions'', and was used in the proof of that result in
\prettyref{sec:Reduction}. More precisely, our goal is to show the
claim \eqref{eq:LowerBoundInRedSec} from \prettyref{prop:LowerBound}
that
\[
\lim_{L\to\infty}P_{x}\left[T_{L-1}^{y,t_{-s}}=0\mbox{ for some }y\in F_{L}\right]=1\mbox{ for all }s>0\tag{\ref{eq:LowerBoundInRedSec}'},
\]
(see \eqref{eq:TraversalsDef}, \eqref{eq:GridDefinition} and \eqref{eq:defofts}
for the definitions).

As previously mentioned, a natural approach is to apply the second
moment method to the counting random variable $\sum_{y\in F_{L}}1_{\left\{ T_{L-1}^{y,t_{-s}}=0\right\} }$,
but this fails as the second moment of this sum is much larger than
the first moment squared. To get around this we introduce a truncation,
which takes the form of a barrier condition. The main point of the
condition is that we require $\sqrt{T_{l}^{y}}$ to stay above $\gamma\left(l\right)$,
where 
\begin{equation}
\gamma\left(l\right)=\gamma\left(l,L,s\right)=\beta\left(l\right)+f\left(l\right),\label{eq:GammaBarrierDef}
\end{equation}
$\beta\left(l\right)$ is the linear function from \prettyref{eq:BetaDef}
(essentially the mean of the $\sqrt{T_{l}^{t_{-s}}}$ when conditioned
on $T_{L-1}^{t_{-s}}=0$) and $f\left(l\right)$ is a convex ``bump''
function given by 
\begin{equation}
f\left(l\right)=f\left(l;L\right)=\min\left(l^{0.49},\left(L-l\right)^{0.49}\right),l\in\left[0,L\right],\label{eq:DefOffFunc}
\end{equation}
(see \prettyref{fig:Illustration-of-barriers}). It turns out that
with this condition, the summands in the counting random variable
decorrelate enough so that the variance should morally speaking not
explode with respect to the first moment squared.

For technical reasons it turns out to help to introduce a second barrier
\begin{equation}
\delta\left(l\right)=\delta\left(l,L,s\right)=\beta\left(l\right)+g\left(l\right),\label{eq:DeltaBarrierDef}
\end{equation}
where $g\left(l\right)\ge f\left(l\right)$ is the larger convex ``bump'',
\begin{equation}
g\left(l\right)=g\left(l;L\right)=\min\left(l^{0.51},\left(L-l\right)^{0.51}\right),l\in\left[0,L\right],\label{eq:DefOfgfunc}
\end{equation}
and require also that the square root of the traversal processes stay
below $\delta\left(l\right)$, or in other words that they stay in
the ``tube'' bounded by $\gamma\left(l\right)$ and $\delta\left(l\right)$.
Furthermore it turns out to be too much to ask for the barrier condition
to be satisfied for $l$ close to $0$ or $L-1$. We therefore introduce
a cutoff
\begin{equation}
l_{0}=l_{0}\left(L\right)=\lfloor\frac{1}{10}\log\log L\rfloor,\label{eq:DefOfCutOff}
\end{equation}
and arrive at the final form of the summands
\begin{equation}
I_{y}=\left\{ \gamma\left(l\right)\le\sqrt{T_{l}^{y,t_{-s}}}\le\delta\left(l\right)\mbox{ for }l=l_{0},\ldots,L-l_{0}\mbox{ and }T_{L-1}^{y,t_{-s}}=0\right\} ,\label{eq:TruncatedSummandLB}
\end{equation}
for $y\in F_{L}$. Finally, since the \prettyref{lem:BranchingProc}
gives the law of $T_{l}^{y,t_{-s}}$ technically speaking only applies
when $x\notin B\left(y,r_{1}\right)^{\circ}$ we sum not over $F_{L}$
but over the smaller set 
\begin{equation}
\tilde{F}_{L}=F_{L}\backslash B\left(x,r_{0}\right).\label{eq:FTildeDef}
\end{equation}
Our truncated counting random variable is thus 
\begin{equation}
Z=\sum_{y\in\tilde{F}_{L}}1_{I_{y}}.\label{eq:LBTruncatedSum}
\end{equation}
Note that $\tilde{F}_{L}$ is only marginaly smaller than $F_{L}$
since $\left|F_{L}\cap B\left(x,r_{0}\right)\right|/\left|F_{L}\right|\le cr_{0}^{2}\to0$
by \prettyref{eq:DefOfRadii} and \prettyref{eq:GridDefinition},
so that with \prettyref{eq:SizeOFfl} 
\begin{equation}
\left|\tilde{F}_{L}\right|=\left(1-o\left(1\right)\right)\left|F_{L}\right|\asymp c\left(\log L\right)^{3/2}e^{2L}.\label{eq:SizeOfSetOfFeps}
\end{equation}
Obviously
\begin{equation}
\left\{ Z>0\right\} \subset\left\{ T_{L-1}^{y,t_{-s}}=0\mbox{ for some }y\in F_{L}\right\} .\label{eq:ZPositiveMeansSomeTzero}
\end{equation}
We will show that in fact $Z>0$ with probability tending to one,
giving our goal \prettyref{eq:LowerBoundInRedSec}. This will be done
in two steps. First we will show in \prettyref{prop:LowerBoundOneProfile}
that for all $s>0$ 
\begin{equation}
E_{x}\left[Z\right]\asymp\left(\log L\right)^{3/2}l_{0}L^{s}\to\infty,\mbox{ as }L\to\infty.\label{eq:ExpOfZGoesToInfinity}
\end{equation}
For the second step (which is considerably more challenging) we show
in \prettyref{prop:MainTwoProfileBound} that for all $s>0$
\begin{equation}
E_{x}\left[Z^{2}\right]=\left(E_{x}\left[Z\right]\right)^{2}\left(1+o\left(1\right)\right),\mbox{ as }L\to\infty.\label{eq:SecondMomentOfZ}
\end{equation}
We will see that the lower bound \prettyref{prop:LowerBound} (i.e.
\prettyref{eq:LowerBoundInRedSec}) is an easy consequence of \prettyref{eq:SecondMomentOfZ},
via the Paley-Zygmund inequality.

The proof of \prettyref{eq:SecondMomentOfZ} is the heart of this
section. The second moment $E_{x}\left[Z^{2}\right]$ is a sum of
``two point probabilities'' $P_{x}\left[I_{y}\cap I_{z}\right]$
for $y,z\in\tilde{F}_{L}$, and bounding $E_{x}\left[Z^{2}\right]$
amounts to bounding these terms. Since $r_{0}\downarrow0$ most pairs
are at distance at least $2r_{0}$, and it turns out that for such
pairs the events $I_{y}$ and $I_{z}$ are exactly independent (essentially
because they depend on the behaviour of Brownian motion in disjoint
balls $B\left(y,r_{0}\right)$ and $B\left(z,r_{0}\right)$). Because
of this, the contribution of such terms to the second moment $E_{x}\left[Z^{2}\right]$
will be shown to be at most $E_{x}\left[Z\right]^{2}$. Thus the proof
\prettyref{eq:SecondMomentOfZ} is about showing that terms for pairs
at distance less than $2r_{0}$ are negligible, or in other words
\begin{equation}
{\displaystyle \sum_{y,z\in\tilde{F}_{L}:d\left(y,z\right)<2r_{0}}}P_{x}\left[I_{y}\cap I_{z}\right]=o\left(E_{x}\left[Z\right]^{2}\right).\label{eq:GoalToShow}
\end{equation}

\prettyref{lem:TwoProfileBoundEdgeCase} and \prettyref{prop:MainTwoProfileBound}
will provide bounds for $P_{x}\left[I_{y}\cap I_{z}\right]$ in this
regime that will allow us to show \prettyref{eq:GoalToShow}. We state
\prettyref{lem:TwoProfileBoundEdgeCase} and \prettyref{prop:MainTwoProfileBound}
in this section, but since their proofs are intricate (especially
that of \prettyref{prop:MainTwoProfileBound}, the main bound) they
are postponed until the next section.

Before starting the  proofs we state a bound on the probability that
a conditioned Galton-Watson process stays in the tube bounded by $\gamma\left(l\right)$
and $\delta\left(l\right)$. It will be proven (together with the
barrier bound \prettyref{lem:GWBarrierAlpha} from the previous section)
in \prettyref{sec:BoundaryCrossingProofs}.
\begin{lem}
\label{lem:GaltonWatsonBarrierBoundsSecLB}For all $s\in\left(-1,1\right)$
we have that\textup{
\begin{equation}
\mathbb{G}_{t_{s}}\left[\gamma\left(l\right)\le\sqrt{T_{l}}\le\delta\left(l\right)\mbox{ for }l\in\left\{ l_{0},\ldots,L-l_{0}\right\} |T_{L-1}=0\right]\asymp\frac{l_{0}}{L}.\label{eq:GWBarrierGammaDelta}
\end{equation}
}
\end{lem}
We now start the proofs of this section by proving the estimate \prettyref{eq:ExpOfZGoesToInfinity}
on $E_{x}\left[Z\right]$.
\begin{prop}
\label{prop:LowerBoundOneProfile}($x\in\mathbb{T}$) For all $s\in\left(-1,1\right)$
\begin{equation}
E_{x}\left[Z\right]\asymp c\left(\log L\right)^{3/2}l_{0}L^{s},\label{eq:LowerBoundOneProfileExp}
\end{equation}
and for all $y,z\in\tilde{F}_{L}$,
\begin{equation}
P_{x}\left[I_{y}\right]=P_{x}\left[I_{z}\right]\asymp ce^{-2L}L^{s}l_{0}.\label{eq:LowerBoundOneProfile}
\end{equation}
\end{prop}
\begin{proof}
By \prettyref{eq:SizeOfSetOfFeps} the first claim \prettyref{eq:LowerBoundOneProfileExp}
follows from \prettyref{eq:LowerBoundOneProfile}. The equality in
\prettyref{eq:LowerBoundOneProfile} holds since $T_{l}^{y,t_{-s}},l\ge0,$
and $T_{l}^{z,t_{-s}},l\ge0,$ have the same law by \prettyref{lem:BranchingProc}.
For the bound in \prettyref{eq:LowerBoundOneProfile} note that (recall
\prettyref{eq:TruncatedSummandLB})
\begin{equation}
P_{x}\left[I_{y}\right]=P_{x}\left[I_{y}|T_{L-1}^{y,t_{-s}}=0\right]P_{x}\left[T_{L-1}^{y,t_{-s}}=0\right].\label{eq:Conditioning}
\end{equation}
By \prettyref{lem:NotHitByrL} we have
\begin{equation}
P_{x}\left[T_{L-1}^{y,t_{-s}}=0\right]\asymp e^{-2L}L^{1+s},\label{eq:AlittleBitOfColumnA}
\end{equation}
and using \prettyref{lem:BranchingProc}
\[
P_{x}\left[I_{y}|T_{L-1}^{y,t_{-s}}=0\right]=\mathbb{G}_{t_{s}}\left[\gamma\left(l\right)\le\sqrt{T_{l}}\le\delta\left(l\right)\mbox{ for }l=l_{0},\ldots,L-l_{0}|T_{L-1}=0\right].
\]
Thus by \prettyref{eq:GWBarrierGammaDelta} it follows that
\begin{equation}
P_{x}\left[I_{y}|T_{L-1}^{y,t_{-s}}=0\right]\asymp\frac{l_{0}}{L}.\label{eq:ALittleBitOfColumnB}
\end{equation}
Plugging this into \prettyref{eq:Conditioning} together with \prettyref{eq:AlittleBitOfColumnA}
gives \prettyref{eq:LowerBoundOneProfile}.
\end{proof}
We now turn our attention to the main step of the proof of the lower
bound \prettyref{prop:LowerBound}, namely the second moment bound
\prettyref{eq:SecondMomentOfZ}. For this we will need bounds on the
two point probility $P_{x}\left[I_{y}\cap I_{z}\right]$.

We start with the case of $y$ and $z$ such that $B\left(y,r_{0}\right)$
and $B\left(z,r_{0}\right)$ are disjoint. The events will be independent
in this case, and to show this we need the independence result \prettyref{eq:IndepOfDisjoint}
which now follows (we also include \prettyref{eq:IndepOfStartingPoint}
since it will be used later in \prettyref{sec:TwoPointBound} and
its proof is similar).
\begin{lem}
\label{lem:IndepOfDisjoint}($x\in\mathbb{T}$) For all $t\ge0$ and
$y,z\in\tilde{F}_{L}$ 
\begin{equation}
\mbox{if }d\left(y,z\right)>2r_{0}\mbox{ then }\left(T_{l}^{y,t}\right)_{l\ge0}\mbox{ and }\left(T_{l}^{z,t}\right)_{l\ge0}\mbox{ are independent under }P_{x}.\label{eq:IndepOfDisjoint}
\end{equation}
Also 
\begin{equation}
\mbox{for }w,v\in\mathbb{T}\mbox{ the }P_{w}-\mbox{law of }\left(T_{l}^{v,t}\right)_{l\ge0,t\ge0}\mbox{\,\ depends only on }d\left(w,v\right).\label{eq:IndepOfStartingPoint}
\end{equation}
\end{lem}
\begin{proof}
To see \eqref{eq:IndepOfStartingPoint} note that $T_{l}^{v,t},l\ge0,$
depends only on $T_{l}^{v,n}-T_{l}^{v,n-1}=T_{l}^{v,1}\circ\theta_{R_{n}^{v,0}},l\ge0,n\ge1,$
where $T_{l}^{v,1}\circ\theta_{R_{n}^{v,0}}$ counts the traversals
that happen during the $n$-th excursion from $\partial B\left(v,r_{1}\right)$
to $\partial B\left(v,r_{0}\right)$. The traversal count $T_{l}^{v,1}\circ\theta_{R_{n}^{v,0}}$
depends only on the excursion $W_{\left(R_{n}^{v,0}+\cdot\right)\wedge D_{n}^{v,0}}$,
and furthermore it is a rotationally invariant function of that excursion.
Let $\tilde{w}\in\mathbb{T}$ be any point such that $d\left(\tilde{w},v\right)=d\left(w,v\right)$
and let $u$ be any point in $\partial B\left(v,r_{1}\right)$. By
the rotational invariance \prettyref{eq:RotationalInvarianceInTorusBall}
of $W_{t}$ in $B\left(v,r_{0}\right)$ and the strong Markov property
the $P_{w}-$ and $P_{\tilde{w}}$-laws of $T_{l}^{v,1}\circ\theta_{R_{1}^{v,0}}$
coincide (that is the law depends only on $d\left(w,v\right)$), and
the $P_{w}-$ and $P_{u}-$laws of $T_{l}^{v,1}\circ\theta_{R_{n}^{v,1}}$
coincide for $n\ge2$ (that is the law does not even depend on $d\left(w,v\right)$).
Furthermore the strong Markov property implies that the $T_{l}^{v,1}\circ\theta_{R_{n}^{v,0}},n\ge1,$
are independent. This gives \eqref{eq:IndepOfStartingPoint}.

To see \eqref{eq:IndepOfDisjoint} we similarly use that for $v\in\left\{ y,z\right\} $
the process $T_{l}^{v,t},l\ge0,$ depends only on $T_{l}^{v,1}\circ\theta_{R_{n}^{y,0}},l\ge0$,
which in turn depend only on the excursions $W_{\left(R_{n}^{v,0}+\cdot\right)\wedge D_{n}^{v,0}},n\ge0$.
The latter excursions refer to disjoint intervals of time for each
$n\ge0$ and $v\in\left\{ y,z\right\} $, since $B\left(y,r_{0}\right)$
and $B\left(z,r_{0}\right)$ are disjoint. Therefore, using rotational
invariance and the strong Markov property as above, the processes
$l\to T_{l}^{v,1}\circ\theta_{R_{n}^{y,0}},v\in\left\{ y,z\right\} ,n\ge1,$
are mutually independent. This implies \eqref{eq:IndepOfDisjoint}.
\end{proof}
The two point probability for $y$ and $z$ such that $B\left(y,r_{0}\right)$
and $B\left(z,r_{0}\right)$ are disjoint can now be computed easily.
\begin{cor}
\label{cor:IndeptOfIyAndIz}($x\in\mathbb{T}$) For all $y,z\in\tilde{F}_{L}$
such that $d\left(y,z\right)>2r_{0}$
\[
P_{x}\left[I_{y}\cap I_{z}\right]=P_{x}\left[I_{y}\right]^{2}.
\]

\end{cor}
Next we state a bound on the two point probability for $y$ and $z$
for which the largest non-overlapping balls $B\left(y,r_{k}\right)$
and $B\left(z,r_{k}\right)$ are of radius $r_{k}$ for $0\le k\le l_{0}$.
It will be proven in the next section. Note that the right-hand side
is almost that of \prettyref{eq:LowerBoundOneProfile} squared.
\begin{lem}
\label{lem:TwoProfileBoundEdgeCase}For all $y,z\in\tilde{F}_{L}$
such that $2r_{l_{0}}\le d\left(y,z\right)<2r_{0}$ and $s\in\left(-1,1\right)$
we have
\begin{equation}
P_{x}\left[I_{y}\cap I_{z}\right]\le c\left(e^{-\left(2L-2l_{0}\right)}L^{s}l_{0}^{0.51}g\left(l_{0}\right)\right)^{2}.\label{eq:TwoProfileBoundEdgeCase}
\end{equation}

\end{lem}
Next we we state the two point probability bound for the most important
(and difficult) regime, which gives a bound for points $y$ and $z$
which are such that the largest non-overlapping ball is of radius
$r_{k}$ for $l_{0}<k<L-l_{0}$.
\begin{prop}
\label{prop:MainTwoProfileBound}($x\in\mathbb{T}$) For all $s\in\left(0,1\right)$,
$l_{0}<k<L-l_{0}$ and all $y,z\in\tilde{F}_{L}$ such that $2r_{k}<d\left(y,z\right)\le2r_{k-1}$
we have
\begin{equation}
P_{x}\left[I_{y}\cap I_{z}\right]\le c\left(s\right)e^{-\left(4L-2k\right)-cf\left(k\right)}L^{2s}l_{0}^{1.02}g\left(k\right)^{2}\left(\log L\right)^{1.02}.\label{eq:TwoProfileBoundInStatement}
\end{equation}
\end{prop}
\begin{rem}
\label{rem:2ndmomentboundremark}This bound is key to the whole approach.
Since the proof (which is carried out in the next section) is involved,
let us spend a few words on the heuristic which explains it. By \eqref{eq:LowerBoundOneProfile}
the claim \eqref{eq:TwoProfileBoundInStatement} is equivalent to
\begin{equation}
P_{x}\left[I_{y}|I_{z}\right]\le e^{-2\left(L-k\right)-cf\left(k\right)}L^{s}l_{0}^{0.02}g\left(k\right)^{2}\left(\log L\right)^{1.02}.\label{eq:ConditionedRigorousBound}
\end{equation}
Recall that because of the approximate hierarchical structure, we
expect that $T_{l}^{y,t_{-s}}$ and $T_{l}^{z,t_{-s}}$ roughly coincide
for $l\le k$ and ``decouple'' for $l\ge k$ (see \prettyref{fig:PosOfCirclesTWB}).
Therefore, avoiding the ball $B\left(z,r_{L}\right)$ in $t_{-s}$
excursions from $\partial B(z,r_{1})$ to $\partial B\left(z,r_{0}\right)$,
when conditioning on $I_{y}$, is essentially equivalent to avoiding
$B\left(z,r_{L}\right)$ in $\gamma\left(k\right)^{2}$ excursions
from $\partial B\left(z,r_{k+1}\right)$ to $\partial B\left(z,r_{k}\right)$.
By \prettyref{eq:BallExitProbForRadii} each such excursion avoids
$B\left(z,r_{L}\right)$ with probability $1-\frac{1}{L-k}$. We therefore
expect that $P_{x}\left[I_{z}|I_{y}\right]$ is essentially at most
\begin{equation}
\begin{array}{l}
\left(1-\frac{1}{L-k}\right)^{\gamma\left(k\right)^{2}}\\
\times P_{z}\left[\gamma\left(l\right)\le\sqrt{T_{l}^{z,t_{-s}}}\mbox{ for }l=k,\ldots,L-l_{0}|T_{k}^{z,t_{-s}}=\gamma\left(k\right)^{2},T_{L-1}^{z,t_{-s}}=0\right].
\end{array}\label{eq:MorallySpeaking}
\end{equation}
Straight-forward computation and the definition \eqref{eq:GammaBarrierDef}
of $\gamma\left(k\right)$ gives that the top line of \eqref{eq:MorallySpeaking}
is at most $e^{-2\left(L-k\right)-cf\left(k\right)}L^{\left(1+s\right)\left(1-\frac{k}{L}\right)}$.
Furthermore, the process $\sqrt{T_{\cdot}^{z,t_{-s}}}$ should behave
roughly as a Gaussian process, so that the conditional probability
in \eqref{eq:MorallySpeaking} should correspond to the probability
that a Brownian bridge starting at $\gamma\left(l\right)$ at time
$0$ and ending at $0$ at time $L-k$ stays above the linear function
with the same starting and ending points during the time interval
$\left[0,L-k-l_{0}\right]$. This probability is of order $\sqrt{l_{0}}/\left(L-k-l_{0}\right)$,
e.g. by the reflection principle. These considerations thus suggest
that $P_{x}\left[I_{y}|I_{z}\right]$ should essentially be upper-bounded
by $ce^{-\left(2L-2l\right)-cf\left(k\right)}L^{s}\sqrt{l_{0}}$,
which is (marginally) better than the bound we derive rigorously.
\end{rem}
Finally for the last case, that is when the largest non-overlapping
balls $B\left(y,r_{k}\right)$ and $B\left(z,r_{k}\right)$ have radius
$r_{k}$ for $k\ge L-l_{0}$, we have the following trivial bound
which follows directly from \prettyref{eq:LowerBoundOneProfile}
\begin{equation}
P_{x}\left[I_{y}\cap I_{z}\right]\le P_{x}\left[I_{y}\right]\le ce^{-2L}l_{0}L^{s}\mbox{ for all }y,z\in\tilde{F}_{l}.\label{eq:yzfarawaytrivial}
\end{equation}
We have now arrived at the heart of this section, which is the bound
on the second moment of the counting random variable $Z$.
\begin{prop}
\label{prop:SecondMomentOfA}$(x\in\mathbb{T}$) For all $s>0$
\begin{equation}
E_{x}\left[Z^{2}\right]=\left(E_{x}\left[Z\right]\right)^{2}\left(1+o\left(1\right)\right),\mbox{ as }L\to\infty.\label{eq:SecondMomentOfA}
\end{equation}
\end{prop}
\begin{proof}
Write
\[
E_{x}\left[Z^{2}\right]=\sum_{y,z\in\tilde{F}_{L}}P_{x}\left[I_{y}\cap I_{z}\right].
\]
Decompose the set of pairs of $y,z\in\tilde{F}_{L}$ by setting
\[
\begin{array}{ccl}
G_{0} & = & \left\{ \left(y,z\right):y,z\in\tilde{F}_{L}\mbox{ s.t. }d\left(y,z\right)>2r_{0}\right\} ,\\
G_{k} & = & \left\{ \left(y,z\right):y,z\in\tilde{F}_{L}\mbox{ s.t. }2r_{k}<d\left(y,z\right)\le2r_{k-1}\right\} \mbox{ for }1\le k<L,\\
G_{L} & = & \left\{ \left(y,z\right):y,z\in\tilde{F}_{L}\mbox{ s.t. }d\left(y,z\right)\le2r_{L-1}\right\} .
\end{array}
\]
We have that $\bigcup_{k=0}^{L}G_{k}=\tilde{F}_{L}\times\tilde{F}_{L}$
and therefore 
\begin{equation}
E_{x}\left[Z^{2}\right]=\sum_{\left\{ y,z\right\} \in G_{0}}P_{x}\left[I_{y}\cap I_{z}\right]+\sum_{k=1}^{L}\sum_{\left\{ y,z\right\} \in G_{k}}P_{x}\left[I_{y}\cap I_{z}\right].\label{eq:SecondMomentDecomp}
\end{equation}
By \prettyref{cor:IndeptOfIyAndIz} we have
\begin{equation}
\sum_{\left\{ y,z\right\} \in G_{0}}P_{x}\left[I_{y}\cap I_{z}\right]=\sum_{\left\{ y,z\right\} \in G_{0}}P_{x}\left[I_{y}\right]P_{x}\left[I_{z}\right]\le\sum_{y,z\in\tilde{F}_{L}}P_{x}\left[I_{y}\right]P_{x}\left[I_{z}\right]=\left(E_{x}\left[Z\right]\right)^{2},\label{eq:FirstTermInDecomp}
\end{equation}
so that \prettyref{eq:SecondMomentOfA} will follow once we have
shown that
\begin{equation}
\sum_{k=1}^{L}\sum_{\left\{ y,z\right\} \in G_{k}}P_{x}\left[I_{y}\cap I_{z}\right]=o\left(E_{x}\left[Z\right]^{2}\right).\label{eq:RestOfTerms}
\end{equation}
We first bound
\begin{equation}
\sum_{k=1}^{L}\sum_{\left\{ y,z\right\} \in G_{k}}P_{x}\left[I_{y}\cap I_{z}\right]\le c\sum_{k=1}^{L}e^{4L-2k}\sup_{\left\{ y,z\right\} \in G_{k}}P_{x}\left[I_{y}\cap I_{z}\right],\label{eq:SumWithSizeOfGk}
\end{equation}
where have used that for $1\le k\le L$
\begin{equation}
\begin{array}{ccccl}
\left|G_{k}\right| & \le & \left|\tilde{F}_{L}\right|{\displaystyle \sup_{v\in\tilde{F}_{L}}}\left|\tilde{F}_{L}\cap B\left(v,2r_{k-1}\right)\right| & \overset{\eqref{eq:GridDefinition}}{\le} & \left|\tilde{F}_{L}\right|\left(c\left|\tilde{F}_{L}\right|r_{k-1}^{2}\right)\\
 & \overset{\eqref{eq:DefOfRadii}}{=} & c\left|\tilde{F}_{L}\right|^{2}e^{-\frac{3}{2}\log\log L}e^{-2k} & \overset{\eqref{eq:SizeOfSetOfFeps}}{\le} & c\left(\log L\right)^{3/2}e^{4L-2k}.
\end{array}\label{eq:SizeOfJk}
\end{equation}
Next we split the sum on the right-hand side of \prettyref{eq:SumWithSizeOfGk}
into three parts
\begin{equation}
\sum_{k=1}^{L}\cdot\le\sum_{1\le k\le l_{0}}\cdot+\sum_{l_{0}<k<L-l_{0}}\cdot+\sum_{L-l_{0}\le k\le L}\cdot.\label{eq:ThreewayDecomp}
\end{equation}
For the first sum on the right-hand side we have by \prettyref{lem:TwoProfileBoundEdgeCase}
that it is at most, 
\begin{equation}
\begin{array}{l}
{\displaystyle c{\displaystyle \left(\log L\right)^{3/2}\sum_{1\le k\le l_{0}}}}e^{4L-2k}\left(e^{-\left(2L-2l_{0}\right)}L^{s}l_{0}^{0.51}g\left(l_{0}\right)\right)^{2}\\
=c\left(\log L\right)^{3/2}L^{2s}l_{0}^{1.02}g\left(l_{0}\right)^{2}e^{4l_{0}}\sum_{k=1}^{l_{0}}e^{-2k}\overset{\eqref{eq:DefOfCutOff}}{\le}cL^{2s}\left(\log L\right)^{19/10}l_{0}^{1.02}g\left(l_{0}\right)^{2}.
\end{array}\label{eq:HighK}
\end{equation}
For the middle sum on the right-hand side of \prettyref{eq:ThreewayDecomp}
we use \prettyref{prop:MainTwoProfileBound} to obtain an upper bound
of 
\begin{equation}
\begin{array}{l}
c\left(s\right)\left(\log L\right)^{3/2}{\displaystyle \sum_{l_{0}<k<L-l_{0}}}e^{4L-2k}l_{0}^{1.02}g\left(k\right)^{2}\left(\log L\right)^{1.02}e^{-\left(4L-2k\right)-cf\left(k\right)}L^{2s}\\
=c\left(s\right)\left(\log L\right)^{2.52}l_{0}^{1.02}L^{2s}{\displaystyle \sum_{l_{0}<k<L-l_{0}}}g\left(k\right)^{2}e^{-cf\left(k\right)}\le c\left(s\right)\left(\log L\right)^{2.52}l_{0}^{1.02}L^{2s},
\end{array}\label{eq:MiddleK}
\end{equation}
since
\[
{\displaystyle \sum_{l_{0}<k<L-l_{0}}}g\left(k\right)^{2}e^{-cf\left(k\right)}\to0\mbox{ by }\eqref{eq:DefOffFunc},\eqref{eq:DefOfgfunc}\mbox{ and }\eqref{eq:DefOfCutOff}.
\]
For the last sum on the right-hand side of \prettyref{eq:ThreewayDecomp}
we obtain from \prettyref{eq:yzfarawaytrivial} the following upper
bound
\begin{equation}
\begin{array}{ccl}
c{\displaystyle \sum_{L-l_{0}\le k\le L}}e^{4L-2k}e^{-2L}l_{0}L^{s} & = & c\left(\log L\right)^{3/2}l_{0}L^{s}{\displaystyle \sum_{L-l_{0}\le k\le L}}e^{2L-2k}\\
 & = & c\left(\log L\right)^{3/2}l_{0}L^{s}{\displaystyle \sum_{0\le k^{'}\le l_{0}}}e^{2k^{'}}\\
 & \le & c\left(\log L\right)^{3/2}l_{0}L^{s}e^{2l_{0}}\overset{\eqref{eq:DefOfCutOff}}{=}c\left(\log L\right)^{17/10}l_{0}L^{s}.
\end{array}\label{eq:LowK}
\end{equation}
Combining \eqref{eq:HighK}-\eqref{eq:LowK} we thus obtain this upper
bound on the right-hand side of \eqref{eq:SumWithSizeOfGk}:
\[
\begin{array}{cl}
 & cL^{2s}\left(\log L\right)^{19/10}l_{0}^{1.02}g\left(l_{0}\right)^{2}+c\left(s\right)\left(\log L\right)^{2.52}l_{0}^{1.02}L^{2s}+c\left(\log L\right)^{17/10}l_{0}L^{s}\\
 & \le c\left(s\right)L^{2s}\left(\log L\right)^{2.52}l_{0}^{1.02}g\left(l_{0}\right)^{2}\\
 & \overset{\eqref{eq:LowerBoundOneProfileExp}}{=}c\left(s\right)\left(E_{x}\left[Z\right]\right)^{2}\left(\log L\right)^{-0.48}l_{0}^{-0.98}g\left(l_{0}\right)^{2}\overset{\eqref{eq:DefOfgfunc},\eqref{eq:DefOfCutOff}}{=}o\left(E_{x}\left[Z\right]^{2}\right).
\end{array}
\]
This gives \eqref{eq:RestOfTerms}, so the claim \prettyref{eq:SecondMomentOfA}
follows.
\end{proof}
We have now reached the final goal of this section: the proof of \prettyref{prop:LowerBound}.
\begin{proof}[Proof of \prettyref{prop:LowerBound}]
By \prettyref{eq:ZPositiveMeansSomeTzero} it suffices to show that
$P_{x}\left[Z>0\right]\to1$, and by the Paley-Zygmund inequality
we have
\[
P_{x}\left[Z>0\right]\ge\frac{E_{x}\left[Z\right]^{2}}{E_{x}\left[Z^{2}\right]}.
\]
Thus the claim \eqref{eq:LowerBoundInRedSec} follows by \prettyref{prop:SecondMomentOfA}.
\end{proof}
Of three main ingredients (propositions \ref{prop:UpperBound}-\ref{prop:Concentration})
used to prove the main result \prettyref{thm:MainResult} we have
now derived most of the first two. Still missing are the proofs of
the barrier estimate \prettyref{lem:GWBarrierAlpha} and the large
deviation result \prettyref{lem:LDForBinSumOfGeo} (used but not proven
in \prettyref{sec:UpperBound} for the proof of \prettyref{prop:UpperBound}),
the barrier estimates \prettyref{lem:GaltonWatsonBarrierBoundsSecLB}
used in this section to prove \prettyref{prop:LowerBound}, and the
two point probability bounds \prettyref{lem:TwoProfileBoundEdgeCase}
and \prettyref{prop:MainTwoProfileBound} also used in this section.
The next section deals with these two point probability estimates.

\section{\label{sec:TwoPointBound}Bounds on two point probabilities}

In this section we will prove the crucial two point probabilty bounds
\prettyref{lem:TwoProfileBoundEdgeCase} and \prettyref{prop:MainTwoProfileBound},
which were used to prove the lower bound \prettyref{prop:LowerBound}
in the previous section. Recall these give a bounds on the probability
$P_{x}\left[I_{y}\cap I_{z}\right]$ where for $v\in\mathbb{T}$
\[
I_{v}=\left\{ \gamma\left(l\right)\le\sqrt{T_{l}^{y,t_{-s}}}\le\delta\left(l\right)\mbox{ for }l=l_{0},\ldots,L-l_{0}\mbox{ and }T_{L-1}^{y,t_{-s}}=0\right\} \tag{\ref{eq:TruncatedSummandLB}'}.
\]

We will need to consider certain traversal processes that ``start
at lower scales''. For each $k\ge1$ we define
\begin{equation}
T_{l}^{y,k,m}=\sup\left\{ n\ge0:R_{n}^{y,l}\le D_{m}^{y,k}\right\} ,l\ge k,m\in\mathbb{R}_{+},\label{eq:TraversalsDefStartingFromk}
\end{equation}
to be the number of traversals from scale $l$ to $l+1$ during the
first $t$ excursions from scale $k-1$ to scale $k$ (cf. the definition
\prettyref{eq:TraversalsDef} of $T_{l}^{y,t}$). The definitions
\prettyref{eq:TraversalsDef} and \prettyref{eq:TraversalsDefStartingFromk}
imply the crucial ``compatability'' property that
\begin{equation}
T_{l}^{y,k,m}=T_{l}^{y,t}\mbox{ for }l\ge k,\mbox{\,\ on }\left\{ m=T_{k}^{y,t}\right\} ,\label{eq:TkandTCompatability}
\end{equation}
since on the latter event $W_{t}$ does not visit $B\left(y,r_{k}\right)$
between $D_{m}^{y,k}$ and $D_{t}^{y,0}$. Furthermore, the process
$T_{l}^{y,k,t}$ satisfies essentially the same properties as $T_{l}^{y,t}$.
In particular:
\begin{lem}
\label{lem:LawOfTyt}If $y\in\mathbb{T}$, $k\ge1$, $v\notin B\left(y,r_{k}\right)^{\circ}$
and $t\ge0$, the $P_{v}-$law of $\left(T_{k+l}^{y,k,t}\right)_{l\ge0}$
is $\mathbb{G}_{t}$ .\end{lem}
\begin{proof}
Almost identical to the proof of \prettyref{lem:BranchingProc}.
\end{proof}
The $T_{l}^{y,k,t}$ also satisfy a similar independence property.
\begin{lem}
\label{lem:IndepOfDisjointTk}($x\in\mathbb{T}$) For all $t\ge0$
and $y,z\in\tilde{F}_{L}$ (see \eqref{eq:FTildeDef}) it holds that
\begin{equation}
\mbox{if }d\left(y,z\right)>2r_{k}\mbox{ then }\left(T_{l}^{y,k,t}\right)_{l\ge k}\mbox{ and }\left(T_{l}^{z,k,t}\right)_{l\ge k}\mbox{ are independent under }P_{x}.\label{eq:IndepOfDisjointTk}
\end{equation}
\end{lem}
\begin{proof}
Almost identical to the proof of \prettyref{eq:IndepOfDisjoint}.
\end{proof}
We will need the following barrier crossing estimates for the Galton-Watson
process $T_{l}$, which will be proven (together with the previously
used barrier estimates \prettyref{lem:GWBarrierAlpha} and \prettyref{lem:GaltonWatsonBarrierBoundsSecLB})
in \prettyref{sec:BoundaryCrossingProofs}. The first corresponds
to checking the barrier for $l\ge k$, and the second to checking
it for $l\le k$.
\begin{lem}
\label{lem:GWBarrierGammaDeltaUpTok}For all $l_{0}<k<L-l_{0}-1$,
$s\in\left(-1,1\right)$, and $\gamma\left(k\right)^{2}\le a\le\delta\left(k\right)^{2}$
we have \textup{
\begin{equation}
\mathbb{G}_{a}\left[\gamma\left(l\right)\le\sqrt{T_{l-k}}\mbox{ for }l=k,\ldots,L-l_{0}|T_{L-1-k}=0\right]\le c\frac{l_{0}^{0.51}g\left(k\right)}{L-k-l_{0}-1}.\label{eq:GWBarrierGammaDeltaFromk}
\end{equation}
}If $l_{0}+1<k<L-l_{0}$ and $s\in\left(-1,1\right)$
\begin{equation}
\mathbb{G}_{t_{s}}\left[\gamma\left(l\right)\le\sqrt{T_{l}}\le\delta\left(l\right)\mbox{ for }l=l_{0},\ldots,k|T_{L-1}=0\right]\le\frac{c\sqrt{l_{0}}g\left(k+1\right)}{k-l_{0}-1}.\label{eq:GWBarrierGammmaDeltaUpTok}
\end{equation}

\end{lem}
We are now ready to prove \prettyref{lem:TwoProfileBoundEdgeCase}
from the previous section, which gives the bound \prettyref{eq:TwoProfileBoundEdgeCase}
on the two point probability $P_{x}\left[I_{y}\cap I_{z}\right]$
for $y$ and $z$ such that $B\left(y,r_{k}\right)$ and $B\left(z,r_{k}\right)$
for $0\le k\le l_{0}$ are the largest non-overlapping balls around
$y$ and $z$.
\begin{proof}[Proof of \prettyref{lem:TwoProfileBoundEdgeCase}]
By \prettyref{eq:TruncatedSummandLB} and \prettyref{eq:TkandTCompatability}
with $k=l_{0}+1$ we have for $v\in\left\{ y,z\right\} $,
\[
I_{v}\subset\tilde{I}_{v}\overset{\mbox{def}}{=}\left\{ \begin{array}{c}
\gamma\left(l\right)\le\sqrt{T_{l}^{v,l_{0}+1,m}}\mbox{ for }l=l_{0}+1,\ldots,L-l_{0}\mbox{ and }T_{L-1}^{v,l_{0}+1,m}=0\\
\mbox{for some }\gamma^{2}\left(l_{0}+1\right)\le m\le\delta^{2}\left(l_{0}+1\right)
\end{array}\right\} .
\]
Thus by \prettyref{lem:IndepOfDisjointTk} with $k=l_{0}+1$ (recall
that $d\left(y,z\right)\ge2r_{l_{0}}$)
\[
P_{x}\left[I_{y}\cap I_{z}\right]\le P_{x}\left[\tilde{I}_{y}\cap\tilde{I}_{z}\right]\le\left(P_{x}\left[\tilde{I}_{y}\right]\right)^{2}.
\]
Thus to get \prettyref{eq:TwoProfileBoundEdgeCase} it suffices to
show
\begin{equation}
P_{x}\left[\tilde{I}_{y}\right]\le ce^{-\left(2L-2l_{0}\right)}L^{s}\sqrt{l_{0}}g\left(l_{0}\right).\label{eq:SufficesToShowASDASD}
\end{equation}
Now let
\[
\tilde{I}_{y,m}=\left\{ \gamma\left(l\right)\le\sqrt{T_{l}^{y,l_{0}+1,m}}\mbox{ for }l=l_{0}+1,\ldots,L-l_{0}\mbox{ and }T_{L-1}^{y,l_{0}+1,m}=0\right\} ,
\]
so that $\tilde{I}_{y}=\cup_{\gamma^{2}\left(l_{0}+1\right)\le m\le\delta^{2}\left(l_{0}+1\right)}\tilde{I}_{y,m}$.
If $\tilde{I}_{y,m}$ holds and $T_{L-1}^{y,l_{0}+1,m+1}=0$, then
also $\tilde{I}_{y,m+1}$ holds. Using this we obtain that 
\begin{equation}
\tilde{I}_{y}\subset\left(\cup_{\gamma\left(l_{0}+1\right)^{2}\le m<\delta\left(l_{0}+1\right)^{2}}\tilde{I}_{y,m}\cap\left\{ T_{L-1}^{y,l_{0}+1,m+1}>0\right\} \right)\cup\tilde{I}_{y,\delta^{2}\left(l_{0}+1\right)},\label{eq:EdgeCaseDecomp}
\end{equation}
For $y\notin B\left(x,r_{0}\right)$ we have by \prettyref{eq:TraversalsDefStartingFromk}
that 
\[
\begin{array}{ccl}
\left\{ T_{L-1}^{y,l_{0}+1,m}=0\mbox{ and }T_{L-1}^{y,l_{0}+1,m+1}>0\right\}  & = & \left\{ D_{m}^{y,l_{0}+1}<H_{B\left(y,r_{L}\right)}<D_{m+1}^{y,l_{0}+1}\right\} \\
 & = & \left\{ H_{B\left(y,r_{L}\right)}<T_{B\left(y,r_{l_{0}+1}\right)}\right\} \circ\theta_{R_{m+1}^{y,l_{0}+1}},
\end{array}
\]
so that by the strong Markov property at time $R_{m+1}^{y,l_{0}+1}$
we have
\begin{equation}
\begin{array}{ccl}
P_{x}\left[\tilde{I}_{y,m}\cap\left\{ T_{L-1}^{y,l_{0}+1,m+1}>0\right\} \right] & = & P_{x}\left[\tilde{I}_{y,m}P_{W_{R_{m+1}^{y,l_{0}+1}}}\left[H_{B\left(y,r_{L}\right)}<T_{B\left(y,r_{l_{0}+1}\right)}\right]\right]\\
 & = & P_{x}\left[\tilde{I}_{y,m}\right]\frac{1}{L-l_{0}-1},
\end{array}\label{eq:EdgeCaseMarkov}
\end{equation}
where we have used \prettyref{eq:HitLProbStartingFromk} and that
$\tilde{I}_{y,m}$ is $\mathcal{F}_{D_{m}^{y,l_{0}+1}}-$measurable.
Thus
\[
P_{x}\left[\tilde{I}_{y}\right]\le\frac{1}{L-l_{0}-1}\sum_{\gamma\left(l_{0}+1\right)^{2}\le m<\delta\left(l_{0}+1\right)^{2}}P_{x}\left[\tilde{I}_{y,m}\right]+P_{x}\left[\tilde{I}_{y,\delta\left(l_{0}+1\right)^{2}}\right].
\]
Also $P_{x}\left[\tilde{I}_{y,m}\right]$ equals
\begin{equation}
q_{m}P_{x}\left[T_{L-1}^{y,l_{0}+1,m}=0\right],\label{eq:EdgeCaseCond}
\end{equation}
where we write
\[
q_{m}=P_{x}\left[\gamma\left(l\right)\le\sqrt{T_{l}^{y,l_{0}+1,m}}\mbox{ for }l=l_{0}+1,\ldots,L-l_{0}|T_{L-1}^{y,l_{0}+1,m}=0\right].
\]
Similary to in the proof of \prettyref{lem:NotHitByrL} we have using
\prettyref{eq:HitLProbStartingFromk},
\begin{equation}
P_{x}\left[T_{L-1}^{y,l_{0}+1,m}=0\right]=\left(1-\frac{1}{L-l_{0}-1}\right)^{m}.\label{eq:EdgeCaseNotHit}
\end{equation}
We thus find that 
\begin{equation}
\begin{array}{ccl}
P_{x}\left[\tilde{I}_{y}\right] & \le & \frac{1}{L-l_{0}-1}{\displaystyle \sum_{\gamma\left(l_{0}+1\right)^{2}\le m<\delta\left(l_{0}+1\right)^{2}}}\left(1-\frac{1}{L-l_{0}-1}\right)^{m}q_{m}+\left(1-\frac{1}{L-l_{0}-1}\right)^{\delta\left(l_{0}+1\right)^{2}}q_{\delta\left(l_{0}+1\right)^{2}}\\
 & \le & \left({\displaystyle \sup_{\gamma\left(l_{0}+1\right)^{2}\le m<\delta\left(l_{0}+1\right)^{2}}q_{m}}\right)\left(1-\frac{1}{L-l_{0}-1}\right)^{\gamma\left(l_{0}+1\right)^{2}},
\end{array}\label{eq:A}
\end{equation}
where we have summed a geometric series. Now
\[
\left(1-\frac{1}{L-l_{0}-1}\right)^{\gamma\left(l_{0}+1\right)^{2}}\le e^{-\frac{\gamma\left(l_{0}+1\right)^{2}}{L-l_{0}-1}}\overset{\eqref{eq:GammaBarrierDef}}{\le}e^{-\frac{\beta\left(l_{0}+1\right)^{2}}{L-l_{0}-1}}\overset{\eqref{eq:BetaDef}}{=}e^{-\frac{t_{-s}}{L}\frac{L-l_{0}-1}{L}}\overset{\eqref{eq:TsDivdedByL}}{\le}ce^{-2\left(L-l_{0}\right)}L^{1+s}.
\]
Also by \prettyref{lem:LawOfTyt}
\begin{equation}
q_{m}\le\sup_{\gamma\left(l_{0}+1\right)^{2}\le a\le\delta\left(l_{0}+1\right)^{2}}\mathbb{G}_{a}\left[\gamma\left(l\right)\le\sqrt{T_{l-l_{0}-1}^{y}}\mbox{ for }l=l_{0}+1,\ldots,L-l_{0}|T_{L-1-l_{0}}^{y}=0\right].\label{eq:B}
\end{equation}
Thus \prettyref{eq:GWBarrierGammaDeltaFromk} with $k=l_{0}+1$ gives
that
\begin{equation}
q_{m}\le c\frac{g\left(l_{0}+1\right)l_{0}^{0.51}}{L-2l_{0}-2}\overset{\eqref{eq:DefOfgfunc},\eqref{eq:DefOfCutOff}}{\le}c\frac{g\left(l_{0}\right)l_{0}^{0.51}}{L}.\label{eq:C}
\end{equation}
Combinig \eqref{eq:A}, \eqref{eq:B} and \eqref{eq:C} we obtain
that
\[
P_{x}\left[\tilde{I}_{y}\right]\le ce^{-2\left(L-l_{0}\right)}L^{1+s}\times c\frac{l_{0}^{0.51}g\left(l_{0}\right)}{L},
\]
which is equivalent to \prettyref{eq:SufficesToShowASDASD}, so the
proof of \prettyref{lem:TwoProfileBoundEdgeCase} is complete.
\end{proof}
We now move to the more difficult bound, namely \prettyref{prop:MainTwoProfileBound},
which deals with $y$ and $z$ whose largest non-overlapping balls
has radius $r_{k}$ for $l_{0}<k<L-l_{0}$. More precisely, \prettyref{prop:MainTwoProfileBound}
claims that for any $s\in\left(-1,1\right)$ and $y,z\in\tilde{F}_{L}$
such that
\[
2r_{k}<d\left(y,z\right)\le2r_{k-1}\mbox{ for }l_{0}<k<L-l_{0},
\]
we have 
\[
P_{x}\left[I_{y}\cap I_{z}\right]\le ce^{-\left(4L-2k\right)-cf\left(k\right)}L^{2s}l_{0}^{1.02}g\left(k\right)^{2}\left(\log L\right)^{1.02}\tag{\ref{eq:TwoProfileBoundInStatement}'}.
\]
In the remainder of this section we consider $y,z,k$ and $s$ to
be fixed. Since $2r_{k-1}+r_{k}\le r_{k-2}$ (see \prettyref{eq:DefOfRadii})
we have
\begin{equation}
B\left(z,r_{k}\right)\subset B\left(y,r_{k-2}\right)\backslash B\left(y,r_{k}\right)\mbox{ and }B\left(y,r_{k}\right)\subset B\left(z,r_{k-2}\right)\backslash B\left(z,r_{k}\right),\label{eq:BranchingPoint}
\end{equation}
\begin{figure}
\includegraphics[scale=0.75]{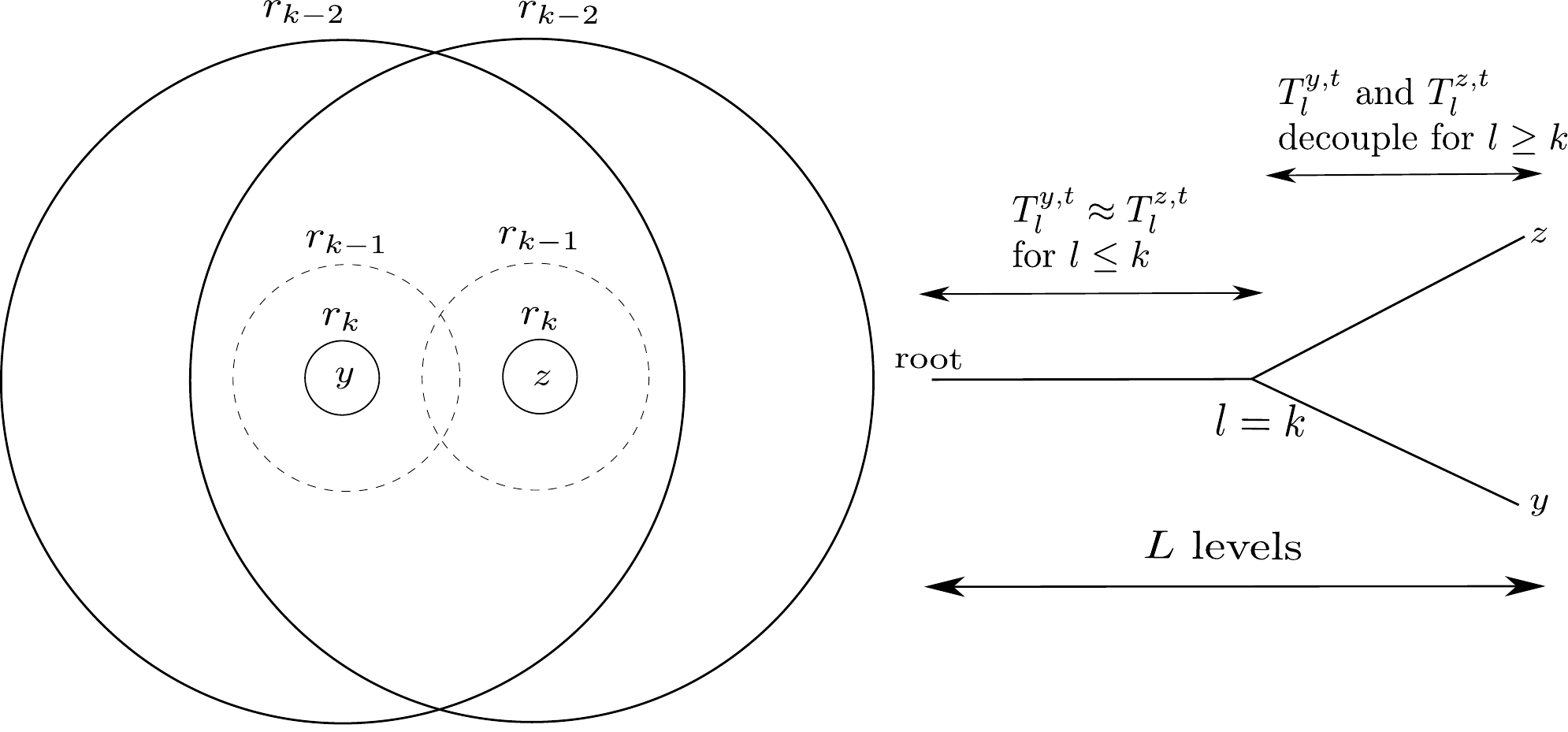}\caption{\label{fig:PosOfCirclesTWB}(Left) The position of $\partial B\left(v,r_{l}\right)$
for $v\in\left\{ y,z\right\} $ and $l\in\left\{ k-2,k-1,k\right\} $
assumed in \prettyref{prop:MainTwoProfileBound}, cf \prettyref{eq:BranchingPoint}.
(Right) An intuitive illustration of the pseudo-hierarchical structure
which underlies \prettyref{prop:MainTwoProfileBound}, for $y$ and
$z$ at distance roughly $r_{k}$.}
\end{figure}
(see \prettyref{fig:PosOfCirclesTWB}) and by the definition \prettyref{eq:FTildeDef}
of $\tilde{F}_{L}$ 
\[
x\notin B\left(y,r_{0}\right)\cup B\left(z,r_{0}\right).
\]
We will consider separately the cases 
\[
k\le\left(1-\frac{s}{10}\right)L\mbox{ and }k\ge\left(1-\frac{s}{10}\right)L.
\]

\subsection{Main bound: late branching}

Here we consider the case $\left(1-\frac{s}{10}\right)L\le k<L-l_{0}$.
It turns out that for this regime we can ignore the contribution from
the barrier condition on $T_{l}^{y,t_{-s}}$ and $T_{l}^{z,t_{-s}}$
for $l\ge k-2$ and still get a good enough bound. Therefore we let
\begin{equation}
J_{y}^{\uparrow}=\left\{ \gamma\left(l\right)\le\sqrt{T_{l}^{y,t_{-s}}}\le\delta\left(l\right)\mbox{ for }l=l_{0},\ldots,k-3\right\} ,\label{eq:JUp}
\end{equation}
denote the barrier condition applied only up to $k-3$. We will bound
the probability of
\begin{equation}
J_{y}^{\uparrow}\cap\left\{ H_{B\left(y,r_{L}\right)}\ge D_{t_{-s}}^{y,0}\right\} \cap\left\{ H_{B\left(z,r_{L}\right)}\ge D_{\gamma\left(k\right)^{2}}^{z,k}\right\} ,\label{eq:LargeKBiggerEvent}
\end{equation}
(which we will see contains the event $I_{y}\cap I_{z}$). We first
bound the contribution from the part of \prettyref{eq:LargeKBiggerEvent}
referring to $y$.
\begin{lem}
\label{lem:HighYsAlone}For all $\left(1-\frac{s}{10}\right)L\le k<L-l_{0}$,
\begin{equation}
P_{x}\left[J_{y}^{\uparrow}\cap\left\{ H_{B\left(y,r_{L}\right)}\ge D_{t_{-s}}^{y,0}\right\} \right]\le ce^{-2L}L^{s}\sqrt{l_{0}}g\left(k-2\right).\label{eq:HighYsAlone}
\end{equation}
\end{lem}
\begin{proof}
Since $\left\{ H_{B\left(y,r_{L}\right)}\ge D_{t_{-s}}^{y,0}\right\} =\left\{ T_{L-1}^{y,t_{-s}}=0\right\} $
(recall \eqref{eq:ExtinctionImpliesNotHit})  the probability in \eqref{eq:HighYsAlone}
is
\[
P_{x}\left[T_{L-1}^{y,t_{-s}}=0\right]P_{x}\left[J_{y}^{\uparrow}|T_{L-1}^{y,t_{-s}}=0\right].
\]
By \prettyref{lem:NotHitByrL} the first of these is at most $ce^{-2L}L^{1+s}$.
By \prettyref{lem:BranchingProc} the second equals
\[
\mathbb{G}_{t_{-s}}\left[\gamma\left(l\right)\le\sqrt{T_{l}}\le\delta\left(l\right)\mbox{ for }l=l_{0},\ldots,k-3|T_{L-1}=0\right],
\]
and is thus bounded by $c\sqrt{l_{0}}g\left(k-3\right)/\left(k-4-l_{0}\right)$,
by \prettyref{eq:GWBarrierGammmaDeltaUpTok} with $k-3$ in place
of $k$. Since $k\ge cL$ this gives the claim.
\end{proof}
It remains to bound the contribution from the part of \prettyref{eq:LargeKBiggerEvent}
referring to $z$. This should be roughly independent of the part
referring to $y$. To make this decoupling rigorous we must bound
the probability of $\left\{ H_{B\left(z,r_{L}\right)}\ge D_{\gamma\left(k\right)^{2}}^{z,k}\right\} $
conditioned on avoiding $B\left(y,r_{L}\right)$. The next lemma is
a first step in this direction, and bounds the conditional probability
of hitting $B\left(z,r_{L}\right)$ from $\partial B\left(z,r_{k+1}\right)$
before escaping to to $\partial B\left(z,r_{k}\right)$.
\begin{lem}
\label{lem:ConditionedHittingProb}For any $v\in\partial B\left(z,r_{k+1}\right)$
\begin{equation}
P_{v}\left[H_{B\left(z,r_{L}\right)}<T_{B\left(z,r_{k}\right)}|H_{B\left(y,r_{L}\right)}>T_{B\left(y,r_{k-2}\right)}\right]\ge P_{v}\left[H_{B\left(z,r_{L}\right)}<T_{B\left(z,r_{k}\right)}\right]\left(1-c\frac{1}{L-k}\right).\label{eq:HitZBound}
\end{equation}
\end{lem}
\begin{proof}
The right-hand is bounded below by
\[
P_{v}\left[H_{B\left(z,r_{L}\right)}<T_{B\left(z,r_{k}\right)},H_{B\left(y,r_{L}\right)}>T_{B\left(y,r_{k-2}\right)}\right].
\]
By the strong Markov property this equals
\[
P_{v}\left[H_{B\left(z,r_{L}\right)}<T_{B\left(z,r_{k}\right)},P_{T_{B\left(z,r_{k}\right)}}\left[H_{B\left(y,r_{L}\right)}>T_{B\left(y,r_{k-2}\right)}\right]\right].
\]
Since $\partial B\left(z,r_{k}\right)\subset A=B\left(y,r_{k-2}\right)\backslash B\left(r_{k}\right)$
we have that
\[
\begin{array}{ccl}
P_{T_{B\left(z,r_{k}\right)}}\left[H_{B\left(y,r_{L}\right)}>T_{B\left(y,r_{k-2}\right)}\right] & \ge & \inf_{v\in A}P_{v}\left[H_{B\left(y,r_{L}\right)}>T_{B\left(y,r_{k-2}\right)}\right]\\
 & = & P_{w}\left[H_{B\left(y,r_{L}\right)}>T_{B\left(y,r_{k-2}\right)}\right],
\end{array}
\]
for an arbitrary $w\in\partial B\left(y,r_{k}\right)$. Now \prettyref{eq:HitZBound}
follows since by \eqref{eq:HitLProbStartingFromk} the latter probability
is 
\[
\frac{L-k}{L-k+2}\ge1-\frac{c}{L-k}.
\]

\end{proof}
We now aim to ``decouple'' the event $J_{y}^{\uparrow}$ from the
part of \prettyref{eq:LargeKBiggerEvent} that referes to $z$. The
main tool for this is a recursion which we now describe. Let
\begin{equation}
J\mbox{ be an arbitrary }\left(T_{l}^{y,t_{-s}}\right)_{l\in\left\{ 0,\ldots,k-3\right\} }-\mbox{measurable event},\label{eq:WhichJs}
\end{equation}
(here we will apply it with $J=J_{y}^{\uparrow}$, but later we will
use also $J=C\left(\mathbb{R}_{+},\mathbb{T}\right)$), and
\begin{equation}
A_{n}=J\cap\left\{ H_{B\left(y,r_{L}\right)}\ge D_{t_{-s}}^{y,0}\right\} \cap\left\{ H_{B\left(z,r_{L}\right)}\ge D_{n}^{z,k}\right\} ,n\ge0.\label{eq:DefOfAn}
\end{equation}
We have the following bound, which ``extracts'' the cost of an excursion
from scale $k+1$ to $k$ avoiding $B\left(z,r_{L}\right)$, one at
a time. The idea is that whether an excursion hits $B\left(z,r_{L}\right)$
or not can only affect the event $J$ through the end point of the
excursion. But we will use \prettyref{eq:IndepOfStartingPoint} to
show that the end point does not affect $J$. Furthermore, we will
use \prettyref{lem:ConditionedHittingProb} to show that the cost
of avoiding $B\left(z,r_{L}\right)$ when conditioned on $\left\{ H_{B\left(y,r_{L}\right)}\ge D_{t_{-s}}^{y,0}\right\} $
is almost the same as the unconditioned cost.
\begin{lem}
\label{lem:RecursionStep}For all $n\ge1$
\begin{equation}
\begin{array}{l}
P_{x}\left[A_{n}\right]\le\left(1-\frac{1}{L-k}\left(1-\frac{c}{L-k}\right)\right)P_{x}\left[A_{n-1}\right].\end{array}\label{eq:RecursionStep}
\end{equation}
\end{lem}
\begin{proof}
Let
\begin{eqnarray}
 & B=1_{\left\{ H_{B\left(z,r_{L}\right)}\ge D_{n-1}^{z,k}\right\} },\mbox{ and let},\label{eq:BIndicatorFunc}\\
 & S=T_{B\left(y,k-2\right)}\circ\theta_{R_{n}^{z,k}}+R_{n}^{z,k},\label{eq:DefOfSLastTime}
\end{eqnarray}
be first time after $R_{n}^{z,k}$ that $W_{t}$ leaves $B\left(y,k-2\right)$.
By the assumption \prettyref{eq:WhichJs} the event $J$ only depends
on $T_{l}^{y,t_{-s}}$ for $l\le k-3$, which depend only ``on what
$W_{t}$ does in $B\left(y,r_{k-2}\right)^{c}$''. Therefore $J$
is measurable with respect to $W_{\cdot\wedge R_{n}^{z,k}}$ and $T_{l}^{y,t}\left(W_{S+\cdot}\right),t\ge0,l\ge0$
(where $T_{l}^{y,t_{-s}}\left(W_{S+\cdot}\right)$ counts the traversals
that take place after time $S$). Therefore there exists a measurable
function $f$ such that
\begin{equation}
1_{J}=f\left(W_{\cdot\wedge R_{n}^{z,k}},\left(T_{l}^{y,t}\left(W_{S+\cdot}\right)\right)_{t\ge0,l\ge0}\right).\label{eq:JIndicatorFunc}
\end{equation}
The event $\left\{ H_{B\left(y,r_{L}\right)}\ge D_{t_{-s}}^{y,0}\right\} $
depends only on the same random variables together with 
\begin{equation}
C=1_{\left\{ D_{t_{-s}}^{y,0}\le R_{n}^{z,k}\right\} }+1_{\left\{ D_{t_{-s}}^{y,0}\ge R_{n}^{z,k}\right\} \cap\left\{ H_{B\left(y,r_{L}\right)}\circ\theta_{R_{n}^{z,k}}+R_{n}^{z,k}\ge S\right\} },\label{eq:CIndicatorFunc}
\end{equation}
which gives encodes the dependence on $W_{(R_{n}^{z,k}+\cdot)\wedge S}$
(if $t_{-s}$ of $y's$ excursions from scale $1$ to $0$ have not
been completed by time $R_{n}^{z,k}$ then $W_{t}$ needs to avoid
$B\left(y,r_{L}\right)$ between $R_{n}^{z,k}$ and $S$). Thus there
is a function $g$ such that
\begin{equation}
1_{\left\{ H_{B\left(y,r_{L}\right)}\ge D_{t_{-s}}^{y,0}\right\} }=g\left(W_{\cdot\wedge R_{n}^{z,k}},\left(T_{l}^{y,t}\left(W_{S+\cdot}\right)\right)_{t\ge0,l\ge0}\right)C.\label{eq:gIndicatorFunc}
\end{equation}
Letting $h=fg$ we have
\begin{equation}
1_{A_{n-1}}=h\left(W_{\cdot\wedge R_{n}^{z,k}},\left(T_{l}^{y,t}\left(W_{S+\cdot}\right)\right)_{t\ge0,l\ge0}\right)BC.\label{eq:Anminus1indicatorfuncs}
\end{equation}
Furthermore 
\begin{equation}
1_{A_{n}}=h\left(W_{\cdot\wedge R_{n}^{z,k}},\left(T_{l}^{y,t}\left(W_{S+\cdot}\right)\right)_{t\ge0,l\ge0}\right)BCD,\mbox{ where,}\label{eq:Anindicatorfuncs}
\end{equation}
\begin{equation}
D=1_{\left\{ H_{B\left(z,r_{L}\right)}\circ\theta_{R_{n}^{z,k}}>T_{B\left(z,r_{k}\right)}\circ\theta_{R_{n}^{z,k}}\right\} }.\label{eq:DIndicatorFunc}
\end{equation}
Now by \prettyref{eq:IndepOfStartingPoint} and the strong Markov
property applied at time $S$, the collection $\left(T_{l}^{y,t}\left(W_{S+\cdot}\right)\right)_{t\ge0,l\ge0}$
is independent of $W_{\cdot\wedge S}$, since $W_{S}\in\partial B\left(y,r_{k-2}\right)$.
Thus letting
\[
\bar{h}\left(w_{\cdot}\right)=E_{v}\left[h\left(w_{\cdot},\left(T_{l}^{y,t}\right)_{t\ge0,l\ge0}\right)\right]\mbox{ for }w_{\cdot}\in C\left(\mathbb{R}_{+},\mathbb{T}\right),
\]
for some arbitrary $v\in\partial B\left(y,r_{k-2}\right)$, we have
from \prettyref{eq:Anminus1indicatorfuncs}
\begin{equation}
P_{x}\left[A_{n-1}\right]=E_{x}\left[\bar{h}\left(W_{\cdot\wedge R_{n}^{z,k}}\right)BC\right],\label{eq:Anminus1eq}
\end{equation}
and from \prettyref{eq:Anindicatorfuncs}
\begin{equation}
P_{x}\left[A_{n}\right]=E_{x}\left[\bar{h}\left(W_{\cdot\wedge R_{n}^{z,k}}\right)BCD\right].\label{eq:Aneq}
\end{equation}
Using the strong Markov property, \eqref{eq:DefOfSLastTime}, \eqref{eq:CIndicatorFunc}
and \eqref{eq:DIndicatorFunc},
\[
\begin{array}{l}
E_{x}\left[CD|\mathcal{F}_{R_{n}^{z,k}}\right]=1_{\left\{ D_{t_{-s}}^{y,0}\le R_{n}^{z,k}\right\} }P_{W_{R_{n}^{z,k}}}\left[H_{B\left(z,r_{L}\right)}>T_{B\left(z,r_{k}\right)}\right]\\
\quad+1_{\left\{ D_{t_{-s}}^{y,0}\ge R_{n}^{z,k}\right\} }P_{W_{R_{n}^{z,k}}}\left[H_{B\left(z,r_{L}\right)}>T_{B\left(z,r_{k}\right)},H_{B\left(y,r_{L}\right)}>T_{B\left(y,r_{k-2}\right)}\right].
\end{array}
\]
We have that $P_{W_{R_{n}^{z,k}}}\left[H_{B\left(z,r_{L}\right)}>T_{B\left(z,r_{k}\right)}\right]=\frac{1}{L-k}$
(recall \prettyref{eq:HitLProbStartingFromk} and $W_{R_{n}^{z,k}}\in\partial B(z,r_{k+1})$)
so that by \prettyref{lem:ConditionedHittingProb}
\[
\begin{array}{l}
P_{W_{R_{n}^{z,k}}}\left[H_{B\left(z,r_{L}\right)}>T_{B\left(z,r_{k}\right)},H_{B\left(y,r_{L}\right)}>T_{B\left(y,r_{k-2}\right)}\right]\\
\le\left(1-\frac{1}{L-k}\left(1-\frac{c}{L-k}\right)\right)P_{W_{R_{n}^{z,k}}}\left[H_{B\left(y,r_{L}\right)}>T_{B\left(y,r_{k-2}\right)}\right].
\end{array}
\]
Using this and the strong Markov property ``in reverse'' we have
\[
\begin{array}{l}
E_{x}\left[CD|\mathcal{F}_{R_{n}^{z,k}}\right]\\
\le\left(1-\frac{1}{L-k}\left(1-\frac{c}{L-k}\right)\right)\left\{ 1_{\left\{ D_{t_{-s}}^{y,0}\le R_{n}^{z,k}\right\} }+1_{\left\{ D_{t_{-s}}^{y,0}\ge R_{n}^{z,k}\right\} }P_{W_{R_{n}^{z,k}}}\left[H_{B\left(y,r_{L}\right)}>T_{B\left(y,r_{k-2}\right)}\right]\right\} \\
\le\left(1-\frac{1}{L-k}\left(1-\frac{c}{L-k}\right)\right)E_{x}\left[C|\mathcal{F}_{R_{n}^{z,k}}\right].
\end{array}
\]
Using this with \prettyref{eq:Anminus1eq} and \prettyref{eq:Aneq}
yields \prettyref{eq:RecursionStep} (note that $B$ is $\mathcal{F}_{R_{n}^{z,k}}$-measurable).
\end{proof}
The above lemma gives the following corollary which fully ``extracts''
the cost of the part of \prettyref{eq:LargeKBiggerEvent} referring
to $z$.
\begin{cor}
\label{cor:RecursionImplemented}For any event $J$ as in \prettyref{eq:WhichJs}
we have that
\begin{equation}
\begin{array}{l}
P_{x}\left[J\cap\left\{ H_{B\left(y,r_{L}\right)}\ge D_{t_{-s}}^{y,0}\right\} \cap\left\{ H_{B\left(z,r_{L}\right)}\ge D_{\gamma\left(k\right)^{2}}^{z,k}\right\} \right]\\
\le ce^{-2\left(L-k\right)-cf\left(k\right)}L^{\left(1+s\right)\left(1-\frac{k}{L}\right)}P_{x}\left[J\cap\left\{ H_{B\left(y,r_{L}\right)}\ge D_{t_{-s}}^{y,0}\right\} \right].
\end{array}\label{eq:RecursionImplemented}
\end{equation}
\end{cor}
\begin{proof}
The probability on the left hand-side is $P_{x}\left[A_{\lfloor\gamma\left(k\right)^{2}\rfloor}\right]$.
Applying \prettyref{lem:RecursionStep} recursively we have
\begin{equation}
P_{x}\left[A_{\lfloor\gamma\left(k\right)^{2}\rfloor}\right]\le\left(1-\frac{1}{L-k}\left(1-\frac{c}{L-k}\right)\right)^{\lfloor\gamma\left(k\right)^{2}\rfloor}P_{x}\left[J\cap\left\{ H_{B\left(y,r_{L}\right)}\ge D_{t_{-s}}^{y,0}\right\} \right],\label{eq:RecrusionAn}
\end{equation}
since $A_{0}=J\cap\left\{ H_{B\left(y,r_{L}\right)}\ge D_{t_{-s}}^{y,0}\right\} $.
Now 
\begin{equation}
\left(1-\frac{1}{L-k}\left(1-\frac{c}{L-k}\right)\right)^{\lfloor\gamma\left(k\right)^{2}\rfloor}\le ce^{-\frac{\gamma\left(k\right)^{2}}{L-k}\left(1-\frac{c}{L-k}\right)}\le ce^{-\frac{\gamma\left(k\right)^{2}}{L-k}},\label{eq:ETT}
\end{equation}
since $\gamma\left(k\right)^{2}\le2\beta\left(k\right)^{2}=2t_{-s}\left(1-k/L\right)^{2}\le4L^{2}\left(1-k/L\right)^{2}\le4\left(L-k\right)^{2}$
(recall \prettyref{eq:defofts}, \prettyref{eq:BetaDef} and \prettyref{eq:GammaBarrierDef}).
Also 
\[
\frac{t_{-s}}{L}\left(1-\frac{k}{L}\right)+cf\left(k\right)\le\frac{\gamma\left(k\right)^{2}}{L-k},
\]
so that since $e^{-t_{-s}/L}=e^{-2L}L^{1+s}$ (recall \prettyref{eq:TsDivdedByL})
\begin{equation}
e^{-\frac{\gamma\left(k\right)^{2}}{L-k}}\le ce^{-2\left(L-k\right)-cf\left(k\right)}L^{\left(1+s\right)\left(1-\frac{k}{L}\right)}.\label{eq:TVA}
\end{equation}
Using \prettyref{eq:ETT} and \prettyref{eq:TVA} in \prettyref{eq:RecrusionAn}
we obtain \prettyref{eq:RecursionImplemented}.
\end{proof}
We are now ready to prove the two point probability estimate for large
$k$ (we will see later that the event bounded below contains $I_{y}\cap I_{z}$).
Recall the definition \prettyref{eq:JUp} of $J_{y}^{\uparrow}$. 
\begin{prop}
\label{prop:LargeKCase}If $\left(1-\frac{s}{10}\right)L\le k<L-l_{0}$
then
\begin{equation}
\begin{array}{l}
P_{x}\left[J_{y}^{\uparrow}\cap\left\{ H_{B\left(y,r_{L}\right)}\ge D_{t_{-s}}^{y,0}\right\} \cap\left\{ H_{B\left(z,r_{L}\right)}\ge D_{\gamma\left(k\right)^{2}}^{z,k}\right\} \right]\\
\le ce^{-\left(4L-2k\right)-cf\left(k\right)}L^{2s}\sqrt{l_{0}}g\left(k-2\right).
\end{array}\label{eq:LargeKCase}
\end{equation}
\end{prop}
\begin{proof}
By \prettyref{cor:RecursionImplemented} with $J=J_{y}^{\uparrow}$
and \prettyref{lem:HighYsAlone} the probability in question is bounded
by
\[
ce^{-2\left(L-k\right)-cf\left(k\right)}L^{\left(1+s\right)\left(1-\frac{k}{L}\right)}\times ce^{-2L}L^{s}\sqrt{l_{0}}g\left(k-2\right).
\]
Thus \prettyref{eq:LargeKCase} follows since $\left(1+s\right)\left(1-\frac{k}{L}\right)\le s$
for $k\ge\left(1-\frac{s}{10}\right)L$ and $s\in\left(-1,1\right)$.
\end{proof}
We now turn to the bound for smaller $k$.

\subsection{Main bound: early branching}

Here we consider the case $l_{0}<k\le\left(1-\frac{s}{10}\right)L$.
It tuns out that in this regime we can ignore the contribution from
the barrier condition for $l\le k$. To deal with the condition for
$l\ge k$, we will need to decouple the contribution due to $y$ and
that due to $z$. To do this we will need to ``give ourselves a bit
of space'' , and we therefore define
\begin{equation}
k^{+}=k+\lceil100\log L\rceil,\label{eq:DefOfKPlus}
\end{equation}
and let for $v\in\left\{ y,z\right\} $ 
\begin{equation}
J_{v}^{\downarrow}=\left\{ \gamma\left(l\right)\le\sqrt{T_{l}^{y,t_{-s}}}\le\delta\left(l\right)\mbox{ for }l=k^{+},\ldots,L-l_{0}\right\} ,\label{eq:JDown}
\end{equation}
be the barrier conditioned applied only for $l\ge k^{+}$. To obtain
the two point bound for $k\le\left(1-\frac{s}{10}\right)L$ we will
bound the probability of
\begin{equation}
J_{y}^{\downarrow}\cap\left\{ H_{B\left(y,r_{L}\right)}\ge D_{t_{-s}}^{y,0}\right\} \cap J_{z}^{\downarrow}\cap\left\{ \gamma\left(k\right)\le\sqrt{T_{k}^{z,t_{-s}}},H_{B\left(z,r_{L}\right)}\ge D_{t_{-s}}^{z,0}\right\} ,\label{eq:SmallKBiggerEvent}
\end{equation}
which we will see contains the event $I_{y}\cap I_{z}$.  When bounding
$J_{v}^{\downarrow}$ for $v\in\left\{ y,z\right\} $ we will compare
the law of $T_{l}^{v,t_{-s}},l\ge k^{+},$ conditioned on the other
events of \prettyref{eq:SmallKBiggerEvent} to $\mathbb{G}_{a}$ for
$\gamma\left(l\right)^{2}\le a\le\delta\left(l\right)^{2}$, so that
we can apply the barrier crossing bound \prettyref{eq:GWBarrierGammaDeltaFromk}
for the law $\mathbb{G}_{a}$. As a first step in this direction we
let, recalling the definition \eqref{eq:DefOfExcursionTimes}, 
\begin{equation}
X_{\cdot}^{i}=X_{\cdot}^{i}\left(v\right)=W_{\left(R_{i}\left(v,r_{k},r_{k^{+}+1}\right)+\cdot\right)\wedge D_{i}\left(v,r_{k},r_{k^{+}+1}\right)},i=1,\ldots,\label{eq:DefOFXiExcursions}
\end{equation}
be the excursions of $W_{t}$ from $\partial B\left(v,r_{k^{+}+1}\right)$
to $\partial B\left(v,r_{k}\right)$. Let 
\begin{equation}
N=N\left(v\right)=\sup\left\{ n\ge1:D_{n}\left(v,r_{k},r_{k^{+}+1}\right)<D_{t_{-s}}^{v,0}\right\} ,\label{eq:NumberOfExcursions}
\end{equation}
be the number of excursions $X_{\cdot}^{i}$ that take place before
time $D_{t_{-s}}^{v,0}$. Note that 
\begin{equation}
\sum_{i=1}^{N}T_{l}^{v,\infty}\left(X_{\cdot}^{i}\right)=T_{l}^{v,t_{-s}}\mbox{\,\ for }l\ge k^{+},\label{eq:XExcursionsGiveTproc}
\end{equation}
where $T_{l}^{v,\infty}\left(X_{\cdot}^{i}\right)$ counts traversals
that take place during the excursion $X_{\cdot}^{i}$. Let
\[
J_{v,n}^{\downarrow}=\left\{ \gamma\left(l\right)\le\sqrt{\sum_{i=1}^{n}T_{l}^{v,\infty}\left(X_{\cdot}^{i}\right)}\le\delta\left(l\right)\mbox{\,\ for }l=k^{+},\ldots,L-l_{0}\right\} ,n\ge0,
\]
and note that by \eqref{eq:XExcursionsGiveTproc} and \eqref{eq:JDown}
\begin{equation}
1_{J_{v}^{\downarrow}}=1_{J_{v,N}^{\downarrow}}.\label{eq:JvNequalsJv}
\end{equation}
We thus aim to bound $P_{x}\left[J_{v,n}^{\downarrow}\right]$ and
are therefore interested in the law of $\sum_{i=1}^{n}T_{l}^{v,\infty}\left(X_{\cdot}^{i}\right)$.
This will be given by a Galton-Watson process with immigration: let
$\tilde{\mathbb{G}}_{n}$ denote the law such that $\left(T_{l}\right)_{\ge0}$
is a critical branching process with $T_{k-1}=0$ and immigration
of $n$ individuals in generations $k,k+1,\ldots,k^{+}$. That is,
\begin{equation}
\mbox{let }\tilde{\mathbb{G}}_{n}\mbox{ be the law of }\left(\sum_{p=k}^{k^{+}}T_{l}^{p}\right)_{l\ge0},\label{eq:DefOfGtilde}
\end{equation}
where $T_{k+\cdot}^{k},T_{k+1+\cdot}^{k+1},\ldots,T_{k^{+}+\cdot}^{k^{+}}$
are iid with law $\mathbb{G}_{n}$, and where we set $T_{l}^{p}=0$
for $l<p$.

To show that $\sum_{i=1}^{n}T_{l}^{v,\infty}\left(X_{\cdot}^{i}\right)$
has this law the first step is the following lemma giving the law
of an individual $T_{l}^{v,\infty}\left(X_{\cdot}^{i}\right)$.
\begin{lem}
\label{lem:BPWithImmigrationLaw}For $v\in\left\{ y,z\right\} $ and
any $u\in\partial B\left(v,r_{k^{+}+1}\right)$ 
\[
\mbox{the }P_{u}-\mbox{law of }\left(T_{l}^{y,\infty}\left(W_{\cdot\wedge T_{B\left(v,r_{k}\right)}}\right)\right)_{l\ge k^{+}}\mbox{ is }\tilde{\mathbb{G}}_{1}.
\]
\end{lem}
\begin{proof}
Let
\[
S_{l}=T_{B\left(v,r_{l}\right)},l\ge0,
\]
and consider for $k\le p\le k^{+}$ the number of traversals at each
scale which happen between $S_{p+1}$ and $S_{p}$,
\[
T_{l}^{p}\overset{\mbox{def}}{=}T_{l}^{v,\infty}\left(W_{\left(S_{p+1}+\cdot\right)\wedge S_{p}}\right),l\ge0.
\]
Since $T_{B\left(v,r_{k}\right)}=S_{k}>\ldots>S_{k^{+}}>S_{k^{+}+1}=0$
we have
\begin{equation}
T_{l}^{v,\infty}\left(W_{\cdot\wedge T_{B\left(v,r_{k}\right)}}\right)=\sum_{p=k}^{k^{+}}T_{l}^{p}.\label{eq:WrritenasSum}
\end{equation}
A proof similar to that of \prettyref{lem:BranchingProc} shows that
the law of $T_{p+\cdot}^{p},$ is $\mathbb{G}_{1}$, and the strong
Markov property shows that the $T_{l}^{p},l\ge0,$ are independent.
Thus the claim follows by \prettyref{eq:WrritenasSum} and the definition
\prettyref{eq:DefOfGtilde} of $\tilde{\mathbb{G}}_{1}$.
\end{proof}
From this we easily get the law of the sum $\sum_{i=1}^{n}T_{l}^{v,\infty}\left(X_{\cdot}^{i}\right)$:
\begin{cor}
\textup{($v\in\left\{ y,z\right\} $) We have
\[
\mbox{the }P_{x}-\mbox{law of }\left(\sum_{i=1}^{n}T_{l}^{v,\infty}\left(X_{\cdot}^{i}\right)\right)_{l\ge0}\mbox{ is }\tilde{\mathbb{G}}_{n}.
\]
}\end{cor}
\begin{proof}
By the strong Markov property applied at times $R_{i}\left(v,r_{k},r_{k^{+}+1}\right),i=1,\ldots,n,$
and \prettyref{lem:BPWithImmigrationLaw} $\left(T_{l}^{v,\infty}\left(X_{\cdot}^{i}\right)\right)_{l\ge0}$
are iid for $i=1,\ldots,n$ with law $\mathbb{\tilde{G}}_{1}$. Thus
clearly $\left(\sum_{i=1}^{n}T_{l}^{v,\infty}\left(X_{\cdot}^{i}\right)\right)_{l\ge0}$
has law $\tilde{\mathbb{G}}_{n}$ by the definition \prettyref{eq:DefOfGtilde}.
\end{proof}
We now provide a bound on the barrier crossing event corresponding
to $J_{v,n}^{\downarrow}$ for the Galton-Watson law $\tilde{\mathbb{G}}_{n}$.
\begin{lem}
\label{lem:ImmigrationBPBarrier}For any $n\ge0$ we have that
\begin{equation}
\mathbb{\tilde{G}}_{n}\left[\gamma\left(l\right)\le\sqrt{T_{l}}\le\delta\left(l\right)\mbox{ for }l=k^{+},\ldots,L-l_{0}|T_{L-1}=0\right]\le c\frac{g\left(k^{+}\right)l_{0}^{0.51}}{L-l_{0}-k^{+}-1}.\label{eq:ImmigrationBPBarrier}
\end{equation}
\end{lem}
\begin{proof}
By definition of $\tilde{\mathbb{G}}_{n}$ the law of $\left(T_{l+k^{+}}\right)_{l\ge0}$
under $\tilde{\mathbb{G}}_{n}\left[\cdot|T_{L-1}=0,T_{k^{+}}=a\right]$
is the $\mathbb{G}_{a}\left[\cdot|T_{L-k^{+}-1}=0\right]$ law of
$\left(T_{l}\right)_{l\ge0}$. Thus the probability in question is
bounded above by 
\[
\sup_{\gamma\left(k^{+}\right)^{2}\le a\le\delta\left(k^{+}\right)^{2}}\mathbb{G}_{a}\left[\gamma\left(l\right)\le\sqrt{T_{l-k^{+}}}\le\delta\left(l\right)\mbox{\,\ for }l=k^{+},\ldots,L-l_{0}|T_{L-k^{+}-1}=0\right],
\]
The required bound therefore follows by \prettyref{eq:GWBarrierGammaDeltaFromk}
with $k^{+}\le\left(1-\frac{s}{5}\right)L$ in place of $k$.
\end{proof}
We now summarize our work so far for the regime $k\le\left(1-\frac{s}{10}\right)L$
in the form of a bound on the conditional probability of $J_{v,n}^{\downarrow}$.
We will see that the conditioning essentially corresponds to conditioning
on $\left\{ H_{B\left(v,r_{L}\right)}>D_{t_{-s}}^{v,0}\right\} $.
\begin{lem}
\label{lem:BarrierPartLow}($v\in\left\{ y,z\right\} $) For all $n\ge1$,
\begin{equation}
\sup_{n\ge0}P_{x}\left[J_{v,n}^{\downarrow}|H_{B\left(v,r_{L}\right)}>D_{n}\left(v,r_{k},r_{k^{+}+1}\right)\right]\le c\frac{g\left(k^{+}\right)l_{0}^{0.51}}{L-l_{0}-k-1}.\label{eq:BarrierPartLow}
\end{equation}
\end{lem}
\begin{proof}
The event that we condition on in \eqref{eq:BarrierPartLow} can be
rewritten as $\left\{ \sum_{i=1}^{n}T_{L-1}^{v,\infty}\left(X_{\cdot}^{i}\right)=0\right\} $.
Therefore by \prettyref{lem:BPWithImmigrationLaw} the probability
in \eqref{eq:BarrierPartLow} equals that in \eqref{eq:ImmigrationBPBarrier},
so that the required bound follows by \prettyref{lem:ImmigrationBPBarrier}.
\end{proof}
The above lemma will be used to give the contribution from $J_{y}^{\downarrow}$
and $J_{z}^{\downarrow}$ to our bound on the probability of \eqref{eq:SmallKBiggerEvent}.
These contributions should be roughly independent, but to obtain a
rigorous bound we will need a decoupling. Our approach is inspired
by Lemma 7.4 \cite{Demboetal_LatePoints}. The first step in obtaining
the decoupling is the next lemma which essentially speaking shows
that the exit distribution of Brownian motion from a ball (both unconditioned
and conditioned to avoid a smaller ball) does not depend much on the
starting point, as long as the starting point is not close to the
boundary.
\begin{lem}
($v\in\left\{ y,z\right\} $) Let $\lambda$ be the uniform distribution
on $\partial B\left(v,r_{k}\right)$. For any $u\in\partial B\left(v,k^{+}\right)\cup\partial B\left(v,k^{+}+1\right)$
and measurable $B$,
\begin{equation}
P_{u}\left[W_{T_{B\left(v,r_{k}\right)}}\in B\right]=\left(1+O\left(L^{-100}\right)\right)\lambda\left(B\right),\label{eq:HarmonicMeasure}
\end{equation}
and for any $u\in\partial B\left(v,k^{+}\right)$,\textup{
\begin{equation}
P_{u}\left[W_{T_{B\left(v,r_{k}\right)}}\in B|H_{B\left(v,r_{k^{+}+1}\right)}>T_{B\left(v,r_{k}\right)}\right]=\left(1+O\left(L^{-99}\right)\right)\lambda\left(B\right).\label{eq:HarmonicMeasureCond}
\end{equation}
}\end{lem}
\begin{proof}
A classical result on the harmonic measure of Brownian motion says
that for $R>0$ and $u\in B\left(0,R\right)\subset\mathbb{R}^{2}$
\begin{equation}
P_{u}^{\mathbb{R}^{2}}\left[W_{T_{B\left(0,R\right)}}\in db\right]=\frac{R^{2}-\left|u\right|^{2}}{\left|u-b\right|^{2}}\tilde{\lambda}\left(db\right),\label{eq:HarmonicMeasureR2}
\end{equation}
where $\tilde{\lambda}$ is the uniform distribution on $\partial B\left(0,R\right)$
(see Theorem 3.43 \cite{PeresMoertersBrownianMotion}). With $R=r_{k}$
and $\left|u\right|=r_{k^{+}}$ or $\left|u\right|=r_{k^{+}+1}$ this
implies \prettyref{eq:HarmonicMeasure}, since $r_{k^{+}+1}/r_{k}\le cL^{-100}$
(by \prettyref{eq:DefOfRadii} and \prettyref{eq:DefOfKPlus}, and
using also that $B\left(0,r_{k}\right)\subset\mathbb{R}^{2}$ can
be identified with $B\left(v,r_{k}\right)\subset\mathbb{T}$; see
\eqref{eq:LawOfBMInTorusAndR2Same}). To get \prettyref{eq:HarmonicMeasureCond}
note that $P_{u}\left[W_{T_{B\left(v,r_{k}\right)}}\in B,H_{B\left(v,r_{k^{+}+1}\right)}>T_{B\left(v,r_{k}\right)}\right]$
equals 
\[
P_{u}\left[W_{T_{B\left(v,r_{k}\right)}}\in B\right]-P_{u}\left[W_{T_{B\left(v,r_{k}\right)}}\in B,H_{B\left(v,r_{k^{+}+1}\right)}<T_{B\left(v,r_{k}\right)}\right].
\]
The first term equals $\left(1+O\left(L^{-100}\right)\right)\lambda\left(B\right)$
by \prettyref{eq:HarmonicMeasure}. Also by the strong Markov property
applied at time $H_{B\left(v,r_{k^{+}+1}\right)}$ and \prettyref{eq:HarmonicMeasure}
the second term equals
\[
P_{u}\left[H_{B\left(v,r_{k^{+}+1}\right)}<T_{B\left(v,r_{k}\right)}\right]\left(1+O\left(L^{-100}\right)\right)\lambda\left(B\right).
\]
Thus $P_{u}\left[W_{T_{B\left(v,r_{k}\right)}}\in B,H_{B\left(v,r_{k^{+}+1}\right)}>T_{B\left(v,r_{k}\right)}\right]$
equals
\[
\lambda\left(B\right)\left\{ P_{u}\left[H_{B\left(v,r_{k^{+}+1}\right)}>T_{B\left(v,r_{k}\right)}\right]+O\left(L^{-100}\right)\right\} .
\]
But by \prettyref{eq:HitLProbStartingFromk}
\[
P_{u}\left[H_{B\left(v,r_{k^{+}+1}\right)}>T_{B\left(v,r_{k}\right)}\right]=\frac{1}{k^{+}+1-k}\ge L^{-1},
\]
so that in fact $P_{u}\left[W_{T_{B\left(0,R\right)}}\in B,H_{B\left(v,r_{k^{+}+1}\right)}>T_{B\left(v,r_{k}\right)}\right]$
is equal to
\[
\lambda\left(B\right)P_{u}\left[H_{B\left(v,r_{k^{+}+1}\right)}>T_{B\left(v,r_{k}\right)}\right]\left(1+O\left(L^{-99}\right)\right).
\]
This gives \prettyref{eq:HarmonicMeasureCond}.
\end{proof}
As a step in the ``decoupling'' of $J_{y}^{\downarrow}$ and $J_{z}^{\downarrow}$
we will now use the previous lemma to show that the part of an excursion
from $\partial B\left(v,r_{k^{+}+1}\right)$ to $\partial B\left(v,r_{k}\right)$
that takes place within $B\left(v,r_{k^{+}+1}\right)$ is almost independent
from the end point of the excursion. This will be used to show that
$J_{y}^{\downarrow}$, when conditioned to avoid $B\left(y,r_{L}\right)$,
is almost independent of the parts of \eqref{eq:SmallKBiggerEvent}
that refer to $z$ (note that $J_{y}^{\downarrow}$ only depends on
the parts of the excursions that take place in $B\left(y,r_{k^{+}+1}\right)$),
and vice versa with $y$ and $z$ swapped.

To this end, let 
\[
S=S\left(v\right)=\sup\left\{ D_{n}^{v,k^{+}}:D_{n}^{v,k^{+}}<T_{B\left(v,r_{k}\right)}\right\} ,
\]
be the time the last excursion from scale $k^{+}+1$ to scale $k^{+}$
before $T_{B\left(v,r_{k}\right)}$ ends. For $a\in\partial B\left(v,r_{k^{+}+1}\right)$
and $b\in\partial B\left(v,r_{k}\right)$ let
\begin{equation}
\mu_{a,b}\left[\cdot\right]=P_{a}\left[W_{\cdot\wedge T_{B\left(v,r_{k}\right)}}\in\cdot|W_{T_{B\left(v,r_{k}\right)}}=b\right],\label{eq:muab}
\end{equation}
be the law of an excursion starting in $a$ conditioned to end in
$b$. Let
\begin{equation}
\begin{array}{ccl}
\tilde{\mu}_{a,b}\left[\cdot\right] & = & \mu_{a,b}\left[\cdot|H_{B\left(v,r_{L}\right)}>T_{B\left(v,r_{k}\right)}\right]\\
 & = & P_{a}\left[W_{\cdot\wedge T_{B\left(v,r_{k}\right)}}\in\cdot|H_{B\left(v,r_{L}\right)}>T_{B\left(v,r_{k}\right)},W_{T_{B\left(v,r_{k}\right)}}=b\right],
\end{array}\label{eq:mutildeab}
\end{equation}
be the law of an excursion conditioned to end in $b$ and avoid $B\left(v,r_{L}\right)$,
and let
\begin{equation}
\tilde{\mu}_{a}\left[\cdot\right]=P_{a}\left[W_{\cdot\wedge T_{B\left(v,r_{k}\right)}}\in\cdot|H_{B\left(v,r_{L}\right)}>T_{B\left(v,r_{k}\right)}\right],\label{eq:mutildea}
\end{equation}
be the law of an excursion avoiding $B\left(v,r_{L}\right)$, without
conditioning on the end point. The result says that:
\begin{lem}
\label{lem:IndepOfEndPoint}($v\in\left\{ y,z\right\} $) For any
$u\in\partial B\left(v,k^{+}+1\right)$, $w\in B\left(v,r_{k}\right)$
we have
\begin{equation}
\tilde{\mu}_{u,w}\left[W_{\cdot\wedge S}\in\cdot\right]=\tilde{\mu}_{u}\left[W_{\cdot\wedge S}\in\cdot\right]\left(1+O\left(L^{-99}\right)\right).\label{eq:IndepOfEndPoint}
\end{equation}
\end{lem}
\begin{proof}
We will show that
\begin{equation}
\mu_{u,w}\left[W_{\cdot\wedge S}\in\cdot\right]=\left(1+O\left(L^{-99}\right)\right)P_{w}\left[W_{\cdot\wedge S}\in\cdot\right].\label{eq:UnCond}
\end{equation}
The claim then follows, since the left-hand side of \prettyref{eq:IndepOfEndPoint}
equals
\begin{equation}
\frac{\mu_{u,w}\left[W_{\cdot\wedge S}\in\cdot,H_{B\left(v,r_{L}\right)}>T_{B\left(v,r_{k}\right)}\right]}{\mu_{u,w}\left[H_{B\left(v,r_{L}\right)}>T_{B\left(v,r_{k}\right)}\right]},\label{eq:DenomAndNum}
\end{equation}
so that we can apply \prettyref{eq:UnCond} to the denominator and
numerator of \prettyref{eq:DenomAndNum} (note that $H_{B\left(v,r_{L}\right)}>T_{B\left(v,r_{k}\right)}$
is $W_{\cdot\wedge S}-$measurable) to get that \prettyref{eq:DenomAndNum}
equals
\[
\frac{P_{w}\left[W_{\cdot\wedge S}\in\cdot,H_{B\left(v,r_{L}\right)}>T_{B\left(v,r_{k}\right)}\right]}{P_{w}\left[H_{B\left(v,r_{L}\right)}>T_{B\left(v,r_{k}\right)}\right]}\left(1+O\left(L^{-99}\right)\right),
\]
which equals the right-hand side of \prettyref{eq:IndepOfEndPoint}.

To show \prettyref{eq:UnCond} we note that $P_{u}\left[W_{\cdot\wedge S}\in A,S=D_{n}^{y,k^{+}},W_{T_{B\left(y,r_{k}\right)}}\in B\right]$
equals
\[
P_{u}\left[W_{\cdot\wedge D_{n}^{y,k^{+}}}\in A,T_{B\left(y,r_{k}\right)}\circ\theta_{D_{n}^{y,k^{+}}}<H_{B\left(y,r_{k^{+}+1}\right)}\circ\theta_{D_{n}^{y,k^{+}}},W_{T_{B\left(y,r_{k}\right)}}\in B\right].
\]
By the strong Markov property this probability can be written as
\begin{equation}
P_{u}\left[W_{\cdot\wedge D_{n}^{v,k^{+}}}\in A,P_{W_{D_{n}^{v,k^{+}}}}\left[T_{B\left(v,r_{k}\right)}<H_{B\left(v,r_{k^{+}+1}\right)},W_{T_{B\left(y,r_{k}\right)}}\in B\right]\right].\label{eq:MarkovApplied}
\end{equation}
But
\[
\begin{array}{l}
P_{W_{D_{n}^{v,k^{+}}}}\left[T_{B\left(v,r_{k}\right)}<H_{B\left(v,r_{k^{+}+1}\right)},W_{T_{B\left(v,r_{k}\right)}}\in B\right]\\
=P_{W_{D_{n}^{v,k^{+}}}}\left[W_{T_{B\left(v,r_{k}\right)}}\in B|T_{B\left(v,r_{k}\right)}<H_{B\left(v,r_{k^{+}+1}\right)}\right]P_{W_{D_{n}^{v,k^{+}}}}\left[T_{B\left(v,r_{k}\right)}<H_{B\left(v,r_{k^{+}+1}\right)}\right]\\
=\left(1+O\left(L^{-99}\right)\right)\lambda\left(B\right)P_{W_{D_{n}^{v,k^{+}}}}\left[T_{B\left(v,r_{k}\right)}<H_{B\left(v,r_{k^{+}+1}\right)}\right],
\end{array}
\]
by \prettyref{eq:HarmonicMeasureCond}. Thus the probability \prettyref{eq:MarkovApplied}
equals
\[
\begin{array}{l}
\left(1+O\left(L^{-99}\right)\right)\lambda\left(B\right)P_{u}\left[W_{\cdot\wedge D_{n}^{v,k^{+}}}\in A,P_{W_{D_{n}^{v,k^{+}}}}\left[T_{B\left(v,r_{k}\right)}<H_{B\left(v,r_{k^{+}+1}\right)}\right]\right]\\
=\left(1+O\left(L^{-99}\right)\right)\lambda\left(B\right)P_{u}\left[W_{\cdot\wedge D_{n}^{v,k^{+}}}\in A,S=D_{n}^{y,k^{+}}\right],
\end{array}
\]
and we get that
\[
\begin{array}{l}
P_{u}\left[W_{\cdot\wedge S}\in A,S=D_{n}^{y,k^{+}},W_{T_{B\left(y,r_{k}\right)}}\in B\right]\\
=\left(1+O\left(L^{-99}\right)\right)\lambda\left(B\right)P_{u}\left[W_{\cdot\wedge D_{n}^{v,k^{+}}}\in A,S=D_{n}^{y,k^{+}}\right].
\end{array}
\]
Thus summing over $n$ we obtain
\[
P_{u}\left[W_{\cdot\wedge S}\in A,W_{T_{B\left(y,r_{k}\right)}}\in B\right]=\left(1+O\left(L^{-99}\right)\right)\lambda\left(B\right)P_{u}\left[W_{\cdot\wedge S}\in A\right].
\]
Using \prettyref{eq:HarmonicMeasure} this gives
\[
P_{u}\left[W_{\cdot\wedge S}\in A,W_{T_{B\left(y,r_{k}\right)}}\in B\right]=\left(1+O\left(L^{-99}\right)\right)P_{u}\left[W_{T_{B\left(y,r_{k}\right)}}\in B\right]P_{u}\left[W_{\cdot\wedge S}\in A\right],
\]
from which \prettyref{eq:UnCond} follows (recall \prettyref{eq:muab}).
\end{proof}
We now prove a bound that deals with the contribution from the event
$J_{v}^{\downarrow}$, even when conditioned on ``what goes on outside
$B\left(v,r_{k}\right)$''. To this end let
\[
Y_{\cdot}^{i}=Y_{\cdot}^{i}\left(v\right)=W_{\left(D_{i}\left(v,r_{k},r_{k^{+}+1}\right)+\cdot\right)\wedge R_{i+1}\left(v,r_{k},r_{k^{+}+1}\right)},i\ge1,
\]
be the excursions from $\partial B\left(v,r_{k^{+}+1}\right)$ to
$\partial B\left(v,r_{k}\right)$. Define the $\sigma$-algebra 
\[
\mathcal{G}=\mathcal{G}\left(v\right)=\sigma\left(W_{\cdot\wedge R_{1}\left(v,r_{k},r_{k^{+}+1}\right)},Y_{\cdot}^{i}:i\ge1\right).
\]
The bound says that:
\begin{prop}
\label{prop:UntanglingJy}For any $l_{0}<k<L-l_{0}$ and $v\in\left\{ y,z\right\} $
we have that \textup{
\begin{equation}
P_{x}\left[J_{v}^{\downarrow}\cap\left\{ H_{B\left(v,r_{L}\right)}\ge D_{t_{-s}}^{v,0}\right\} |\mathcal{G}\right]\le c\frac{g\left(k^{+}\right)l_{0}^{0.51}}{L-l_{0}-k-1}P_{x}\left[H_{B\left(v,r_{L}\right)}\ge D_{t_{-s}}^{v,0}|\mathcal{G}\right].\label{eq:Blaja}
\end{equation}
}\end{prop}
\begin{proof}
Recall the definitions \prettyref{eq:NumberOfExcursions} of $N=N\left(v\right)$
and \prettyref{eq:DefOFXiExcursions} of $X_{\cdot}^{i}$. We have
that
\begin{equation}
\left\{ H_{B\left(v,r_{L}\right)}\ge D_{t_{-s}}^{v,0}\right\} =\left\{ H_{B\left(v,r_{L}\right)}\left(X_{\cdot}^{i}\right)>T_{B\left(y,r_{k}\right)}\left(X_{\cdot}^{i}\right),i=1,\ldots,N\right\} \overset{\mbox{def}}{=}A.\label{eq:AvoidingSmallestBallInTermsOfExcursions}
\end{equation}
Also $N\le T_{k^{+}}^{y,t_{-s}}$ since each excursion $X_{\cdot}^{i}$
contains at least one traversal $k^{+}\to k^{+}+1$ (recall \prettyref{eq:TraversalsDef}).
Thus on the event $J_{v}^{\downarrow}$ (see \prettyref{eq:JDown})
we have 
\begin{equation}
N\left(v\right)\le\delta^{2}\left(k^{+}\right)\le2L^{2},\mbox{\,\ by }\eqref{eq:defofts},\eqref{eq:BetaDef},\eqref{eq:DeltaBarrierDef}.\label{eq:BoundOnN}
\end{equation}
Let $\rho_{u},u\in B\left(y,r_{k}\right)$ denote the map that rotates
$B\left(v,r_{k}\right)\subset\mathbb{T}$ around $v$ so that $u$
lies on the same horizontal line as $v$, and for any path $w\in C\left(\mathbb{R}_{+},\mathbb{T}\right)$
let $\rho\left(w\right)=\left(\rho_{w_{0}}\left(w_{t}\right)\right)_{t\ge0}$.
We have that $T_{l}^{y,\infty}\left(X_{\cdot}^{i}\right)=T_{l}^{y,\infty}\left(\rho\left(X_{\cdot\wedge S}^{i}\right)\right)$
(recall \prettyref{eq:DefOfSLastTime}), and therefore $J_{v}^{\downarrow}$
only depends only on 
\[
\bar{X}_{\cdot}^{i}=\rho\left(X_{\cdot\wedge S}^{i}\right),i=1,\ldots,n,
\]
so that there exists a family of measurable functions $f_{1},f_{2},\ldots$,
such that
\[
1_{J_{y,n}^{\downarrow}}=f_{n}\left(\bar{X}_{\cdot}^{1},\ldots,\bar{X}_{\cdot}^{n}\right).
\]
We thus have (recall \eqref{eq:JvNequalsJv} and \eqref{eq:AvoidingSmallestBallInTermsOfExcursions})
\begin{equation}
P_{x}\left[J_{v}^{\downarrow}\cap\left\{ H_{B\left(v,r_{L}\right)}\ge D_{t_{-s}}^{v,0}\right\} |\mathcal{G}\right]=E_{x}\left[f_{N}\left(\bar{X}_{\cdot}^{1},\ldots,\bar{X}_{\cdot}^{N}\right)1_{A}|\mathcal{G}\right].\label{eq:FirstEquality}
\end{equation}
Now under $P_{x}\left[\cdot|\mathcal{G}\right]$ the $X_{\cdot}^{i},i=1\ldots,N,$
are independent and $X_{\cdot}^{i}$ has the law $\mu_{X_{0}^{i},X_{\infty}^{i}}$
from \prettyref{eq:muab} (note that $X_{0}^{i}$ and $X_{\infty}^{i}$
are $\mathcal{G}-$measurable). Thus \prettyref{eq:FirstEquality}
in fact equals
\begin{equation}
1_{\left\{ N\le2L^{2}\right\} }\otimes_{i=1}^{N}\mu_{X_{0}^{i},X_{\infty}^{i}}\left[f_{N}\left(\bar{Z}_{\cdot}^{1},\ldots,\bar{Z}_{\cdot}^{N}\right)1_{\left\{ H_{B\left(y,r_{L}\right)}\left(Z_{\cdot}^{i}\right)>T_{B\left(y,r_{k}\right)}\left(Z_{\cdot}^{i}\right),i=1,\ldots,N\right\} }\right],\label{eq:SecondEquality}
\end{equation}
where the vector $\left(Z_{\cdot}^{1},\ldots,Z_{\cdot}^{N}\right)$
has law $\otimes_{i=1}^{N}\mu_{X_{0}^{i},X_{\infty}^{i}}$ and
\[
\bar{Z}_{\cdot}^{i}=\rho\left(Z_{\cdot\wedge S}^{i}\right).
\]
Now (recall \eqref{eq:muab} and \eqref{eq:mutildeab})
\[
\mu_{a,b}\left[\cdot1_{\left\{ H_{B\left(v,r_{L}\right)}>T_{B\left(y,r_{k}\right)}\right\} }\right]=\tilde{\mu}_{a,b}\left[\cdot\right]\mu_{a,b}\left[H_{B\left(v,r_{L}\right)}>T_{B\left(y,r_{k}\right)}\right],
\]
so that in fact \prettyref{eq:SecondEquality} equals
\begin{equation}
1_{\left\{ N\le2L^{2}\right\} }\left(\otimes_{i=1}^{N}\tilde{\mu}_{X_{0}^{i},X_{\infty}^{i}}\left[f_{N}\left(\bar{Z}_{\cdot}^{1},\ldots,\bar{Z}_{\cdot}^{N}\right)\right]\right)P_{x}\left[H_{B\left(v,r_{L}\right)}\ge D_{t_{-s}}^{v,0}|\mathcal{G}\right].\label{eq:Uno}
\end{equation}
Using \prettyref{eq:IndepOfEndPoint} together with $\left(1+cL^{-99}\right)^{2L^{2}}\le c$
this is bounded above by
\begin{equation}
c{\displaystyle \sup_{n\ge0}}\left(\otimes_{i=1}^{n}\tilde{\mu}_{u}\left[f_{n}\left(\bar{Z}_{\cdot}^{1},\ldots,\bar{Z}_{\cdot}^{n}\right)\right]\right)P_{x}\left[H_{B\left(v,r_{L}\right)}\ge D_{t_{-s}}^{v,0}|\mathcal{G}\right],\label{eq:BASVASFDASF}
\end{equation}
for an arbitrary $u\in\partial B\left(v,r_{k^{+}+1}\right)$ (the
law of $\rho\left(W_{\cdot\wedge S}\right)$ under $\tilde{\mu}_{u}$
is independent of $u$, see \prettyref{eq:RotationalInvarianceInTorusBall}).
Now consider the law of $\left(\bar{X}_{\cdot}^{1},\ldots,\bar{X}_{\cdot}^{n}\right)$
under $P_{x}\left[\cdot|H_{B\left(v,r_{L}\right)}>D_{n}\left(v,r_{k},r_{k^{+}+1}\right)\right]$.
By the strong Markov property and the rotational invariance \prettyref{eq:RotationalInvarianceInTorusBall}
this vector is iid with law $\tilde{\mu}_{u}$. Thus we have that
\begin{equation}
\begin{array}{l}
\otimes_{i=1}^{n}\tilde{\mu}_{u}\left[f_{N}\left(\bar{Z}_{\cdot}^{1},\ldots,\bar{Z}_{\cdot}^{N}\right)\right]\\
=P_{x}\left[f_{n}\left(\bar{X}_{\cdot}^{1},\ldots,\bar{X}_{\cdot}^{n}\right)|H_{B\left(v,r_{L}\right)}>D_{n}\left(v,r_{k},r_{k^{+}+1}\right)\right]\\
=P_{x}\left[J_{v,n}^{\downarrow}|H_{B\left(v,r_{L}\right)}>D_{n}\left(v,r_{k},r_{k^{+}+1}\right)\right]\le c\frac{g\left(k^{+}\right)l_{0}^{0.51}}{L-l_{0}-k-1},
\end{array}\label{eq:Tres}
\end{equation}
by \prettyref{lem:BarrierPartLow}. Combining this with \prettyref{eq:FirstEquality}-\prettyref{eq:BASVASFDASF}
gives the claim.
\end{proof}
We are now ready to prove the two point probability estimate for small
$k$.
\begin{prop}
\label{prop:SmallKCase}If $k\le\left(1-\frac{s}{10}\right)L$ then
\begin{equation}
\begin{array}{l}
P_{x}\left[J_{y}^{\downarrow}\cap\left\{ H_{B\left(y,r_{L}\right)}\ge D_{t_{-s}}^{y,0}\right\} \cap J_{z}^{\downarrow}\cap\left\{ \gamma\left(k\right)\le\sqrt{T_{k}^{z,t_{-s}}},H_{B\left(z,r_{L}\right)}\ge D_{t_{-s}}^{z,0}\right\} \right]\\
\le c\left(s\right)e^{-\left(4L-2k\right)-cf\left(k\right)}L^{2s}l_{0}^{1.02}\left(g\left(k^{+}\right)\right)^{2}.
\end{array}\label{eq:SmallKCase}
\end{equation}
\end{prop}
\begin{proof}
We first use \prettyref{prop:UntanglingJy} with $v=y$. Note that
with this choice of $v$ we have
\[
J_{z}^{\downarrow}\cap\left\{ \gamma\left(k\right)\le\sqrt{T_{k}^{z,t_{-s}}},H_{B\left(z,r_{L}\right)}\ge D_{t_{-s}}^{z,0}\right\} \in\mathcal{G}\left(y\right),
\]
since these events depend only ``on what goes on outside $B\left(y,r_{k}\right)$''
(recall \prettyref{eq:BranchingPoint} and \prettyref{eq:JDown}).
Thus it follows by \prettyref{prop:UntanglingJy} that the probability
in \prettyref{eq:SmallKCase} is bounded above by 
\[
c\frac{g\left(k^{+}\right)l_{0}^{0.51}}{L-l_{0}-k-1}P_{x}\left[\left\{ H_{B\left(y,r_{L}\right)}\ge D_{t_{-s}}^{y,0}\right\} \cap J_{z}^{\downarrow}\cap\left\{ \gamma\left(k\right)\le\sqrt{T_{k}^{z,t_{-s}}},H_{B\left(z,r_{L}\right)}\ge D_{t_{-s}}^{z,0}\right\} \right].
\]
Next we let $v=z$ in \prettyref{prop:UntanglingJy}. We now have
that
\[
\left\{ H_{B\left(y,r_{L}\right)}\ge D_{t_{-s}}^{y,0}\right\} \cap\left\{ \gamma\left(k\right)\le\sqrt{T_{k}^{z,t_{-s}}}\right\} \in\mathcal{G}\left(z\right),
\]
so that again by \prettyref{prop:UntanglingJy} the probability in
\prettyref{eq:SmallKCase} is bounded above by
\[
\left(c\frac{g\left(k^{+}\right)l_{0}^{0.51}}{L-l_{0}-k-1}\right)^{2}P_{x}\left[\left\{ H_{B\left(y,r_{L}\right)}\ge D_{t_{-s}}^{y,0}\right\} \cap\left\{ \gamma\left(k\right)\le\sqrt{T_{k}^{z,t_{-s}}},H_{B\left(z,r_{L}\right)}\ge D_{t_{-s}}^{z,0}\right\} \right].
\]
Now $\left\{ \gamma\left(k\right)\le\sqrt{T_{k}^{z,t_{-s}}},H_{B\left(z,r_{L}\right)}\ge D_{t_{-s}}^{z,0}\right\} \subset\left\{ H_{B\left(z,r_{L}\right)}\ge D_{\gamma\left(k\right)^{2}}^{z,k}\right\} $
(recall \prettyref{eq:ExcursionTimeAbreviation} and \prettyref{eq:TraversalsDef})
so that by \prettyref{cor:RecursionImplemented} with $J=C\left(\mathbb{R}_{+},\mathbb{T}\right)$
the probability in \prettyref{eq:SmallKCase} is at most
\begin{equation}
\left(c\frac{g\left(k^{+}\right)l_{0}^{0.51}}{L-l_{0}-k-1}\right)^{2}\times ce^{-2\left(L-k\right)-cf\left(k\right)}L^{1+s}\times P_{x}\left[H_{B\left(y,r_{L}\right)}\ge D_{t_{-s}}^{y,0}\right].\label{eq:AlmostDone}
\end{equation}
The latter probability is bounded above by $ce^{-2L}L^{1+s}$ by \prettyref{lem:NotHitByrL},
so \prettyref{eq:AlmostDone} simplifies to the right-hand side of
\prettyref{eq:SmallKCase} by noting that $L-l_{0}-k-1\ge c\left(s\right)L$.
\end{proof}
We can now finish the section by deriving the two point probability
estimate \prettyref{prop:MainTwoProfileBound} using \prettyref{prop:LargeKCase}
and \prettyref{prop:SmallKCase}.
\begin{proof}[Proof of \prettyref{prop:MainTwoProfileBound}]
We have (recall \prettyref{eq:TraversalsDef}, \prettyref{eq:ExtinctionImpliesNotHit},
\prettyref{eq:TruncatedSummandLB}, \prettyref{eq:JUp} and \prettyref{eq:JDown})
\[
I_{y}\subset J_{y}^{\uparrow}\cap\left\{ H_{B\left(y,r_{L}\right)}\ge D_{t_{-s}}^{y}\right\} \mbox{ and }I_{y}\subset J_{y}^{\downarrow}\cap\left\{ H_{B\left(y,r_{L}\right)}\ge D_{t_{-s}}^{y}\right\} .
\]
Similarly
\[
I_{z}\subset\left\{ H_{B\left(z,r_{L}\right)}\ge D_{\gamma\left(k\right)^{2}}^{z,k}\right\} \mbox{ and }I_{z}\subset J_{z}^{\downarrow}\cap\left\{ \gamma\left(k\right)\le\sqrt{T_{k}^{z,t_{-s}}},H_{B\left(z,r_{L}\right)}\ge D_{t_{-s}}^{z,0}\right\} .
\]
Thus by \prettyref{prop:LargeKCase} and \prettyref{prop:SmallKCase}
we have
\[
P_{x}\left[I_{y}\cap I_{z}\right]\le\begin{cases}
ce^{-\left(4L-2k\right)-cf\left(k\right)}L^{2s}\sqrt{l_{0}}g\left(k-2\right) & \mbox{ if }k\ge\left(1-\frac{s}{10}\right)L,\\
c\left(s\right)e^{-\left(4L-2k\right)-cf\left(k\right)}L^{2s}l_{0}^{1.02}\left(g\left(k^{+}\right)\right)^{2} & \mbox{ if }k\le\left(1-\frac{s}{10}\right)L.
\end{cases}
\]
Thus \prettyref{eq:TwoProfileBoundInStatement} follows, since $g\left(k-2\right)\le cg\left(k\right)$
and $g\left(k^{+}\right)\le cg\left(k\right)\left(\log L\right)^{0.51}$.
\end{proof}
This also completes the proof the lower bound \prettyref{prop:LowerBound},
modulo the barrier crossing results \prettyref{lem:GaltonWatsonBarrierBoundsSecLB}
and \prettyref{lem:GWBarrierGammaDeltaUpTok} which we have as of
yet only stated. Recall that the proof of the upper bound \prettyref{prop:UpperBound}
was also completed in \prettyref{sec:UpperBound} modulo the barrier
crossing result \prettyref{lem:GWBarrierAlpha}. The next section
gives the proof of these results.

\section{\label{sec:BoundaryCrossingProofs}Barrier estimate proofs}

In this section we prove the barrier crossing estimates \prettyref{lem:GWBarrierAlpha},
\prettyref{lem:GaltonWatsonBarrierBoundsSecLB} and \prettyref{lem:GWBarrierGammaDeltaUpTok}
for the Galton-Watson process $\left(T_{l}\right)_{l\ge0}$ that were
crucial in the proofs of the upper bound \prettyref{prop:UpperBound}
in \prettyref{sec:UpperBound} and the lower bound \prettyref{prop:LowerBound}
in sections \ref{sec:Lowerbound} and \ref{sec:TwoPointBound}.

The kind of barrier bounds we need appear in the literature for the
Brownian bridge process (indeed they are an integral part of the analysis
of branching Brownian motion that provides the inspiration for the
proof of our main result, see \cite{ABKGenealogy,Bramson1ConvergenceofSolutionsOfKolmogorovEqn,KistlerDerridasRandomEnergyModelsBeyondSpinGlasses}).
Our approach is to derive the needed bounds for the Galton-Watson
process from these Brownian bridge results, via a comparison to the
Bessel bridge. Roughly speaking the squared Bessel process of dimension
\emph{zero} is the continuous state space version of the Galton-Watson
process $\left(T_{l}\right)_{l\ge0}$, so that the $\mathbb{G}_{t_{s}}\left[\cdot|T_{L-1}=0\right]-$law
of $T_{l}$ should be similar to a squared Bessel bridge on $\left[0,L\right]$
of dimension \emph{zero}. A squared Bessel bridge of dimension \emph{one}
is a Brownian bridge squared. Our approach to get barrier bounds for
$T_{l}$ from bounds for Brownian bridge is thus to first translate
from ``discrete to continuous state space'' and then make a ``change
of dimension''.

For the first step we exploit that $\left(T_{l}\right)_{l\ge0}$ is
the law of the discrete edge local times of random walk on the path
$\left\{ 0,1,\ldots,L\right\} $, while the law of the continuous
local times of the vertices is a squared Bessel process of dimension
one. For the second step we use an explicit expression for the Radon-Nikodym
derivative of law of the squared Bessel bridge of dimension one with
respect to the law of the bridge with dimension zero.

\begin{figure}
\includegraphics{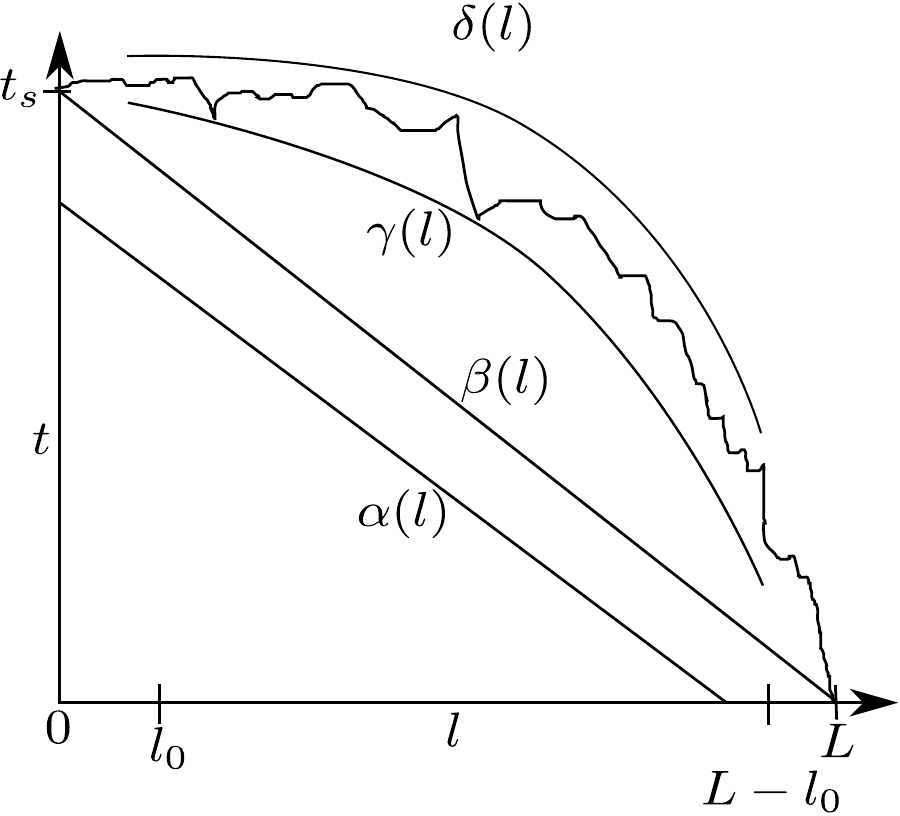}\\

\caption{\label{fig:Illustration-of-barriers}Illustration of the functions
$\alpha\left(l\right),\beta\left(l\right),\gamma\left(l\right)$ and
$\delta\left(l\right)$, and a sample paths that stays in the ``tube''
bounded by $\gamma\left(l\right)$ and $\delta\left(l\right)$.}

\framebox{\begin{minipage}[t]{1\columnwidth}%
\emph{In the proof of the upper bound \prettyref{prop:UpperBound}
one shows that with high probability there is no point $y\in F_{L}$
such that $T_{L-1}^{y,t_{s}}=0$ and $\sqrt{T_{l}^{y,t_{s}}}$ stays
above $\alpha\left(l\right)$. In the proof of lower bound \prettyref{prop:LowerBound}
one shows that with high probability there is a point $y\in F_{L}$
such that $T_{L-1}^{y,t_{-s}}=0$ and $\sqrt{T_{l}^{y,t_{-s}}}$ stays
in the aforementioned tube. This is done using bounds on the probability
that $\sqrt{T_{l}}$ starting at $T_{0}=t_{s}$ stays above $\alpha$,
when conditioned on $T_{L-1}=0$, and bounds on the probability that
this process stays in the tube. Note that $\beta\left(l\right)$ is
roughly speaking the mean of the conditioned process. (See \prettyref{lem:GWBarrierAlpha},
\prettyref{lem:GaltonWatsonBarrierBoundsSecLB}, and \prettyref{prop:BarrierSecGWProp}).}%
\end{minipage}}
\end{figure}

For convenience, let us now restate \prettyref{lem:GWBarrierAlpha},
\prettyref{lem:GaltonWatsonBarrierBoundsSecLB} and \prettyref{lem:GWBarrierGammaDeltaUpTok}
as one proposition. Recall first the definitions of $t_{s}=t_{s}\left(L\right)$
from \eqref{eq:defofts} and of the straight line $\beta\left(l\right)$
from \eqref{eq:BetaDef} (giving, roughly speaking, the mean of the
$T_{l}$ when $T_{0}=t_{s}$ and conditioned on $T_{L-1}=0$). Also
recall the definitions of the barriers $\alpha\left(l\right),\gamma\left(l\right)$
and $\delta\left(l\right)$ from \eqref{eq:AlphaBarrierDef}, \eqref{eq:GammaBarrierDef},
and \eqref{eq:DeltaBarrierDef} and the cut-off $l_{0}=l_{0}\left(L\right)$
from \eqref{eq:DefOfCutOff} (see also \prettyref{fig:Illustration-of-barriers}).
In the interest of brevity we introduce the following notation. For
any $T>0$, set $I\subset\left[0,T\right]$ and function $\eta:\left[0,T\right]\to\mathbb{R}$
we let $B_{\eta}\left(I\right)$ denote the event that a process is
above $\eta\left(t\right)$ for all $t\in I$. We let $B^{\eta}\left(I\right)$
denote the event that a process is below $\eta\left(t\right)$ for
all $t\in I$. For two functions $\eta$ and $\psi$ we let $B_{\eta}^{\psi}\left(I\right)=B_{\eta}\left(I\right)\cap B^{\psi}\left(I\right)$.
With this notation, we can now restate \prettyref{lem:GWBarrierAlpha}
as \prettyref{eq:GWBarrierAlphaNewNotation}, \prettyref{lem:GaltonWatsonBarrierBoundsSecLB}
as \prettyref{eq:GWBarrierGammaDeltaNewNotation} and \prettyref{lem:GWBarrierGammaDeltaUpTok}
as \prettyref{eq:GWBarrierGammaDeltaUptoKNewNotation}-\prettyref{eq:GWBarrierGammaDeltaFromkNewNotaiton}.
\begin{prop}
\label{prop:BarrierSecGWProp}For all $L\ge1$ and $s\in\left(-100,100\right)$\textup{
\begin{eqnarray}
\mathbb{G}_{t_{s}}\left[B_{\alpha^{2}}\left(\left\{ 0,\ldots,L-1\right\} \right)|T_{L-1}=0\right] & \le & c\frac{\left(\log L\right)^{4}}{L},\label{eq:GWBarrierAlphaNewNotation}\\
\mathbb{G}_{t_{s}}\left[B_{\gamma^{2}}^{\delta^{2}}\left(\left\{ l_{0},\ldots L-l_{0}\right\} \right)|T_{L-1}=0\right] & \asymp & \frac{l_{0}}{L}.\label{eq:GWBarrierGammaDeltaNewNotation}
\end{eqnarray}
}If also $l_{0}+1<k<L-l_{0}$ then
\begin{equation}
\mathbb{G}_{t_{s}}\left[B_{\gamma^{2}}^{\delta^{2}}\left(\left\{ l_{0},\ldots,k\right\} \right)|T_{L-1}=0\right]\le\frac{c\sqrt{l_{0}}g\left(k\right)}{k-l_{0}}.\label{eq:GWBarrierGammaDeltaUptoKNewNotation}
\end{equation}
If $l_{0}<k<L-l_{0}-1$ and $\gamma\left(k\right)^{2}\le a\le\delta\left(k\right)^{2}$
then \textup{
\begin{equation}
\mathbb{G}_{a}\left[B_{\gamma\left(k+\cdot\right)^{2}}\left(\left\{ 0,\ldots,L-k-l_{0}\right\} \right)|T_{L-1-k}=0\right]\le c\frac{l_{0}^{0.51}g\left(k+1\right)}{L-k-l_{0}-1}.\label{eq:GWBarrierGammaDeltaFromkNewNotaiton}
\end{equation}
}
\end{prop}
We start the proof of \prettyref{prop:BarrierSecGWProp} by recalling
and proving some barrier crossing bounds for the Brownian bridge.
To state these we let $\mathbb{P}_{x},x\in\mathbb{R}$, be the law
on $\left(C\left(\mathbb{R}_{+},\mathbb{R}\right),\mathcal{B}\left(\mathbb{R}_{+},\mathbb{R}\right)\right)$
which turns $X_{t},t\ge0,$ into a standard Brownian motion starting
at $x\in\mathbb{R}$. For $T>0$ and $a,b\in\mathbb{R}$ we write
$\mathbb{P}_{a\to b}^{T}$ for the law of Brownian bridge on $\left(C_{0}\left(\left[0,T\right],\mathbb{R}\right),\mathcal{B}\left(\left[0,T\right],\mathbb{R}\right)\right)$
starting at $a\in\mathbb{R}$ and ending in $b\in\mathbb{R}$ at time
$T$, that is
\[
\mathbb{P}_{a\to b}^{T}\left[\cdot\right]\overset{\mbox{def}}{=}\mathbb{P}_{a}\left[\cdot|X_{T}=b\right].
\]
Equivalently, $\mathbb{P}_{a\to b}^{T}$ is the law of the Gaussian
process on $\left[0,T\right]$ with
\begin{equation}
E_{x}\left[X_{t}\right]=h\left(t\right)\mbox{ and }\mbox{Cov}\left[X_{t},X_{s}\right]=\frac{t\left(T-s\right)}{T}\mbox{ for }0\le s\le t\le T,\label{eq:BrownianBridgeGaussianProc}
\end{equation}
where $h$ is the linear function with $h\left(0\right)=a$ and $h\left(T\right)=b$.
Recall that shifting a Brownian bridge by a linear function results
in a Brownian bridge with a shifted starting and ending point, that
is
\begin{equation}
\begin{array}{c}
\mbox{the }\mathbb{P}_{a\to b}^{T}-\mbox{law of }X_{t}+h\left(t\right)\mbox{ is }\mathbb{P}_{a+h\left(0\right)\to b+h\left(T\right)}^{T}\\
\mbox{for any linear }h:\left[0,T\right]\to\mathbb{R}\mbox{ and }a,b\in\mathbb{R}.
\end{array}\label{eq:ShiftBB}
\end{equation}
We now recall some barrier estimates from the literature. The probability
that Brownian bridge stays above (or below) a linear barrier throughout
its lifetime can be explicitly computed using the reflection principle;
we have
\begin{equation}
\mathbb{P}_{0\to0}^{T}\left[B_{h}\left(\left[0,T\right]\right)\right]=1-\exp\left(-\frac{2ab}{T}\right),\label{eq:BBContLinearBarrier}
\end{equation}
for all $T,a,b>0$ where $h\left(t\right)$ is the linear function
such that $h\left(0\right)=-a<0$ and $h\left(b\right)=-b<0$ (see
Proposition 3 \cite{ScheikeABoundarCrossResForBM}). For a linear
barrier that is ``checked'' only at integer times we have the following
bound 
\begin{equation}
\mathbb{P}_{0\to0}^{T}\left[B_{h}\left(\left[0,T\right]\cap\mathbb{N}\right)\right]\le c\frac{\left(1+a\right)\left(1+b\right)}{T},\label{eq:WebbLemma}
\end{equation}
for all $T,a,b>0$ (see Lemma 6.2 \cite{WebbbExactAsymptoticsofFreezingTransitionOfLogCorrelatedREM}).
Note that for $T$ much larger than $a$ and $b$, the right-hand
side of \prettyref{eq:BBContLinearBarrier} and the right-hand side
of \prettyref{eq:WebbLemma} have the same order. Also note that \prettyref{eq:BBContLinearBarrier}
is trivially a lower bound for the probability in \prettyref{eq:WebbLemma}.

For a linear barrier $h$ which is ``checked'' only during the interval
$\left[t_{1},T-t_{2}\right]$ for $t_{1}+t_{2}<T$ we have the upper
bound 
\begin{equation}
\mathbb{P}_{0\to0}^{T}\left[B_{h}\left(\left[t_{1},T-t_{2}\right]\right)\right]\le\frac{\left(a+\sqrt{t_{1}}\right)\left(b+\sqrt{t_{2}}\right)}{T-t_{1}-t_{2}},\label{eq:ABKLemma}
\end{equation}
where $h\left(t_{1}\right)=-a$ and $h\left(t_{2}\right)=-b$ (see
Lemma 3.4 \cite{ABKGenealogy}). We now adapt the proof of \eqref{eq:ABKLemma}
to give a version of that result for a barrier checked only at integer
times.
\begin{lem}
\label{lem:DiscBarrierStartingLate}For any $T>0$ and $t_{1},t_{2}\ge0$
such that $t_{1}+t_{2}<T$ and any $a,b>0$
\begin{equation}
\mathbb{P}_{0\to0}^{T}\left[B_{h}\left(\left[t_{1},T-t_{2}\right]\cap\mathbb{N}\right)\right]\le c\frac{\left(c+a+\sqrt{t_{1}}\right)\left(c+b+\sqrt{t_{2}}\right)}{T-t_{1}-t_{2}},\label{eq:DiscBarrierStartingLate}
\end{equation}
where where $h\left(t\right)$ is the linear function such that $h\left(t_{1}\right)=-a$
and $h\left(T-t_{2}\right)=-b$.\end{lem}
\begin{proof}
We may condition on $X_{t_{1}}$, $X_{T-t_{2}}$ to get that the left-hand
side of \prettyref{eq:DiscBarrierStartingLate} equals
\begin{equation}
\mathbb{P}_{0\to0}^{T}\left[\mathbb{P}_{X_{t_{1}}\to X_{T-t_{2}}}^{T-t_{1}-t_{2}}\left[B_{h\left(t_{1}+\cdot\right)}\left(\left[0,T-t_{1}-t_{2}\right]\cap\mathbb{N}\right)\right]\right],\label{eq:conditioninontoneTminusttwo}
\end{equation}
Now by \eqref{eq:ShiftBB} and \eqref{eq:WebbLemma} we have for $u,v\in\mathbb{R}$,
\begin{equation}
\mathbb{P}_{u\to v}^{T-t_{1}-t_{2}}\left[B_{h\left(t_{1}+\cdot\right)}\left(\left[0,T-t_{1}-t_{2}\right]\cap\mathbb{N}\right)\right]\le c\frac{\left(1+\left|u+a\right|\right)\left(1+\left|v+b\right|\right)}{T-t_{1}-t_{2}}.\label{eq:Ineq}
\end{equation}
We have (see \prettyref{eq:BrownianBridgeGaussianProc})
\begin{equation}
\mbox{Var}\left[X_{t}\right]=\mbox{Var}\left[X_{T-t}\right]=\frac{t\left(T-t\right)}{T}\asymp t,\label{eq:VarOfBB}
\end{equation}
so that $\mathbb{P}_{0\to0}^{T}\left[\left|X_{t_{1}}+a\right|\right]\le c\sqrt{t_{1}}+a$,
$\mathbb{P}_{0\to0}^{T}\left[\left|X_{T-t_{2}}+b\right|\right]\le c\sqrt{t_{2}}+b$,
and by Hölder's inequality
\[
\begin{array}{ccl}
\mathbb{P}_{0\to0}^{T}\left[\left|X_{t_{1}}+a\right|\left|X_{T-t_{2}}+b\right|\right] & \le & \sqrt{\mathbb{P}_{0\to0}^{T}\left[\left|X_{t_{1}}+a\right|^{2}\right]\mathbb{P}_{0\to0}^{T}\left[\left|X_{T-t_{2}}+b\right|^{2}\right]}\\
 & \le & \sqrt{\left(t_{1}+a^{2}\right)\left(t_{2}+b^{2}\right)}\le\left(\sqrt{t_{1}}+a\right)\left(\sqrt{t_{2}}+b\right).
\end{array}
\]
Thus \eqref{eq:DiscBarrierStartingLate} follows by plugging \prettyref{eq:Ineq}
into \prettyref{eq:conditioninontoneTminusttwo} and taking the expectation.
\end{proof}
Now consider the non-linear barrier $h_{\delta}:\left[0,T\right]\to\mathbb{R}$
given by $h_{\delta}\left(t\right)=\min\left(t^{\delta},\left(T-t\right)^{\delta}\right)$
(note that with $T=L$ we have $f=h_{0.49}$ and $g=h_{0.51}$, see
\prettyref{eq:DefOffFunc} and \prettyref{eq:DefOfgfunc}). Bramson
shows that
\begin{equation}
\mathbb{P}_{0\to0}^{T}\left[B_{h_{\delta}}\left(\left[t,T-t\right]\right)|B_{0}\left(\left[t,T-t\right]\right)\right]\to1\mbox{ as }t\to\infty,\mbox{ for }\delta<\frac{1}{2},\label{eq:BramsonSandwichEpsSmall}
\end{equation}
uniformly in $T$ (see Proposition 6.1 \cite{Bramson1ConvergenceofSolutionsOfKolmogorovEqn})
and
\begin{equation}
\mathbb{P}_{0\to0}^{T}\left[B^{h_{\delta}}\left(\left[t,T-t\right]\right)|B_{0}\left(\left[t,T-t\right]\right)\right]\to1\mbox{ as }t\to\infty,\mbox{ for }\delta>\frac{1}{2},\label{eq:BramsonEpsLarge}
\end{equation}
uniformly in $T$ (see Lemma 2.7 \cite{Bramson1ConvergenceofSolutionsOfKolmogorovEqn}).
Intuitively, \prettyref{eq:BramsonSandwichEpsSmall} and \prettyref{eq:BramsonEpsLarge}
indicate that when conditioned on $B_{0}\left(\left[t,T-t\right]\right)$
Brownian bridge stays close to $h_{0.5}$.  Also, they can be used
to give the following lower bound on the probability that Brownian
bridge manages to stay in a ``tube'' $\left[h_{1/2-c},h_{1/2+c}\right]$
for small $c$, which will be needed for the lower bound of \prettyref{eq:GWBarrierGammaDeltaNewNotation}.
For technical reasons related to how we later apply the result we
let the starting point of Brownian bridge deviate somewhat from $0$,
and require that it also stays above $-\frac{T}{10000}$ during the
initial time interval $\left[0,t\right]$.
\begin{lem}
For any $T>0$, $v\in\left(-1000,1000\right)$ and $\frac{T}{3}>t\ge c$
we have that
\begin{equation}
c\frac{t}{T-2t}\le\mathbb{P}_{v\to0}^{T}\left[B_{h_{0.499}}^{h_{0.501}}\left(\left[t,T-t\right]\right)\cap B_{-\frac{T}{10000}}\left(\left[0,t\right]\right)\right].\label{eq:TubeProbBB}
\end{equation}

\begin{proof}
Let $I=\left[t,T-t\right]$. We will show that
\begin{equation}
c\frac{t}{T-2t}\le\mathbb{P}_{0\to0}^{T}\left[B_{h_{0.4999}}^{h_{0.5001}}\left(I\right)\cap B_{-\frac{T}{20000}}\left(\left[0,t\right]\right)\right].\label{eq:TubeStartingFromZero}
\end{equation}
This implies \eqref{eq:TubeProbBB}, since even if we shift the process
and the barriers by $s\to v\frac{T-s}{T}$ (see \eqref{eq:ShiftBB})
we still have for $s\in I$ and $T$ and $t$ large enough
\[
-\frac{T}{10000}\le-\frac{T}{20000}+v\frac{T-s}{T}\mbox{ and},
\]
\[
h_{0.499}\left(s\right)\le h_{0.4999}\left(s\right)+v\frac{T-s}{T}\le h_{0.5001}\left(s\right)+v\frac{T-s}{T}\le h_{0.501}\left(s\right).
\]
From \eqref{eq:BBContLinearBarrier} we have that
\begin{equation}
\mathbb{P}_{0\to0}^{T}\left[B_{0}\left(I\right)\right]=\mathbb{E}_{0\to0}^{T}\left[\left(1-e^{-\frac{2X_{t}X_{T-t}}{T-2t}}\right)1_{\left\{ X_{t},X_{T-t}\ge0\right\} }\right].\label{eq:OnlyFBarrierProb-1-1}
\end{equation}
Since $1-e^{-x}\ge x/2$ for $x\in\left[0,1\right]$ we thus have
\begin{equation}
\mathbb{P}_{0\to0}^{T}\left[B_{0}\left(I\right)\right]\ge c\frac{t}{T-2t}\mathbb{P}_{0\to0}^{T}\left[\sqrt{t}\ge X_{t},X_{T-t}\ge\frac{1}{1000}\sqrt{t}\right]\ge c\frac{t}{T-2t},\label{eq:StayAboveZero}
\end{equation}
where in the last step we have used \prettyref{eq:VarOfBB} and $\mbox{Cov}\left[X_{t},X_{T-t}\right]=\frac{t^{2}}{T}>0$
(see \prettyref{eq:BrownianBridgeGaussianProc}).

Now since $X_{t}\ge0$ on $B_{0}\left(I\right)$ we have $\mathbb{P}_{0\to0}^{T}\left[B_{-\frac{1}{20000}T}\left(\left[0,t\right]\right)|B_{0}\left(I\right)\right]\ge\mathbb{P}_{0\to0}^{T}\left[B_{-\frac{1}{20000}T}\left(\left[0,t\right]\right)\right]$,
which equals $1-e^{-c\frac{T^{2}}{t}}$ by \prettyref{eq:BBContLinearBarrier}.
Thus
\begin{equation}
\mathbb{P}_{0\to0}^{T}\left[\left(B_{-\frac{1}{20000}T}\left(\left[0,t\right]\right)\right)^{c}|B_{0}\left(I\right)\right]\le\frac{1}{4}\mbox{ for }t\ge c.\label{eq:Un-1-1}
\end{equation}
Also by \eqref{eq:BramsonSandwichEpsSmall} and \eqref{eq:BramsonEpsLarge}
we have for $t\ge c$, 
\begin{equation}
\mathbb{P}_{0\to0}^{T}\left[\left(B_{0.4999}\left(I\right)\right)^{c}|B_{0}\left(I\right)\right]\le\frac{1}{4}\mbox{ and }\mathbb{P}_{0\to0}^{T}\left[\left(B^{0.5001}\left(I\right)\right)^{c}|B_{0}\left(I\right)\right]\le\frac{1}{4}\label{eq:Deux-1-1}
\end{equation}
Now using \eqref{eq:Un-1-1}, \eqref{eq:Deux-1-1} and a union bound
we have that 
\[
\mathbb{P}_{0\to0}^{T}\left[B_{-\frac{1}{20000}T}\left(\left[0,t\right]\right)\cap B_{h_{0.4999}}^{h_{0.5001}}\left(I\right)|B_{0}\left(I\right)\right]\ge\frac{3}{4},
\]
when $t\ge c$. Thus the claim \eqref{eq:TubeStartingFromZero} follows
from \eqref{eq:StayAboveZero}.
\end{proof}
\end{lem}
To use these results to prove \prettyref{prop:BarrierSecGWProp} we
must now compare the law of the conditioned Galton-Watson process
to the law of a Brownian bridge. As mentioned above, this will go
via squared Bessel bridges. Let us introduce the necessary notation.
We let $\mathbb{Q}_{x}^{d},d\ge0,x\ge0,$ be the law on $\left(C\left(\mathbb{R}_{+},\mathbb{R}\right),\mathcal{B}\left(\mathbb{R}_{+},\mathbb{R}\right)\right)$
which turns $X_{t},t\ge0,$ into a $d-$dimensional squared Bessel
processes starting at $x$ (see Chapter XI.1 \cite{RevuzYor-ContMartAndBM}).
Recall that these are non-negative processes. For $d\ge0,T>0,a,b\in\mathbb{R}$
we denote the law of a $d-$squared Bessel bridge on $\left(C_{0}\left(\left[0,T\right]\right),\mathcal{B}\left(\left[0,T\right],\mathbb{R}\right)\right)$
starting at $a$ and ending in $b$ by
\begin{equation}
\mathbb{Q}_{a\to b}^{d,T}\left[\cdot\right]\overset{\mbox{def}}{=}\mathbb{Q}_{a}^{d}\left[\cdot|X_{T}=b\right],\label{eq:BesselBridgeDef}
\end{equation}
(see Chapter XI.3 \cite{RevuzYor-ContMartAndBM}). We will need some
well-known facts about the Bessel bridge. For integer $d\ge1$ the
$d-$dimensional squared Bessel process is simply the norm squared
of $d-$dimensional Brownian motion (Chapter XI.1 \cite{RevuzYor-ContMartAndBM}).
In particular 
\begin{equation}
\mathbb{Q}_{a^{2}}^{1}\left[\cdot\right]=\mathbb{P}_{a}\left[\left(\left|X_{t}\right|^{2}\right)_{t\ge0}\in\cdot\right]\mbox{ for }a\ge0.\label{eq:Bessel1andBMSame}
\end{equation}
Because of this, a $1-$dimensional squared Bessel bridge ending in
zero is the norm squared of a Brownian bridge ending in zero, i.e.
for any $T>0$ and $a\ge0$
\begin{equation}
\mathbb{Q}_{a^{2}\to0}^{1,T}\left[\cdot\right]=\mathbb{P}_{a\to0}^{T}\left[\left|X_{\cdot}\right|^{2}\in\cdot\right].\label{eq:1BesselBridgeBBSame}
\end{equation}
The squared Bessel processes satisfy a well-known addivitiy property
(see Theorem 1.2, Chapter XI.1 \cite{RevuzYor-ContMartAndBM}):
\begin{equation}
\begin{array}{c}
\mbox{If }X_{\cdot}^{1}\mbox{ has law }\mathbb{Q}_{a_{1}}^{d_{1}}\mbox{ and }X_{\cdot}^{2}\mbox{ is independent with law }\mathbb{Q}_{a_{2}}^{d_{2}}\\
\mbox{then }X_{\cdot}^{1}+X_{\cdot}^{2}\mbox{ has law }\mathbb{Q}_{a_{1}+a_{2}}^{d_{1}+d_{2}},\mbox{ for all }d_{1},d_{2},a_{1},a_{2}>0.
\end{array}\label{eq:AdditiveBessel}
\end{equation}
A similar property holds for Bessel bridges (see $(1.b){}_{0}$ \cite{PitmanYorADecompositionofBesselBridges}).
We will use the following special case:
\begin{equation}
\begin{array}{c}
\mbox{If }X_{\cdot}^{1}\mbox{ has law }\mathbb{Q}_{x\to0}^{0,T}\mbox{ and }X_{\cdot}^{2}\mbox{ is independent with law }\mathbb{Q}_{0\to0}^{1,T}\\
\mbox{then }X_{\cdot}^{1}+X_{\cdot}^{2}\mbox{ has law }\mathbb{Q}_{x\to0}^{1,T},\mbox{ for all }T,x>0.
\end{array}\label{eq:AdditiveBesselBridge}
\end{equation}
For the $0-$dimensional Bessel bridge $0$ is an absorbing state
(see (5.3) \cite{PitmanYorADecompositionofBesselBridges}),
\begin{equation}
\mathbb{Q}_{x\to0}^{0,T}\left[X_{s}=0\mbox{\,\ for all }s\ge H_{0}\right]=1\mbox{ for all }x,T>0.\label{eq:ZeroAbsorbing}
\end{equation}
Finally for $0<S<T$ we can write down the Radon-Nikodym derivate
of the laws under $\mathbb{Q}_{x\to0}^{0,T}$ and $\mathbb{Q}_{x\to0}^{1,T}$
of $\left(X_{s}\right)_{s\le S}$ on the event $\left\{ H_{0}>S\right\} $.
\begin{lem}
\label{lem:RadonNikodymBesselBridge}For all $0<S<T$
\begin{equation}
\frac{d\mathbb{Q}_{x\to0}^{0,T}}{d\mathbb{Q}_{x\to0}^{1,T}}|_{\mathcal{F}_{S}\cap\left\{ H_{0}>S\right\} }=\left(\frac{\left(1-\frac{S}{T}\right)^{2}x}{X_{S}}\right)^{1/4}\exp\left(-\frac{3}{8}\int_{0}^{S}\frac{dt}{X_{t}}\right).\label{eq:RadonNikodymBesselBridge}
\end{equation}
\end{lem}
\begin{proof}
A basic property of Bessel bridges ending in zero is that they can
obtained from the Bessel process via (see (5.1) \cite{PitmanYorADecompositionofBesselBridges})
\begin{equation}
\mathbb{Q}_{x\to0}^{d,T}\mbox{ is the }\mathbb{Q}_{x}^{d}-\mbox{law of }\left(\left(1-\frac{t}{T}\right)^{2}X_{\frac{t}{1-t/T}}\right)_{0\le t\le T}.\label{eq:LawOfBridgeFromProc}
\end{equation}
Also  for all $R>0$
\begin{equation}
\frac{d\mathbb{Q}_{x}^{0}}{d\mathbb{Q}_{x}^{1}}|_{\mathcal{F}_{R}\cap\left\{ H_{0}>R\right\} }=\left(\frac{x}{X_{R}}\right)^{1/4}\exp\left(-\frac{3}{8}\int_{0}^{R}\frac{dt}{X_{t}}\right),\label{eq:BesselProcRadonNikodym}
\end{equation}
by the first lemma of Section 12 \cite{CarrSchroederBesselProcsIntegralOfGBMandAsianOptions}
(note that the index of a $d$-dimensional Bessel process is $d/2-1$).
The claim \eqref{eq:RadonNikodymBesselBridge} now follows from \eqref{eq:LawOfBridgeFromProc}
and \eqref{eq:BesselProcRadonNikodym} with $R=\frac{S}{1-S/T}$,
using the substitution $t^{'}=\frac{t}{1-t/T}$ in the integral.
\end{proof}
We are now ready to derive barrier bounds for the zero dimensional
Bessel bridge. We first derive upper bounds for squares of linear
barriers.
\begin{lem}
\label{lem:BesselBridgeUpperBounds}Let $T>0$ and $t_{1},t_{2}>0$
be integers such that $t_{1}<T-t_{2}$. For any linear function $h:\left[0,T\right]\to\mathbb{R}$
such $h\left(t\right)>0$ for $t\in\left[t_{1},T-t_{2}\right]$ and
any $u>0$
\begin{equation}
\begin{array}{l}
\mathbb{Q}_{u^{2}\to0}^{0,T}\left[B_{h^{2}}\left(\left\{ t_{1},\ldots,T-t_{2}\right\} \right)\right]\\
\le c\frac{\left(c+\sqrt{t_{1}}+\left|\bar{u}\left(t_{1}\right)-h\left(t_{1}\right)\right|\right)\left(c+\sqrt{t_{2}}+\left|\bar{u}\left(T-t_{2}\right)-h\left(T-t_{2}\right)\right|\right)}{T-t_{1}-t_{2}},
\end{array}\label{eq:BesselBridgeGeneralUpperbound}
\end{equation}
where $\bar{u}\left(t\right)=\frac{T-t}{T}u$. If also $v\ge h\left(T-t_{2}\right)$
then
\begin{equation}
\begin{array}{l}
\mathbb{Q}_{u^{2}\to0}^{0,T}\left[B_{h^{2}}\left(\left\{ t_{1},\ldots,T-t_{2}\right\} \right),\sqrt{X_{T-t_{2}}}\le v\right]\\
\le c\sqrt{\frac{\bar{u}\left(T-t_{2}\right)}{h\left(T-t_{2}\right)}}\frac{\left(c+\sqrt{t_{1}}+\left|\bar{u}\left(t_{1}\right)+v\frac{t_{1}}{T}-h\left(t_{1}\right)\right|\right)\left(c+v-h\left(T-t_{2}\right)\right)}{T-t_{1}-t_{2}}.
\end{array}\label{eq:BesselBridgeGeneralUpperBoundEndPoint}
\end{equation}
\end{lem}
\begin{proof}
Let $I=\left[t_{1},T-t_{2}\right]$. By \prettyref{eq:ZeroAbsorbing}
the event $B_{h^{2}}\left(I\cap\mathbb{N}\right)$ implies $\left\{ H_{0}>T-t_{2}\right\} $,
since $h>0$ throughout $I$ by assumption. Using this we obtain that
\[
\begin{array}{ccl}
\mathbb{Q}_{u^{2}\to0}^{0,T}\left[B_{h^{2}}\left(I\cap\mathbb{N}\right)\right] & = & \mathbb{Q}_{u^{2}\to0}^{0,T}\left[B_{h^{2}}\left(I\cap\mathbb{N}\right)\cap\left\{ H_{0}>T-t_{2}\right\} \right]\\
 & \overset{\eqref{eq:AdditiveBesselBridge}}{\le} & \mathbb{Q}_{u^{2}\to0}^{1,T}\left[B_{h^{2}}\left(I\cap\mathbb{N}\right)\cap\left\{ H_{0}>T-t_{2}\right\} \right]\\
 & \overset{\eqref{eq:1BesselBridgeBBSame}}{=} & \mathbb{P}_{u\to0}^{T}\left[B_{h}\left(I\cap\mathbb{N}\right)\cap\left\{ H_{0}>T-t_{2}\right\} \right],
\end{array}
\]
where the inequality holds because adding a process with law $\mathbb{Q}_{0\to0}^{1,T}$
to $X_{\cdot}$ only makes the barrier condition easier to satisify,
and the last equality holds because under $\mathbb{P}_{u\to0}^{T}$
we have $\left|X_{t}\right|=X_{t}$ for $t\in I$ on $\left\{ H_{0}>T-t_{2}\right\} $.
Using \eqref{eq:ShiftBB} we thus have that
\[
\mathbb{Q}_{u^{2}\to0}^{0,T}\left[B_{h^{2}}\left(I\cap\mathbb{N}\right)\right]\le\mathbb{P}_{u\to0}^{T}\left[B_{h}\left(I\cap\mathbb{N}\right)\right]=\mathbb{P}_{0\to0}^{T}\left[B_{h-\bar{u}}\left(I\cap\mathbb{N}\right)\right],
\]
and by \eqref{eq:DiscBarrierStartingLate} the right-hand side is
bounded above by the bottom line of \eqref{eq:BesselBridgeGeneralUpperbound}.

For \eqref{eq:BesselBridgeGeneralUpperBoundEndPoint} note that similarly
\prettyref{eq:ZeroAbsorbing} implies that the probability in question
equals
\[
\mathbb{Q}_{u^{2}\to0}^{0,T}\left[B_{h^{2}}\left(I\cap\mathbb{N}\right)\cap\left\{ H_{0}>T-t_{2},\sqrt{X_{T-t_{2}}}\le v\right\} \right].
\]
By \prettyref{lem:RadonNikodymBesselBridge} with $S=T-t_{2}$ this
is bounded above by
\begin{equation}
c\sqrt{\frac{\bar{u}\left(T-t_{2}\right)}{h\left(T-t_{2}\right)}}\mathbb{Q}_{u^{2}\to0}^{1,T}\left[B_{h^{2}}\left(I\cap\mathbb{N}\right)\cap\left\{ H_{0}>T-t_{2},\sqrt{X_{T-t_{2}}}\le v\right\} \right],\label{eq:Bessel1GammaDeltaUpToK}
\end{equation}
since on $B_{h^{2}}\left(I\cap\mathbb{N}\right)$ we have that $\left(\left(\left(1-\frac{T-t_{2}}{T}\right)^{2}u^{2}\right)/X_{T-t_{2}}\right)^{1/4}=\left(\bar{u}\left(T-t_{2}\right)/\sqrt{X_{T-t_{2}}}\right)^{1/2}\le\left(\bar{u}\left(T-t_{2}\right)/h\left(T-t_{2}\right)\right)^{1/2}$.
But by \eqref{eq:1BesselBridgeBBSame} the probability in \eqref{eq:Bessel1GammaDeltaUpToK}
equals
\[
\mathbb{P}_{u\to0}^{T}\left[B_{h}\left(I\cap\mathbb{N}\right),X_{T-t_{2}}\le v\right]\le\mathbb{P}_{u\to v}^{T}\left[B_{h}\left(I\cap\mathbb{N}\right)\right].
\]
Thus the required bound follows by \prettyref{eq:ShiftBB} and \eqref{eq:DiscBarrierStartingLate}.
\end{proof}
We now provide a lower bound on the probability that the zero dimensional
squared Bessel bridge stays in a tube, cf. \eqref{eq:TubeProbBB}.
\begin{lem}
\label{lem:BesselBridgeTube}If $T>0$, \textup{$\frac{T}{3}>t\ge c$,}
$u\ge\frac{1}{1000}T$ and $v\in\left(-1000,1000\right)$ 
\begin{equation}
c\frac{t}{T-2t}\le\mathbb{Q}_{\left(u+v\right)^{2}\to0}^{0,T}\left[B_{\left(\bar{u}+h_{0.499}\right)^{2}}^{\left(\bar{u}+h_{0.501}\right)^{2}}\left(\left[t,T-t\right]\right)\right],\mbox{ where }\bar{u}\left(t\right)=\frac{T-t}{T}u.\label{eq:BesselBridgeTube}
\end{equation}
\end{lem}
\begin{proof}
By \eqref{eq:ShiftBB}, \eqref{eq:TubeProbBB} and \eqref{eq:1BesselBridgeBBSame}
we have
\begin{equation}
c\frac{t}{L-2t}\le\mathbb{Q}_{\left(u+v\right)^{2}\to0}^{1,L}\left[B_{\left(\bar{u}+h_{0.409}\right)^{2}}^{\left(\bar{u}+h_{0.501}\right)^{2}}\left(\left[t,T-t\right]\right)\cap B_{\left(\bar{u}-\frac{1}{10000}T\right)^{2}}\left(\left[0,t\right]\right)\right].\label{eq:ettettettettett}
\end{equation}
By \prettyref{lem:RadonNikodymBesselBridge} with $S=T-t$ the right-hand
side of \eqref{eq:ettettettettett} is bounded above by
\begin{equation}
\mathbb{Q}_{\left(u+v\right)^{2}\to0}^{0,L}\left[A;B_{\left(\bar{u}+h_{0.499}\right)^{2}}^{\left(\bar{u}+h_{0.501}\right)^{2}}\left(\left[t,T-t\right]\right)\cap B_{\left(\bar{u}-\frac{1}{10000}T\right)^{2}}\left(\left[0,t\right]\right)\right]\label{eq:tvatvatvatvatva}
\end{equation}
for $A=\left(\left(1-\frac{T-t}{T}\right)^{2}u^{2}/X_{T-t}\right)^{1/4}\exp\left(-3/8\int_{0}^{T-t}X_{s}^{-1}ds\right)$.
On the event in \eqref{eq:tvatvatvatvatva} 
\[
\left(\frac{\left(1-\frac{T-t}{T}\right)^{2}u^{2}}{X_{T-t}}\right)^{1/4}\ge\left(\frac{\bar{u}\left(T-t\right)}{\bar{u}\left(T-t\right)+h_{0.501}\left(T-t\right)}\right)^{1/2}\ge c,
\]
provided $t\ge c$ (recall the assumption $u\ge\frac{1}{1000}T$),
and $\sqrt{X_{s}}\ge c\bar{u}\left(s\right)$ for $s\in\left[0,T-t\right]$
(note that $\bar{u}\left(s\right)-\frac{1}{10000}T\ge c\bar{u}\left(s\right)$
for $s\le t$) so that
\[
\int_{0}^{T-t}\frac{ds}{X_{s}}\le c\int_{0}^{T-t}\frac{ds}{\bar{u}\left(s\right)^{2}}\le c\left(\frac{T}{u}\right)^{2}\int_{0}^{T-t}\left(T-s\right)^{-2}ds\le c\int_{t}^{\infty}s^{-2}ds\le c.
\]
Thus $A\ge c$ on the event in \eqref{eq:tvatvatvatvatva}, so the
claim \eqref{eq:BesselBridgeTube} follows.
\end{proof}
It remains to derive our goal \prettyref{prop:BarrierSecGWProp} from
\prettyref{lem:BesselBridgeUpperBounds} and \prettyref{lem:BesselBridgeTube},
by comparing the law of $\left(T_{l}\right)_{l\ge0}$ under $\mathbb{G}_{t_{s}}\left[\cdot|T_{L-1}=0\right]$
and $\left(X_{t}\right)_{t\ge0}$ under $\mathbb{Q}_{t_{s}\to0}^{0,L}$.
To do this we exploit that that $\mathbb{G}_{t_{s}}\left[\cdot|T_{L-1}=0\right]$
is essentially the law of the edge local time of the discrete simple
random walk on $\left\{ 0,\ldots,L\right\} $ when conditioned not
to hit $L$, while $\mathbb{Q}_{t_{s}\to0}^{0,L}$ is the law of the
vertex local time of the continuous time version of the same random
walk. We can carry out the comparison using the natural coupling of
discerete and continuous time random walk on $\left\{ 0,\ldots,L\right\} $.

To this end, let $Y_{t},t\ge0,$ be continuous time simple random
walk on $\left\{ 0,\ldots,L\right\} $ with jump rate $1$, and let
$\tilde{\mathbb{P}}_{l}$ be its law when starting from $l\in\left\{ 0,\ldots,L\right\} $.
Let
\begin{equation}
d_{l}=\begin{cases}
1 & \mbox{ if }l=0,\\
2 & \mbox{ if }0<l<L,\\
1 & \mbox{\,\ if }l=L,
\end{cases}\label{eq:Normalization}
\end{equation}
be the degree of the vertices in the path $\left\{ 0,\ldots,L\right\} $
and let
\begin{equation}
L_{l}^{t}=\frac{1}{d_{l}}\int_{0}^{t}1_{\left\{ Y_{s}=l\right\} }ds\mbox{\,\ for }0\le l\le L,t\ge0,\label{eq:ContTimeLocTimeDef}
\end{equation}
be the local time of the random walk $Y_{t}$. Define the inverse
local time of $0$ by
\[
\tau\left(t\right)=\inf\left\{ s\ge0:L_{0}^{s}>t\right\} .
\]
The law of $L_{l}^{\tau\left(t\right)}$ has a nice characterisation
which can be derived from the Second Ray Knight Theorem (see the appendix
for the derivation).
\begin{lem}
\label{lem:RayKnightCont}For all $L\in\left\{ 1,2,\ldots\right\} $
and $t\ge0$ the $\tilde{\mathbb{P}}_{0}-$law of $\left(L_{l}^{\tau\left(t\right)}\right)_{l\in\left\{ 0,\ldots,L\right\} }$
is the $\mathbb{Q}_{2t}^{0}-$law of $\left(\frac{1}{2}X_{l}\right)_{l\in\left\{ 0,\ldots L\right\} }$.
\end{lem}
Note that $L_{l}^{\tau\left(t\right)}$ counts the local time accumlated
at vertex $l$ until local time $t$ has accumulated at $0$. In proving
\prettyref{prop:BarrierSecGWProp} we will in fact need the law of
the local times accumlated during the first $t$ excursions from zero,
that is of $L_{l}^{D_{t}}$ where, 
\begin{equation}
D_{0}=0,D_{1}=H_{0}\circ\theta_{H_{1}}+H_{1}\mbox{\,\ and }D_{n}=D_{1}\circ\theta_{D_{n-1}}+D_{n-1},n\ge2,\label{eq:DeparturesFromZero}
\end{equation}
are the return times to $0$ of $Y_{t}$ (and $H_{l}$ and $\theta_{n}$
are defined on the space $C\left(\mathbb{R}_{+},\left\{ 0,\ldots,L\right\} \right)$
in the natural manner). Next we state a description of the law of
$L_{l}^{D_{t}}$ that follows from \prettyref{lem:RayKnightCont},
where we also condition on the processes hitting zero, since this
is what we do in \prettyref{prop:BarrierSecGWProp}. The proof is
given in the appendix.
\begin{lem}
\label{lem:RayishKnightish}For all $L\in\left\{ 1,2,\ldots\right\} $,
measurable $A\subset\mathbb{R}^{L}$ and $t\in\left\{ 1,2,\ldots\right\} $
\[
\tilde{\mathbb{P}}_{0}\left[\left(L_{l}^{D_{t}}\right)_{l\in\left\{ 1,\ldots,L\right\} }\in A|L_{L}^{D_{t}}=0\right]=\tilde{\mathbb{P}}_{0}\left[\mathbb{Q}_{2L_{1}^{D_{t}}\to0}^{0,L-1}\left[\left(\frac{1}{2}X_{l}\right)_{l\in\left\{ 0,\ldots,L-1\right\} }\in A\right]|L_{L}^{D_{t}}=0\right].
\]

\end{lem}
After applying \prettyref{lem:RayishKnightish} we will need a control
on $L_{1}^{D_{t}}$ conditioned on $L_{L}^{D_{t}}=0$ provided by
the following lemma, whose proof is also in the appendix.
\begin{lem}
\label{lem:LOneGivenLLisZero}For all $L\in\left\{ 1,2,\ldots\right\} $
and $t\in\left\{ 0,1,\ldots,10L^{2}\right\} $
\begin{equation}
\tilde{\mathbb{P}}_{0}\left[\sqrt{t}\frac{L-1}{L}-250\le\sqrt{L_{1}^{D_{t}}}\le\sqrt{t}\frac{L-1}{L}+250|L_{L}^{D_{t}}=0\right]\ge c>0.\label{eq:LowerBoundOnL1Prob}
\end{equation}
If $l\in\left\{ 1,\ldots,L-1\right\} $ then with $\tilde{u}\left(l\right)=\sqrt{2L_{L}^{D_{t}}}\frac{L-l}{L-1}$
and $u\left(l\right)=\sqrt{2t}\frac{L-l}{L}$ we have
\begin{equation}
\tilde{\mathbb{E}}_{0}\left[\left|\tilde{u}\left(l\right)-v\right|^{k}|L_{L}^{D_{t}}=0\right]\le c+\left|u\left(l\right)-v\right|^{k}\mbox{\,\ for }k\in\left\{ 1,2\right\} \mbox{ and }v\in\mathbb{R}.\label{eq:ExpectedDeviation}
\end{equation}

\end{lem}
Next we exhibit the connection with the law $\mathbb{G}_{t_{s}}$.
Let $J_{1},J_{2},...$ be the jump times of $Y_{t}$, and let $J_{0}=0$.
Let 
\[
Z_{n}=Y_{J_{n}},n\ge0,
\]
be the discrete skeleton of the random walk $Y_{t}$. Clearly $Z_{n}$
is a discrete time simple random walk. Let
\[
\tilde{D}_{n}=\inf\left\{ m>\tilde{D}_{n-1}:Z_{m}=0\right\} ,n\ge1,\mbox{ and }\tilde{D}_{0}=0,
\]
be the successive returns to $0$ of $Z_{n}$. Finally let
\begin{equation}
\tilde{T}_{l}^{t}=\sum_{m=1}^{\tilde{D}_{\lfloor t\rfloor}}1_{\left\{ Z_{m}=l+1,Z_{m-1}=l\right\} },l=0,\ldots,L-1,t\ge0,\label{eq:ZTraversals}
\end{equation}
be the number of traversals from $l$ to $l+1$ up to time $\tilde{D}_{\lfloor t\rfloor}$
(equivalently the edge local times of the edges $l\to l+1$ up to
time $\tilde{D}_{\lfloor t\rfloor}$). We have that
\begin{lem}
\label{lem:RayKnightDisc-1}For all $t\in\left\{ 0,1,\ldots\right\} $
the $\tilde{\mathbb{P}}_{0}-$law of $\tilde{T}_{l}^{t},l\in\left\{ 0,\ldots,L-1\right\} ,$
is the $\mathbb{G}_{t}-$law of $T_{l},l\in\left\{ 0,\ldots,L-1\right\} $.\end{lem}
\begin{proof}
The proof is omitted as it is very is similar to that of \prettyref{lem:BranchingProc}.
\end{proof}
To derive \prettyref{prop:BarrierSecGWProp} from \prettyref{lem:BesselBridgeUpperBounds}
and \prettyref{lem:BesselBridgeTube} we will have to ``translate''
between discrete and continuous local time. For this we will use the
following lemma, which gives a large deviation bound for $L_{l}^{D_{t}}$
conditioned on $\tilde{T}_{l}^{t},l\ge0$.
\begin{lem}
If $l\in\left\{ 1,\ldots,L-1\right\} ,t\ge0$ and $\theta>0$ then
with $\mu=\frac{1}{2}\left(\tilde{T}_{l-1}^{t}+\tilde{T}_{l}^{t}\right),$
\begin{equation}
\tilde{\mathbb{P}}_{0}\left[\left|L_{l}^{D_{t}}-\mu\right|\ge\theta|\sigma\left(\tilde{T}_{l}^{t}:l=0,\ldots,L-1\right)\right]\le ce^{-c\frac{\theta^{2}}{\mu}}.\label{eq:LargeDeviationConditionedLocTime}
\end{equation}
\end{lem}
\begin{proof}
The continuous time random walk $Y_{t}$ makes $\tilde{T}_{l-1}^{t}+\tilde{T}_{l}^{t}$
discrete visits to the vertex $l$ up to time $D_{t}$. The holding
times of the continuous time random walk $Y_{t}$ are iid standard
exponential random variables and are independent of the discrete skeleton
of the random walk, so we have that the $\tilde{\mathbb{P}}_{0}\left[\cdot|\sigma\left(\tilde{T}_{l}^{t}:l=0,\ldots,L-1\right)\right]-$law
of $L_{l}^{D_{t}}$ is that of a sum of $\tilde{T}_{l-1}^{t}+\tilde{T}_{l}^{t}$
iid exponentials with mean $1/2$ (from the normalizing factor in
\prettyref{eq:ContTimeLocTimeDef}). Thus the claim follows by a standard
large deviation bound.
\end{proof}
We now state a similar result for $\tilde{T}_{l}^{t}$ when conditioned
on $L_{l}^{D_{t}},l=0,\ldots,L,$ whose proof will be given in the
appendix.
\begin{lem}
\label{lem:LDConditionedTraverals}If $l\in\left\{ 1,\ldots,L-1\right\} ,t>0,$
and $\theta>0$ then with $\mu=\sqrt{L_{l}^{D_{t}}L_{l+1}^{D_{t}}}$,
\begin{equation}
\tilde{\mathbb{P}}_{0}\left[\left|\tilde{T}_{l}^{t}-\mu\right|\ge\theta|\sigma\left(L_{l}^{D_{t}}:l=0,\ldots,L\right)\right]\le ce^{-c\frac{\theta^{2}}{\mu}+c\frac{\theta}{u}}.\label{eq:LargeDeviationConditionedTraversal}
\end{equation}

\end{lem}
We are now ready to the main result \prettyref{prop:BarrierSecGWProp}
of this section.
\begin{proof}[Proof of \prettyref{prop:BarrierSecGWProp}]
We start with the proof of \prettyref{eq:GWBarrierAlphaNewNotation}.
Let $I=\left\{ 1,\ldots,L-\lceil3\left(\log L\right)^{2}\rceil\right\} $.
By \prettyref{lem:RayKnightDisc-1} we have
\begin{equation}
\mathbb{G}_{t_{s}}\left[B_{\alpha^{2}}\left(I\right)|T_{L-1}=0\right]=\tilde{\mathbb{P}}\left[A|\tilde{T}_{L-1}^{t_{s}}=0\right],\mbox{ where},\label{eq:ApplicationOfDiscRK}
\end{equation}
\[
A=\left\{ \sqrt{\tilde{T}_{l}^{\lfloor t_{s}\rfloor}}\ge\alpha\left(l\right)\mbox{ for }l\in I\right\} .
\]
Define also
\[
B=\left\{ \sqrt{L_{l}^{D_{\lfloor t_{s}\rfloor}}}\ge\alpha_{-}\left(l\right)\mbox{ for }l\in I\right\} ,\mbox{ where},
\]
\begin{equation}
\alpha_{-}\left(l\right)=\beta\left(l\right)-2\left(\log L\right)^{2}=\alpha\left(l\right)-\left(\log L\right)^{2},l\in\left\{ 0,1,\ldots,L\right\} .\label{eq:AlphaMinusDef}
\end{equation}
Letting $\mathcal{A}=\sigma\left(\tilde{T}_{l}^{t_{s}}:l=0,\ldots,L-1\right)$
we have by \prettyref{eq:LargeDeviationConditionedLocTime} that for
$l\in\left\{ 1,2,\ldots,L-1\right\} $
\begin{equation}
\tilde{\mathbb{P}}_{0}\left[\left|L_{l}^{D_{\lfloor t_{s}\rfloor}}-\mu_{l}\right|\ge\sqrt{\mu_{l}}\left(\log L\right)^{2}|\mathcal{A}\right]\le ce^{-c\left(\log L\right)^{4}},\mbox{ for }\mu_{l}=\frac{\tilde{T}_{l-1}^{\lfloor t_{s}\rfloor}+\tilde{T}_{l}^{\lfloor t_{s}\rfloor}}{2}.\label{eq:LDApllic}
\end{equation}
On the event $A$ we have for $l\in I$
\[
\begin{array}{ccl}
\mu_{l}-\sqrt{\mu_{l}}\left(\log L\right)^{2} & = & \left(\sqrt{\mu_{l}}-\frac{1}{2}\left(\log L\right)^{2}\right)^{2}-\frac{1}{4}\left(\log L\right)^{4}\\
 & \ge & \left(\alpha\left(l\right)-\frac{1}{2}\left(\log L\right)^{2}\right)^{2}-\frac{1}{4}\left(\log L\right)^{4}\ge\alpha_{-}\left(l\right)^{2},
\end{array}
\]
where we have used that $\left(\log L\right)^{2}\le\alpha\left(l\right)\le\alpha\left(l-1\right)$
for $l\in I$ (see \prettyref{eq:AlphaBarrierDef}, \prettyref{fig:Illustration-of-barriers}).
Therefore \eqref{eq:LDApllic} implies that 
\[
\tilde{\mathbb{P}}_{0}\left[B|A\cap\left\{ \tilde{T}_{L-1}^{t_{s}}=0\right\} \right]\ge1-\sum_{l\in I}e^{-c\left(\log L\right)^{4}}\ge1-o\left(1\right),
\]
so that
\[
\tilde{\mathbb{P}}_{0}\left[A\cap\left\{ \tilde{T}_{L-1}^{t_{s}}=0\right\} \right]\le c\tilde{\mathbb{P}}_{0}\left[A\cap B\cap\left\{ \tilde{T}_{L-1}^{t_{s}}=0\right\} \right]\le c\tilde{\mathbb{P}}_{0}\left[B\cap\left\{ L_{L}^{D_{\lfloor t_{s}\rfloor}}=0\right\} \right],
\]
(note that $\left\{ \tilde{T}_{L-1}^{t_{s}}=0\right\} =\left\{ L_{L}^{D_{\lfloor t_{s}\rfloor}}=0\right\} $
by construction). Therefore we obtain that
\begin{equation}
\tilde{\mathbb{P}}_{0}\left[A|\tilde{T}_{L-1}^{t_{s}}=0\right]\le c\tilde{\mathbb{P}}_{0}\left[B|L_{L}^{D_{\lfloor t_{s}\rfloor}}=0\right].\label{eq:HarDuSettSatt}
\end{equation}
Using \prettyref{lem:RayishKnightish} the right-hand side equals
\begin{equation}
\tilde{\mathbb{E}}_{0}\left[\mathbb{Q}_{2L_{1}^{D_{\lfloor t_{s}\rfloor}}\to0}^{0,L-1}\left[B_{2\alpha_{-}\left(1+\cdot\right)^{2}}\left(I-1\right)\right]|L_{L}^{D_{\lfloor t_{s}\rfloor}}=0\right],\label{eq:ApplyRayishKnightish}
\end{equation}
where $I-1=\left\{ 0,\ldots,L-\lceil3\left(\log L\right)^{2}\rceil-1\right\} $.
By the Bessel bridge barrier bound \prettyref{eq:BesselBridgeGeneralUpperbound}
with $T=L-1$, $t_{1}=0$, $t_{2}=\lceil3\left(\log L\right)^{2}\rceil$,
$h\left(\cdot\right)=\sqrt{2}\alpha_{-}\left(1+\cdot\right)$ (which
is positive on $\left[0,T-t_{2}\right]$) and $u=\sqrt{2L_{1}^{D_{\lfloor t_{s}\rfloor}}}$
the quantity in the expectation is bounded above by
\[
\frac{\left(c+\left|\tilde{u}\left(1\right)-\sqrt{2}\alpha_{-}\left(1\right)\right|\right)\left(c+c\log L+\left|\tilde{u}\left(L-\lceil3\left(\log L\right)^{2}\rceil\right)-\sqrt{2}\alpha_{-}\left(L-\lceil3\left(\log L\right)^{2}\rceil\right)\right|\right)}{L-1-\lceil3\left(\log L\right)^{2}\rceil},
\]
where $\tilde{u}\left(l\right)=\bar{u}\left(l-1\right)=\sqrt{2L_{L}^{D_{\lfloor t_{s}\rfloor}}}\frac{L-l}{L-1}$.
Therefore using \prettyref{eq:ExpectedDeviation} and Hölder's inequality
(note that $\sqrt{2}\beta\left(l\right)=\sqrt{2\lfloor t_{s}\rfloor}\frac{L-l}{L}+O\left(1\right)$)
\prettyref{eq:ApplyRayishKnightish} is bounded above by
\[
c\frac{\left(c+\sqrt{2}\left|\beta\left(1\right)-\alpha_{-}\left(1\right)\right|\right)\left(c+c\log L+\sqrt{2}\left|\beta\left(L-\lceil3\left(\log L\right)^{2}\rceil\right)-\alpha_{-}\left(L-\lceil3\left(\log L\right)^{2}\rceil\right)\right|\right)}{L}.
\]
Thus by \prettyref{eq:AlphaMinusDef} in fact
\[
\tilde{\mathbb{P}}_{0}\left[A|\tilde{T}_{L-1}^{t_{s}}=0\right]\le c\frac{\left(c+c\left(\log L\right)^{2}\right)^{2}}{L}.
\]
Now \prettyref{eq:GWBarrierAlphaNewNotation} follows by \prettyref{eq:ApplicationOfDiscRK}.

The proof of the upper bound of \prettyref{eq:GWBarrierGammaDeltaNewNotation}
and \prettyref{eq:GWBarrierGammaDeltaFromkNewNotaiton} are similar.
For the upper bound of \prettyref{eq:GWBarrierGammaDeltaNewNotation}
we let
\[
I_{1}=\left\{ l_{0},\ldots,L-l_{0}\right\} \mbox{\,\ and }I_{2}=\left\{ l_{0}+1,\ldots,L-l_{0}\right\} ,
\]
and note that similarly to \prettyref{eq:LDApllic} the large deviation
bound \prettyref{eq:LargeDeviationConditionedLocTime} implies that,
\begin{equation}
\tilde{\mathbb{P}}_{0}\left[\left|L_{l}^{D_{\lfloor t_{s}\rfloor}}-\mu_{l}\right|\ge\frac{1}{2}\sqrt{\mu_{l}}f\left(l\right)|\mathcal{A}\right]\le ce^{-cf\left(l\right)^{2}},\label{eq:LDApllic-2}
\end{equation}
where $\mu_{l}$ is as in \prettyref{eq:LDApllic}. On the event $\left\{ \sqrt{\tilde{T}_{l}^{\lfloor t_{s}\rfloor}}\ge\gamma\left(l\right)\mbox{ for }l\in I_{1}\right\} $
we have for $l\in I_{2}$
\[
\mu_{l}-\frac{1}{2}\sqrt{\mu_{l}}f\left(l\right)=\left(\sqrt{\mu_{l}}-\frac{1}{4}f\left(l\right)\right)^{2}-\frac{1}{4}f\left(l\right)^{2}\ge\left(\gamma\left(l\right)-\frac{1}{4}f\left(l\right)\right)^{2}-\frac{1}{4}f\left(l\right)^{2}\ge\beta\left(l\right)^{2},
\]
for $L$ large enough, where we have used that $f\left(l\right)\le\gamma\left(l\right)\le\gamma\left(l-1\right)$
(see \prettyref{eq:GammaBarrierDef}, \prettyref{fig:Illustration-of-barriers}).
Therefore the argument which gave \prettyref{eq:HarDuSettSatt} now
gives 
\[
\begin{array}{l}
\tilde{\mathbb{P}}_{0}\left[\sqrt{\tilde{T}_{l}^{t_{s}}}\ge\gamma\left(l\right)\mbox{ for }l\in I_{1}|\tilde{T}_{L-1}^{t_{s}}=0\right]\\
\le\left(1-\sum_{l\in I_{2}}e^{-cf\left(l\right)^{2}}\right)^{-1}\tilde{\mathbb{P}}_{0}\left[\sqrt{L_{l}^{D_{\lfloor t_{s}\rfloor}}}\ge\beta\left(l\right)\mbox{ for }l\in I_{2}|L_{L}^{D_{\lfloor t_{s}\rfloor}}=0\right].
\end{array}
\]
We have $\sum_{l\in I_{2}}e^{-cf\left(l\right)^{2}}=o\left(1\right)$,
so that using \prettyref{lem:RayishKnightish} the bottom line equals
\begin{equation}
\left(1+o\left(1\right)\right)\tilde{\mathbb{P}}_{0}\left[\mathbb{Q}_{2L_{1}^{D_{\lfloor t_{s}\rfloor}}\to0}^{0,L-1}\left[B_{2\beta\left(1+\cdot\right)^{2}}\left(I_{2}-1\right)\right]|L_{L}^{D_{\lfloor t_{s}\rfloor}}=0\right].\label{eq:ApplyRayishKnigtish2}
\end{equation}
By \prettyref{eq:BesselBridgeGeneralUpperbound} with $T=L-1$, $t_{1}=t_{2}=l_{0}$,
$h\left(\cdot\right)=\sqrt{2}\beta\left(1+\cdot\right)$ and $u=\sqrt{2L_{1}^{D_{\lfloor t_{s}\rfloor}}}$
 the quantity in the expectation is bounded above by
\[
\frac{\left(c+c\sqrt{l_{0}}+\left|\tilde{u}\left(l_{0}+1\right)-\sqrt{2}\beta\left(l_{0}+1\right)\right|\right)\left(c+\sqrt{l_{0}}+\left|\tilde{u}\left(L-l_{0}\right)-\sqrt{2}\beta\left(L-l_{0}\right)\right|\right)}{L-1-\lceil3\left(\log L\right)^{2}\rceil},
\]
where $\tilde{u}\left(l\right)=\bar{u}\left(l-1\right)=\sqrt{2L_{L}^{D_{\lfloor t_{s}\rfloor}}}\frac{L-l}{L-1}$.
Using \prettyref{eq:ExpectedDeviation} and Hölder's inequality we
get that
\[
\tilde{\mathbb{P}}_{0}\left[\sqrt{\tilde{T}_{l}^{t_{s}}}\ge\gamma\left(l\right)\mbox{ for }l\in I_{1}|\tilde{T}_{L-1}^{t_{s}}=0\right]\le c\frac{\left(c+\sqrt{l_{0}}\right)^{2}}{L},
\]
since $\sqrt{2\lfloor t_{s}\rfloor}\frac{L-l}{L}=\sqrt{2}\beta\left(l\right)+O\left(1\right)$.
Thus \prettyref{eq:GWBarrierGammaDeltaNewNotation} follows by \prettyref{lem:RayKnightDisc-1}.

To show \prettyref{eq:GWBarrierGammaDeltaFromkNewNotaiton} we similarily
use \prettyref{eq:LargeDeviationConditionedLocTime} and \prettyref{lem:RayKnightDisc-1}
to prove that
\begin{equation}
\begin{array}{l}
\mathbb{G}_{a}\left[B_{\gamma\left(k+\cdot\right)^{2}}\left(\left\{ 0,\ldots,L-k-l_{0}\right\} \right)|T_{L-1-k}=0\right]\\
=\mathbb{G}_{a}\left[\gamma\left(l+k\right)\le\sqrt{T_{l}}\mbox{ for }l=0,\ldots,L-k-l_{0}|T_{L-k-1}=0\right]\\
\le\left(1-\sum_{l=1}^{L-k-l_{0}}e^{-cf\left(l+k\right)^{2}}\right)^{-1}\\
\quad\quad\times\tilde{\mathbb{P}}_{0}\left[\sqrt{L_{l}^{D_{a}}}\ge\beta\left(l+k\right)\mbox{ for }l=1,\ldots,L-k-l_{0}|L_{L-k}^{D_{a}}=0\right]\\
=\left(1+o\left(1\right)\right)\tilde{\mathbb{P}}_{0}\left[\sqrt{L_{l}^{D_{a}}}\ge\beta\left(l+k\right)\mbox{ for }l=1,\ldots,L-k-l_{0}|L_{L-k}^{D_{a}}=0\right].
\end{array}\label{eq:HarDuSettSatt2}
\end{equation}
By \prettyref{lem:RayishKnightish} with $L-k-1$ in place of $L$
the last probability equals
\[
\tilde{\mathbb{P}}_{0}\left[\mathbb{Q}_{2L_{1}^{D_{a}}\to0}^{0,L-k-1}\left[B_{2\beta\left(k+1+\cdot\right)^{2}}\left(\left\{ 0,\ldots,L-k-l_{0}-1\right\} \right)\right]|L_{L-k-1}^{D_{a}}=0\right],
\]
so that by \prettyref{eq:BesselBridgeGeneralUpperbound} with $T=L-k-1$,
$t_{1}=0$, $t_{2}=l_{0}$, $h\left(\cdot\right)=\beta\left(k+1+\cdot\right)$
and $u=\sqrt{2L_{1}^{D_{a}}}$ the right-hand side of \eqref{eq:HarDuSettSatt2}
is bounded above by
\[
\tilde{\mathbb{E}}_{0}\left[\frac{\left(c+\left|\tilde{u}\left(1\right)-\beta\left(k+1\right)\right|\right)\left(c+\sqrt{l_{0}}+\left|\tilde{u}\left(L-k-l_{0}\right)-\beta\left(L-l_{0}\right)\right|\right)}{L-k-1-l_{0}}|L_{L-k-1}^{D_{a}}=0\right]
\]
where $\tilde{u}\left(l\right)=\bar{u}\left(l-1\right)=\sqrt{2L_{1}^{D_{a}}}\frac{L-k-l}{L-k-1}$.
Using \prettyref{eq:ExpectedDeviation} with $L-k$ in place of $L$
and Hölder's inequality this is at most
\[
c\frac{\left(c+\left|u\left(1\right)-\beta\left(k+1\right)\right|\right)\left(c+\sqrt{l_{0}}+\left|u\left(L-k-l_{0}\right)-\beta\left(L-l_{0}\right)\right|\right)}{L-k-l_{0}-1},
\]
where $u\left(l\right)=\sqrt{a}\frac{L-k-l}{L-k}$. Now since $\gamma\left(k\right)^{2}\le a\le\delta\left(k\right)^{2}$
we have we have that $\beta\left(k+l\right)\le u\left(l\right)\le\beta\left(k+l\right)+g\left(k+l\right)$,
so that for $L\ge c$ this is at most
\[
c\frac{\left(c+g\left(k+1\right)\right)\left(c+\sqrt{l_{0}}+g\left(L-l_{0}\right)\right)}{L-k-l_{0}-1}\le c\frac{g\left(k+1\right)l_{0}^{0.51}}{L-k-l_{0}-1},
\]
(recall \prettyref{eq:DefOfgfunc}), so \prettyref{eq:GWBarrierGammaDeltaFromkNewNotaiton}
follows.

To show \prettyref{eq:GWBarrierGammaDeltaUptoKNewNotation} we similarily
use \prettyref{eq:LargeDeviationConditionedLocTime} prove that
\[
\begin{array}{l}
\mathbb{G}_{t_{s}}\left[\gamma\left(l\right)\le\sqrt{T_{l}}\le\delta\left(l\right)\mbox{ for }l=l_{0},\ldots,k|T_{L-1}=0\right]\\
\le c\tilde{\mathbb{P}}_{0}\left[\sqrt{L_{l}^{D_{\lfloor t_{s}\rfloor}}}\ge\beta\left(l\right)\mbox{ for }l=l_{0},\ldots,k,\sqrt{L_{k}^{D_{\lfloor t_{s}\rfloor}}}\le\delta\left(l\right)+g\left(k\right)|L_{L}^{D_{\lfloor t_{s}\rfloor}}=0\right].
\end{array}
\]
For this we use also that (cf. \prettyref{eq:LDApllic})
\[
\tilde{\mathbb{P}}_{0}\left[L_{k}^{D_{\lfloor t_{s}\rfloor}}\le\mu_{k}+\sqrt{\mu_{k}}\frac{1}{2}g\left(k\right)|\mathcal{A}\right]\le ce^{-g\left(k\right)^{2}}\to0,\mbox{ as }L\to\infty,
\]
and that on the event $\left\{ \gamma\left(l\right)\le\sqrt{\tilde{T}_{l}^{t_{s}}}\le\delta\left(l\right)\mbox{\,\ for }l\in\left\{ l_{0},\ldots,L-l_{0}\right\} \right\} $
we have 
\[
\mu_{k}+\sqrt{\mu_{k}}\frac{1}{2}g\left(k\right)\le\left(\delta\left(k-1\right)+\frac{1}{4}g\left(k\right)\right)^{2}\le\left(\delta\left(k\right)+g\left(k\right)\right)^{2},\mbox{\,\ see }\eqref{eq:DeltaBarrierDef}.
\]
We then use \prettyref{lem:RayishKnightish} to obtain that
\[
\begin{array}{l}
\mathbb{G}_{t_{s}}\left[\gamma\left(l\right)\le\sqrt{T_{l}}\le\delta\left(l\right)\mbox{ for }l=l_{0},\ldots,k|T_{L-1}=0\right]\le\\
c\tilde{\mathbb{E}}_{0}\left[\mathbb{Q}_{2L_{1}^{D_{\lfloor t_{s}\rfloor}}\to0}^{0,L-1}\left[B_{2\beta\left(1+\cdot\right)^{2}}\left(\left\{ l_{0},\ldots,k-1\right\} \right),\sqrt{X_{k-1}}\le\sqrt{2}\left(\delta\left(k\right)+g\left(k\right)\right)\right]|L_{L-1}^{D_{\lfloor t_{s}\rfloor}}=0\right].
\end{array}
\]
By \prettyref{eq:BesselBridgeGeneralUpperBoundEndPoint} with $T=L-1$,
$t_{1}=l_{0}$, $t_{2}=L-k$ and $h=\sqrt{2}\beta\left(1+\cdot\right)$
this is bounded above by
\[
c\tilde{\mathbb{E}}_{0}\left[\sqrt{\frac{\tilde{u}\left(k\right)}{\beta\left(k\right)}}\frac{\left(c+\sqrt{l_{0}}+\left|\tilde{u}\left(l_{0}+1\right)-\beta\left(l_{0}+1\right)\right|\right)\left(c+cg\left(k\right)\right)}{k-l_{0}-1}|L_{L}^{D_{\lfloor t_{s}\rfloor}}=0\right],
\]
where $\tilde{u}\left(l\right)=\bar{u}\left(l-1\right)=\sqrt{2L_{1}^{D_{a}}}\frac{L-l}{L-1}$.
By \prettyref{eq:ExpectedDeviation} and the Hölder inequality this
bounded above by
\[
\frac{c\frac{\sqrt{2\lfloor t_{s}\rfloor}\left(1-k/L\right)}{\beta\left(k\right)}\left(c+\sqrt{l_{0}}\right)\left(c+cg\left(k\right)\right)}{k-l_{0}-1},
\]
which is bounded above by the right-hand side of \prettyref{eq:GWBarrierGammaDeltaUptoKNewNotation},
so \prettyref{eq:GWBarrierGammaDeltaUptoKNewNotation} follows.

It remains to show the lower bound of \prettyref{eq:GWBarrierGammaDeltaNewNotation}.
For this we note that by \prettyref{lem:RayKnightDisc-1},
\begin{equation}
\mathbb{G}_{t_{s}}\left[B_{\gamma^{2}}^{\delta^{2}}\left(I_{1}\right)|T_{L-1}=0\right]=\tilde{\mathbb{P}}_{0}\left[A|\tilde{T}_{L-1}^{t_{s}}=0\right],\label{eq:ApplicationOfDiscRK-1}
\end{equation}
where $I_{1}=\left\{ l_{0},\ldots,L-l_{0}\right\} $ and, 
\[
A=\left\{ \gamma\left(l\right)\le\sqrt{\tilde{T}_{l}^{t_{s}}}\le\delta\left(l\right)\mbox{ for }l\in I_{1}\right\} .
\]
Define also $I_{2}=\left\{ \lfloor\frac{1}{2}l_{0}\rfloor,\ldots,L-\lfloor\frac{1}{2}l_{0}\rfloor\right\} $
and
\[
B=\left\{ \beta\left(l\right)+2f\left(l\right)\le\sqrt{\tilde{L}_{l}^{D_{t_{s}}}}\le\beta\left(l\right)+\frac{1}{2}g\left(l\right)\mbox{ for }l\in I_{2}\right\} .
\]
By \prettyref{eq:LargeDeviationConditionedTraversal} we have, letting
$\mathcal{A}=\sigma\left(L_{l}^{D_{\lfloor t_{s}\rfloor}}:l=0,\ldots,L-1\right)$,
that
\begin{equation}
\tilde{\mathbb{P}}_{0}\left[\left|\tilde{T}_{l}^{t_{s}}-\mu_{l}\right|\ge\sqrt{\mu_{l}}\frac{1}{2}f\left(l\right)|\mathcal{A}\right]\le ce^{-cf\left(l\right)^{2}},\label{eq:LDApllic-1}
\end{equation}
where $\mu_{l}=\sqrt{L_{l}^{D_{\lfloor t_{s}\rfloor}}L_{l+1}^{D_{\lfloor t_{s}\rfloor}}}$.
On the event $B$ we have for $l\in I_{2}$
\[
\begin{array}{ccl}
\mu_{l}-\sqrt{\mu_{l}}\frac{1}{2}f\left(l\right) & = & \left(\sqrt{\mu_{l}}-\frac{1}{4}f\left(l\right)\right)^{2}-\frac{1}{16}f\left(l\right)\\
 & \ge & \left(\sqrt{\left(\beta\left(l\right)+2f\left(l\right)\right)\left(\beta(l+1)+2f\left(l+1\right)\right)}-\frac{1}{4}f\left(l\right)\right)^{2}-\frac{1}{16}f\left(l\right)\\
 & \ge & \left(\beta\left(l\right)+2f\left(l\right)\right)^{2}-\frac{1}{16}f\left(l\right)\ge\gamma\left(l\right)^{2}.
\end{array}
\]
Furthermore on the event $B$, 
\[
\mu_{l}+\sqrt{\mu_{l}}\frac{1}{2}f\left(l\right)=\left(\sqrt{\mu_{l}}+\frac{1}{4}f\left(l\right)\right)^{2}-\frac{1}{16}f\left(l\right)\le\left(\beta\left(l\right)+\frac{1}{2}g\left(l\right)-\frac{1}{4}f\left(l\right)\right)^{2}\le\delta\left(l\right)^{2}.
\]
Therefore \eqref{eq:LDApllic-1} implies that
\[
\tilde{\mathbb{P}}_{0}\left[B|L_{L}^{D_{\lfloor t_{s}\rfloor}}=0\right]\le\left(1-\sum_{l\in I_{2}}e^{-cf\left(k\right)^{2}}\right)^{-1}\tilde{\mathbb{P}}_{0}\left[B|L_{L}^{D_{\lfloor t_{s}\rfloor}}=0\right]\le c\tilde{\mathbb{P}}_{0}\left[A|\tilde{T}_{L-1}^{t_{s}}=0\right],
\]
for $L$ large enough. Using \prettyref{lem:RayishKnightish} the
left-hand side equals
\[
c\tilde{\mathbb{P}}_{0}\left[\mathbb{Q}_{2L_{1}^{D_{\lfloor t_{s}\rfloor}}\to0}^{0,L-1}\left[B_{2\left(\beta\left(1+\cdot\right)+2f\left(1+\cdot\right)\right)^{2}}^{2\left(\beta\left(1+\cdot\right)+\frac{1}{2}g\left(1+\cdot\right)\right)^{2}}\left(I_{2}-1\right)\right]|L_{L}^{D_{\lfloor t_{s}\rfloor}}=0\right].
\]
This is bounded below by
\[
\begin{array}{l}
{\displaystyle \inf_{v\in\left(-500,500\right)}}\mathbb{Q}_{2\left(\beta\left(1\right)+v\right)^{2}\to0}^{0,L-1}\left[B_{2\left(\beta\left(1+\cdot\right)+2f\left(1+\cdot\right)\right)^{2}}^{2\left(\beta\left(1+\cdot\right)+\frac{1}{2}g\left(1+\cdot\right)\right)^{2}}\left(\left[\lfloor\frac{1}{2}l_{0}\rfloor-1,\ldots,L-\lfloor\frac{1}{2}l_{0}\rfloor-1\right]\right)\right]\\
\times\tilde{\mathbb{P}}_{0}\left[\left(\beta\left(1\right)-500\right)^{2}\le L_{1}^{D_{\lfloor t_{s}\rfloor}}\le\left(\beta\left(1\right)+500\right)^{2}|L_{L}^{D_{t_{s}}}=0\right].
\end{array}
\]
By \prettyref{eq:LowerBoundOnL1Prob} the second probability is bounded
below by $c>0$, and by \prettyref{lem:BesselBridgeTube} with $T=L-1,t=\lfloor\frac{1}{2}l_{0}\rfloor$
and $u=\sqrt{2}\beta\left(1\right)$, the first is bounded below by
$cl_{0}/\left(L-1-l_{0}\right)\ge cl_{0}/L$. Therefore the lower
bound of \prettyref{eq:GWBarrierGammaDeltaNewNotation} follows.
\end{proof}
By completing the demonstration of the barrier crossing bounds, we
have now proved all the ``ingredients'' that were used to prove
the upper bound \prettyref{prop:UpperBound} and the lower bound \prettyref{prop:LowerBound}
(except for the small proofs in the appendix). Thus of the tools that
were used to deduce the main result \prettyref{thm:MainResult} only
the concentration result \prettyref{prop:Concentration} remains to
be proven.

\section{\label{sec:Concentration}Concentration of excursion times}

In this section we will prove the concentration result \prettyref{prop:Concentration}
which bounds the total time $D_{t_{s}}^{y,0}$ (recall \eqref{eq:ExcursionTimeAbreviation})
needed to make $t_{s}$ traversals from $\partial B\left(y,r_{0}\right)$
to $\partial B\left(y,r_{1}\right)$. We need the error in the bound
to be smaller than the subleading correction term for $C_{\varepsilon}$,
which is already small compared to the leading order (cf. \prettyref{eq:MainResultIntro}),
and we therefore need a very precise estimate. Essentially, we must
show that 
\begin{equation}
D_{t_{s}}^{y,0}=\frac{1}{\pi}t_{s}\left(1+o\left(\log L/L\right)\right)\mbox{ simultaneously for all }y\in F_{L}.\label{eq:ConcSecInformalGoal}
\end{equation}
The time $D_{n}^{y,0}$ can be written as a sum of $n$ random variables,
namely the time each ``trip'' from $\partial B\left(y,r_{0}\right)$
to $\partial B\left(y,r_{1}\right)$ and back takes. Therefore the
natural approach to get \eqref{eq:ConcSecInformalGoal} - which we
employ - is to derive a Cramer-type large deviation bound on $P_{x}\left[\left|D_{n}^{y,0}-\frac{1}{\pi}n\right|\ge\theta\right]$.

However, several complications arise. Firstly, the typical way to
obtain \eqref{eq:ConcSecInformalGoal} from a large deviation bound
on $D_{n}^{y,0}$ for one $y$, is to use a union bound over $y\in F_{L}$.
This fails in our case, because the best upper bound one can hope
for is $ce^{-c\theta^{2}/n}$ (the bound one gets for sums of iid
random variables), and to obtain \eqref{eq:ConcSecInformalGoal} one
needs to set $n=t_{s}\asymp L^{2}$ and $\theta=c\left(\log L/L\right)n$
for a small constant $c$. This would give a bound of $e^{-c^{2}\left(\log L\right)^{2}}$
which does not ``kill'' $\left|F_{L}\right|\ge e^{2L}$ (recall
\eqref{eq:SizeOFfl}). The issue is similar to that from the proof
of \prettyref{prop:SmartMarkov} in \prettyref{sec:UpperBound} and
the solution is also similar: we take the union bound instead over
a packing of $\sim r_{0}^{-2}\approx\left(\log L\right)^{3/2}$ circles
of radius close to $r_{0}$, in such a way that the concentration
of excursion times for all $y$ in the packing implies the concentration
of excursion times for all $y\in F_{L}$.

Furthermore, the typical way to obtain a large deviation bound on
$D_{n}^{y,0}$ for one $y$ is to write $D_{n}^{y,0}$ as the sum
\[
D_{n}^{y,0}=\sum_{i=1}^{n}\left(D_{i}^{y,0}-R_{i}^{y,0}\right)+\sum_{i=0}^{n-1}\left(R_{i+1}^{y,0}-D_{i}^{y,0}\right),\mbox{ where }D_{0}^{y,0}=0,
\]
of the lengths of each of the $n$ excursions from $\partial B\left(y,r_{1}\right)$
to $\partial B\left(y,r_{0}\right)$ and the lengths of each of the
$n$ excursions from $\partial B\left(y,r_{0}\right)$ to $\partial B\left(y,r_{1}\right)$,
and then use Khasminskii's lemma/Kac's moment formula and the strong
Markov property to obtain large deviation bounds for each of the two
sums, by bounding their exponential moments (cf. \eqref{eq:EasyExpMoment}
and \eqref{eq:EasyBoundExpMoments}). This turns out to work fine
for the first sum, but a further complication arises when applying
this recipe to the second sum. Essentially speaking, the recipe requires
a bound on $E_{z}\left[H_{B\left(y,r_{1}\right)}\right]$ (for appropriate
random $z$ this is the expectation of the summands) that is uniform
over $z\in\partial B\left(y,r_{0}\right)$ and whose error is at most
as large as the $\theta$ which we wish to use. Such a strong uniform
bound turns out to unattainable. Instead, we employ a more sophisticated
technique which inolves considering the Markovian structure of the
starting points $W_{D_{i}^{y,0}},i\ge1,$ of each excursion from $\partial B\left(y,r_{0}\right)$
to $\partial B\left(y,r_{1}\right)$, and computing exactly the expected
length of an excursion when starting from the equilibrium distribution
on starting points.

Let us now start the proof of \prettyref{prop:Concentration}. Recall
\prettyref{eq:DefOfExcursionTimes} for the definition of $D_{n}\left(y,R,r\right)$
and $R_{n}\left(y,R,r\right)$. Most of the results of this section
will be stated for general $0<r<R<\frac{1}{2}$. At the end, when
we carry out the packing argument, we will use the results with $R=r_{0}^{\pm}$
and $r=r_{1}^{\pm}$, for $r_{l}^{\pm}$ as in \eqref{eq:ModifiedRadiiDef}
(therefore it is good keep in mind that in the end we will have $R/r\approx e$
and $R\downarrow0$ as $\left(\log L\right)^{-3/4}$). When it does
not cause confusion we will drop the arguments and write
\[
D_{n}=D_{n}\left(y,R,r\right)\mbox{ and }R_{n}=R_{n}\left(y,R,r\right).
\]
We first introduce rigorously the equilibrium distribution mentioned
above, which will be denoted by $\mu_{r}^{R}$. By Lemma 2.1 of \cite{UenoOnRecurrentMarkovProcesses}
there exists for all $y\in\mathbb{T}$ a pair of probability measures
$\mu_{r}^{R}$ on $\partial B\left(y,R\right)$ and $\mu_{R}^{r}$
on $\partial B\left(y,r\right)$ such that
\begin{equation}
\begin{array}{ccl}
\mu_{r}^{R}\left(\cdot\right) & = & \int_{\partial B\left(y,r\right)}P_{v}\left[W_{H_{\partial B\left(y,R\right)}}\in\cdot\right]\mu_{R}^{r}\left(dv\right)\mbox{, and}\\
\mu_{R}^{r}\left(\cdot\right) & = & \int_{\partial B\left(y,R\right)}P_{v}\left[W_{H_{\partial B\left(y,r\right)}}\in\cdot\right]\mu_{r}^{R}\left(dv\right).
\end{array}\label{eq:RhoAndRhoTilde}
\end{equation}
(Actually these measures are the stationary distributions of the discrete
time Markov chains $\left(W_{D_{n}}\right)_{n\ge1}$ and $\left(W_{R_{n}}\right)_{n\ge1}$).
Next we want to compute an exact formula for $E_{\mu_{r}^{R}}\left[D_{1}\right]$.
For this we will use Green functions. For any measurable $A\subset\mathbb{T}$
let, 
\[
p^{A}\left(t,x,y\right)=P_{x}\left[W_{t}\in dy,H_{A}>t\right],
\]
denote the transition density of $W_{t}$ under $P_{x}$ killed upon
hitting $A$. Recall that the killed Green functions $G^{A}\left(\cdot,\cdot\right)$
is defined by 
\[
G^{A}\left(x,y\right)=\int_{0}^{\infty}p^{A}\left(t,x,y\right)dt\mbox{ for }x,y\in\mathbb{T}.
\]
One can define a measure by 
\[
G^{A}\left(x,B\right)=\int_{B}G^{A}\left(x,y\right)dy\mbox{ for }x\in\mathbb{\mathbb{T}}\mbox{ and measurable }B\subset\mathbb{T}.
\]
Note that
\begin{equation}
G^{A}\left(x,B\right)=E_{x}\left[\int_{0}^{H_{A}}1_{\left\{ W_{t}\in B\right\} }dt\right]\mbox{ for }x\in\mathbb{T}\mbox{ and measurable }B\subset\mathbb{T}.\label{eq:GreenFunctionIsExpectedOccupationMeasure}
\end{equation}
A standard bound on killed Green functions for Brownian motion in
$\mathbb{R}^{2}$ imply the following bounds on the Green function
$G^{B\left(y,R\right)^{c}}\left(u,v\right)$ for $u,v\in B\left(y,r\right)$
(see Lemma 3.36, \cite{PeresMoertersBrownianMotion} and note that
$B\left(y,R\right)\subset\mathbb{T}$ can be identified with a ball
in $\mathbb{R}^{2}$, cf. \prettyref{eq:LawOfBMInTorusAndR2Same})
\begin{equation}
G^{B\left(y,R\right)^{c}}\left(u,v\right)=-\frac{1}{\pi}\log d\left(u,v\right)+\frac{1}{\pi}E_{u}\left[\log d\left(W_{T_{B\left(y,R\right)}},v\right)\right].\label{eq:GreenFunctionBound}
\end{equation}
We are now ready to compute $E_{\mu_{r}^{R}}\left[D_{1}\right]$.
\begin{lem}
\label{lem:OutInOutExpectation}($y\in\mathbb{T}$) For all $0<r<R<\frac{1}{2}$,
\begin{equation}
E_{\mu_{r}^{R}}\left[D_{1}\right]=\frac{1}{\pi}\log\frac{R}{r}.\label{eq:OutInOutExpectation}
\end{equation}
\end{lem}
\begin{proof}
Define a measure $m$ on $\mathbb{T}$ by 
\[
m\left(\cdot\right)=\int_{\partial B\left(y,r\right)}\mu_{R}^{r}\left(dv\right)G^{B\left(y,R\right)^{c}}\left(v,\cdot\right)+\int_{\partial B\left(y,R\right)}\mu_{r}^{R}\left(dv\right)G^{B\left(y,r\right)}\left(v,\cdot\right).
\]
By a theorem of Maruyama and Tanaka (see (2.2), (2.13) and page 121
\cite{UenoOnRecurrentMarkovProcesses}; recall also \eqref{eq:RhoAndRhoTilde})
we have that
\[
m\mbox{ is an invariant measure for }P_{x},
\]
(an intuition for this result can be obtained by considering the corresponding
statement for a Markov chain with discrete state space). By \eqref{eq:GreenFunctionIsExpectedOccupationMeasure},
the second line of \eqref{eq:RhoAndRhoTilde} and the strong Markov
property we have that.
\[
m\left(\mathbb{T}\right)=E_{\mu_{r}^{R}}\left[H_{B\left(y,r\right)}\right]+E_{\mu_{R}^{r}}\left[T_{B\left(y,R\right)}\right]=E_{\mu_{r}^{R}}\left[D_{1}\right].
\]
Since clearly  the only invariant measure for $P_{x}$ is the uniform
distribution $\lambda$ on $\mathbb{T}$ (up to multiplication by
a constant), we have
\begin{equation}
m=c\lambda.\label{eq:MIsMultipleOfLambda}
\end{equation}
Thus 
\begin{equation}
E_{\mu_{r}^{R}}\left[D_{1}\right]=m\left(\mathbb{T}\right)=c.\label{eq:CIsTheAnswer}
\end{equation}
To determine the value of $c$ we note that for $\delta\in\left(0,r\right)$
\[
m\left(B\left(y,\delta\right)\right)=\int_{\partial B\left(y,R\right)}\mu_{R}^{r}\left(du\right)\int_{B\left(y,\delta\right)}G^{B\left(y,R\right)^{c}}\left(u,v\right)\lambda\left(dv\right).
\]
By \eqref{eq:GreenFunctionBound} we have that
\[
G^{B\left(y,R\right)^{c}}\left(u,v\right)=-\frac{1}{\pi}\log\left(r+O\left(\delta\right)\right)+\frac{1}{\pi}\log\left(R+O\left(\delta\right)\right).
\]
Thus for for $v\in\partial B\left(y,r\right)$ and $w\in B\left(y,\delta\right)$
we get that
\[
m\left(B\left(y,\delta\right)\right)=\lambda\left(B\left(y,\delta\right)\right)\left(-\frac{1}{\pi}\log\left(r+O\left(\delta\right)\right)+\frac{1}{\pi}\log\left(R+O\left(\delta\right)\right)\right).
\]
Taking $\delta\to0$ we can now identify the constant in \eqref{eq:MIsMultipleOfLambda}
as $c=\frac{1}{\pi}\log\frac{R}{r}$, and thus \eqref{eq:OutInOutExpectation}
follows from \eqref{eq:CIsTheAnswer}.
\end{proof}
We now start the proofs of the various large deviation bounds we need
to prove \prettyref{prop:Concentration}. We will make the decomposition
\begin{equation}
D_{n}=D_{1}+\sum_{i=2}^{n}\left(D_{i}-R_{i}\right)+\sum_{i=1}^{n-1}\left(R_{i+1}-D_{i}\right),\label{eq:Decomposition}
\end{equation}
and derive bounds for these three terms separatly (we consider $D_{1}$
by itself since the first excursion to $\partial B\left(y,r\right)$
might not actually start in $\partial B\left(y,R\right)$ and vice
versa).

Before we prove the required bounds on $D_{1}$ and $\sum_{i=2}^{n}\left(D_{i}-R_{i}\right)$,
we recall a standard fact about the expected time to exit a ball.
For any $0<R<\frac{1}{2}$ we have for $y\in\mathbb{T}$ all $z\in B\left(y,R\right)\subset\mathbb{T}$
that
\begin{equation}
E_{z}\left[T_{B\left(y,R\right)}\right]=\frac{R^{2}-\left|z-y\right|^{2}}{2},\label{eq:ExpectedTimeToLeaveBallT}
\end{equation}
since the ball $z\in B\left(y,R\right)$ can be identified with a
ball in $\mathbb{R}^{2}$. Recall also Khasminskii's lemma (a consequence
of Kac's moment formula, see (6) \cite{FitzsimmonsPitmanKacsMomentFormula}),
which implies that for any measurable $A\subset\mathbb{T}$ and any
$n\ge1$,
\begin{equation}
\sup_{z\in\mathbb{T}}E_{z}\left[H_{A}^{n}\right]\le n!\left(\sup_{z\in\mathbb{T}}E_{z}\left[H_{A}\right]\right)^{n}.\label{eq:KhasminskiiLemma}
\end{equation}
We have the following crude upper bound on $E_{z}\left[H_{B\left(0,r\right)}\right]$
(see (2.1) \cite{DemboPeresEtAl-CoverTimesforBMandRWin2D})
\begin{equation}
\sup_{z\in\mathbb{T}}E_{z}\left[H_{B\left(0,r\right)}\right]\le c\log r^{-1},\mbox{ for any }0<r<\frac{1}{2}.\label{eq:CrudeHittingTimeBound}
\end{equation}
We now prove the large deviation bound for $D_{1}$.
\begin{lem}
\label{lem:BoundOnFirstGuys}($x,y\in\mathbb{T}$)For all $0<r<R<\frac{1}{2}$
and $u\ge0$ 
\begin{equation}
P_{x}\left[D_{1}\ge u\right]\le ce^{-cu/\left(\log r^{-1}\right)}.\label{eq:BoundOnFirstGuys}
\end{equation}
\end{lem}
\begin{proof}
By the exponential Chebyshev inequality we have for all $\lambda>0$
\[
P_{x}\left[D_{1}\ge u\right]\ge E_{x}\left[\exp\left(\lambda D_{1}\right)\right]e^{-\lambda u}.
\]
By the strong Markov property applied at time $H_{B\left(y,r\right)}$
(recall \prettyref{eq:DefOfExcursionTimes})
\[
E_{x}\left[\exp\left(\lambda D_{1}\right)\right]\le\left(\sup_{z\in\mathbb{T}}E_{z}\left[\exp\left(\lambda H_{B\left(y,r\right)}\right)\right]\right)\left(\sup_{z\in\mathbb{T}}E_{z}\left[\exp\left(\lambda T_{B\left(y,R\right)}\right)\right]\right).
\]
By \eqref{eq:ExpectedTimeToLeaveBallT} and \eqref{eq:KhasminskiiLemma}
we have that
\begin{equation}
\sup_{z\in\mathbb{T}}E_{z}\left[T_{B\left(y,R\right)}^{m}\right]\le m!R^{2}\mbox{ for all }m\ge1.\label{eq:MomentMountExitingR}
\end{equation}
Thus using the series expansion of $e^{x}$ we have
\[
\sup_{z\in\mathbb{T}}E_{z}\left[\exp\left(\lambda T_{B\left(y,R\right)}\right)\right]\le\sum_{k\ge0}\left(\lambda R^{2}\right)^{k}\le2,
\]
provided $\lambda R^{2}\le\frac{1}{2}$. Similarily but using \eqref{eq:CrudeHittingTimeBound}
instead of \eqref{eq:ExpectedTimeToLeaveBallT} we have that
\[
\sup_{z\in\mathbb{T}}E_{z}\left[\exp\left(H_{B\left(y,r\right)}\right)\right]\le\sum_{k\ge0}\left(\lambda c\log r^{-1}\right)^{k}\le2,
\]
provided $c\lambda\log r^{-1}\le\frac{1}{2}$, where $c$ is the constant
from \eqref{eq:CrudeHittingTimeBound}. Thus setting $\lambda=c\frac{1}{\log r^{-1}}$
for a small enough constant $c$ we obtain \eqref{eq:BoundOnFirstGuys}. 
\end{proof}
The next lemma gives the large deviation bound the sum $\sum_{i=2}^{n}\left(D_{i}-R_{i}\right)$,
that is on the time spent ``going from $\partial B\left(y,r\right)$
to $\partial B\left(y,R\right)$''.
\begin{lem}
\label{lem:InOutLD}($x,y\in\mathbb{T}$) For any $0<r<R<\frac{1}{2}$,
$n\ge2$ and $\delta\in\left(0,1\right)$,
\[
\begin{array}{l}
P_{x}\left[\frac{R^{2}-r^{2}}{2}\left(n-1\right)\left(1-\delta\right)\le\sum_{i=2}^{n}\left(D_{i}-R_{i}\right)\le\frac{R^{2}-r^{2}}{2}\left(n-1\right)\left(1+\delta\right)\right]\\
\ge1-ce^{-c\left(n-1\right)\delta^{2}\left(\frac{R^{2}-r^{2}}{R}\right)^{2}}.
\end{array}
\]
\end{lem}
\begin{proof}
By the strong Markov property applied at times $R_{i},i\ge2,$ we
have for $\lambda\in\mathbb{R}$,
\begin{equation}
E_{x}\left[\exp\left(\lambda\sum_{i=2}^{n}\left(D_{i}-R_{i}\right)\right)\right]\le\left(\sup_{z\in\partial B\left(y,r\right)}E_{z}\left[\exp\left(\lambda T_{B\left(y,R\right)}\right)\right]\right)^{n}.\label{eq:EasyExpMoment}
\end{equation}
The equality \eqref{eq:ExpectedTimeToLeaveBallT} implies that
\begin{equation}
E_{z}\left[T_{B\left(y,R\right)}\right]=\frac{R^{2}-r^{2}}{2}\mbox{\,\ for }z\in\partial B\left(y,r\right).\label{eq:Mean}
\end{equation}
Using the series expansion of $e^{x}$, \eqref{eq:MomentMountExitingR}
and \eqref{eq:Mean} one obtains a bound for $E_{z}\left[\exp\left(\lambda T_{B\left(y,R\right)}\right)\right]$
involving a geometric series, so that for $\lambda\in\mathbb{R}$
such that $\left|\lambda\right|R^{2}\le\frac{1}{2}$ (making the geomeric
series summable) one has
\begin{equation}
E_{z}\left[\exp\left(\lambda T_{B\left(y,R\right)}\right)\right]\le1+\lambda\frac{R^{2}-r^{2}}{2}+2\lambda^{2}R^{2}\le e^{\lambda\frac{R^{2}-r^{2}}{2}+2\lambda^{2}R^{2}}.\label{eq:EasyBoundExpMoments}
\end{equation}
Therefore by the exponential Chebyshev inequality have for all $\lambda\in\left(0,\frac{1}{2}\right)$
\[
\begin{array}{rcl}
P_{x}\left[\sum_{i=2}^{n}\left(D_{i}-R_{i}\right)\ge\frac{R^{2}-r^{2}}{2}\left(n-1\right)\left(1+\delta\right)\right] & \le & e^{-\lambda\left(n-1\right)\left(\delta\frac{R^{2}-r^{2}}{2}-2\lambda R^{2}\right)},\mbox{ and }\\
P_{x}\left[\sum_{i=2}^{n}\left(D_{i}-R_{i}\right)\le\frac{R^{2}-r^{2}}{2}\left(n-1\right)\left(1-\delta\right)\right] & \le & ce^{-\lambda\left(n-1\right)\left(\delta\frac{R^{2}-r^{2}}{2}-2\lambda R^{2}\right)}.
\end{array}
\]
Setting $\lambda=\frac{\delta}{8}\frac{R^{2}-r^{2}}{R^{2}}<\frac{1}{2}$
the claim follows.
\end{proof}
Next we aim to prove a similar bound on the sum $\sum_{i=1}^{n-1}\left(R_{i+1}-D_{i}\right)$,
i.e. on the time spent ``going from $\partial B\left(y,R\right)$
to $\partial B\left(y,r\right)$''. This is much more delicate, essentially
because $E_{z}\left[H_{B\left(y,r\right)}\right]$ is not constant
over $z\in\partial B\left(y,R\right)$. We consider the excursions
\[
W_{\left(D_{i}+\cdot\right)\wedge R_{i+1}},i\ge1,
\]
as a $C_{0}\left([0,\infty),\mathbb{T}\right)-$valued sequence. By
the strong Markov property of $W_{t}$ this sequence is a Markov chain
with transition kernel
\begin{equation}
K\left(\omega,A\right)=\int_{\partial B\left(y,R\right)}P_{\omega\left(\infty\right)}\left[W_{H_{B\left(y,R\right)}}\in du\right]P_{u}\left[W_{\cdot\wedge H_{B\left(y,r\right)}}\in A\right],\label{eq:TransitionKernel}
\end{equation}
for $\omega\in C_{0}\left([0,\infty),\mathbb{T}\right)$ and measurable
$A\subset C_{0}\left([0,\infty),\mathbb{T}\right)$.

We employ a renewal argument which consists in making successive attempts
to replace the law $P_{\omega\left(\infty\right)}\left[W_{H_{B\left(y,R\right)}}\in du\right]$
of the transition from the previous excursion to the start of the
next excursion by the law $\mu_{r}^{R}$. We will see that we can
make this succeed with a probability $q$ given by 
\begin{equation}
q=q\left(y,R,r\right)\overset{\mbox{def}}{=}\inf_{u\in\partial B\left(y,R\right),v\in\partial B\left(y,r\right)}\frac{P_{v}\left[W_{H_{B\left(y,R\right)}}\in du\right]}{\mu_{r}^{R}\left(du\right)}.\label{eq:DefOfQ}
\end{equation}
We have the following lower bound on $q$.
\begin{lem}
($y\in\mathbb{T}$) For all $0<r<R<\frac{1}{2}$,
\begin{equation}
q\ge\left(\frac{R-r}{R+r}\right)^{2}.\label{eq:qLowerBound}
\end{equation}
\end{lem}
\begin{proof}
Because of \eqref{eq:RhoAndRhoTilde}
\[
\mu_{r}^{R}\left(du\right)\le\sup_{u\in\partial B\left(y,R\right),v\in\partial B\left(y,r\right)}P_{v}\left[W_{H_{B\left(y,R\right)}}\in du\right].
\]
Also by \eqref{eq:HarmonicMeasureR2} we have
\[
\frac{R^{2}-r^{2}}{\left(R+r\right)^{2}}\le\inf P_{v}\left[W_{H_{B\left(y,R\right)}}\in du\right]\le\sup P_{v}\left[W_{H_{B\left(y,R\right)}}\in du\right]\le\frac{R^{2}-r^{2}}{\left(R-r\right)^{2}},
\]
where $\sup$ and $\inf$ are over $u\in\partial B\left(y,R\right),v\in\partial B\left(y,r\right)$.
Recalling \eqref{eq:DefOfQ}, the claim follows.
\end{proof}
When a renewal succeds, the transition from the end $\omega\left(\infty\right)$
of the previous path to start of the next will be given by $\mu_{r}^{R}$.
When it does not succeed, it will be given by $\nu_{\omega\left(\infty\right)}$,
where for each $a\in\partial B\left(y,r\right)$ we define $\nu_{a}$
by 
\begin{equation}
\nu_{a}\left(A\right)=\nu_{a}\left(y,R,r;A\right)=\frac{P_{a}\left[W_{T_{B\left(y,R\right)}}\in A\right]-q\mu_{r}^{R}\left(A\right)}{1-q},\label{eq:DefOfModifiedTransition}
\end{equation}
for measurable $A\subset B\left(y,R\right)$. By \eqref{eq:DefOfQ}
this is a probability measure.

We now construct a chain with the law of $W_{\left(D_{i}+\cdot\right)\wedge R_{i+1}},i\ge1,$
on a probability space $\left(\mathbb{P},\mathcal{S},S\right)$ in
a certain way that makes the renewal structure explicit. Define on
$\left(\mathbb{P},S,\mathcal{S}\right)$ an iid sequence
\[
I_{1},I_{2},\ldots,
\]
of independent Bernoulli random variables (indicating whether a renewal
takes place) with success probability $q$, and define a sequence
$X_{\cdot}^{1},X_{\cdot}^{2},\ldots$ of random trajectories in $C_{0}\left([0,\infty),\mathbb{T}\right)$
such that
\begin{equation}
X_{\cdot}^{1}\mbox{ has law }P_{x}\left[W_{\left(D_{1}+\cdot\right)\wedge R_{2}}\in d\omega\right],\label{eq:InitialLaw}
\end{equation}
and $X_{\cdot}^{i+1}$ depends on $X_{\cdot}^{1},\ldots,X_{\cdot}^{i}$
and $I_{1},\ldots,I_{i}$ only through $X_{\infty}^{i}$ and $I_{i}$,
in that 
\begin{equation}
X_{\cdot}^{i+1}\mbox{ is sampled according to law }\begin{cases}
P_{\mu_{r}^{R}}\left[W_{\cdot\wedge H_{B\left(y,r\right)}}\in d\omega\right] & \mbox{ if }I_{i}=1,\\
P_{\nu_{X_{\infty}^{i}}}\left[W_{\cdot\wedge H_{B\left(y,r\right)}}\in d\omega\right] & \mbox{ if }I_{i}=0.
\end{cases}\label{eq:Xiplus1sampled}
\end{equation}
The reason for the previous construction is the following lemma.
\begin{lem}
\label{lem:LawOFXiandLawOfExcOfBMCoincide}The $\mathbb{P}-$law of
$\left(X_{\cdot}^{i}\right)_{i\ge1}$ coincides with the $P_{x}-$law
of $\left(W_{\left(D_{i}+\cdot\right)\wedge R_{i+1}}\right)_{i\ge1}.$\end{lem}
\begin{proof}
By construction $\left(X_{\cdot}^{i}\right)_{i\ge1}$ is a Markov
chain on the space of excursions $C_{0}\left([0,\infty),\mathbb{T}\right)$,
and it has transition kernel
\[
\tilde{K}\left(\omega,A\right)=qP_{\mu_{r}^{R}}\left[W_{\cdot\wedge H_{B\left(y,r\right)}}\in A\right]+\left(1-q\right)\int_{\partial B\left(y,R\right)}\nu_{\omega\left(\infty\right)}\left(dw\right)P_{w}\left[W_{\cdot\wedge H_{B\left(y,r\right)}}\in A\right].
\]
By \eqref{eq:DefOfModifiedTransition} we see that $\tilde{K}\left(\omega,A\right)=K\left(\omega,A\right)$
(recall \prettyref{eq:TransitionKernel}), so $\left(W_{\left(D_{i}+\cdot\right)\wedge R_{i+1}}\right)_{i\ge1}$
and $\left(X_{\cdot}^{i}\right)_{i\ge1}$ share the same transition
kernel. Furthermore by \eqref{eq:InitialLaw} they share the same
starting distribution. Thus \prettyref{lem:LawOFXiandLawOfExcOfBMCoincide}
follows. 
\end{proof}
We can thus derive a large deviation bound for $\sum_{i=1}^{n-1}\left(R_{i+1}-D_{i}\right)$
by deriving a bound for $\sum_{i=1}^{n}H_{B\left(y,r\right)}\left(X_{\cdot}^{i}\right)$.
The latter will be facilitated by the built-in renewal structure provided
by the $I_{1},I_{2},\ldots$. To exploit this we let 
\[
J_{1}=0\mbox{ and }J_{i}=\inf\left\{ m>J_{i-1}:I_{m}=1\right\} ,i\ge2,
\]
be the renewal times. Define the total time spent ``going from $\partial B\left(y,R\right)$
to $\partial B\left(y,r\right)$'' during the $m-$th renewal by
\begin{equation}
G_{m}=\sum_{J_{m}<i\le J_{m+1}}H_{B\left(y,r\right)}\left(X_{\cdot}^{i}\right),m\ge1.\label{eq:DefOfGm}
\end{equation}
We have the following.
\begin{lem}
Under $\mathbb{P}$
\begin{eqnarray}
 & G_{1},G_{2},\ldots,\mbox{ are independent},\label{eq:IndependenceOfGs}\\
 & \mbox{ and }G_{2},G_{3},\ldots\mbox{ are iid.}\label{eq:GsSameLaw}
\end{eqnarray}
\end{lem}
\begin{proof}
\eqref{eq:IndependenceOfGs} and \eqref{eq:IndependenceOfGs} both
follow by the construction \eqref{eq:Xiplus1sampled} of $\left(X_{\cdot}^{i}\right)_{i\ge1}$,
since whenever $I_{i}=1$ the starting point of the next trajectory
is sampled according to $\mu_{r}^{R}$, i.e. ``the past is forgotten''.
\end{proof}
To be able to later compute a large deviation bound for $\sum_{i=1}^{m}G_{i}$
we now compute the mean of $G_{i}$, and a bound on its moments.
\begin{lem}
For $m\ge2$
\begin{equation}
\mathbb{E}\left[G_{m}\right]=\frac{E_{\mu_{r}^{R}}\left[H_{B\left(y,r\right)}\right]}{q},\label{eq:MeanOfG}
\end{equation}
and for $m\ge1$ and $k\ge1$
\begin{equation}
\mathbb{E}\left[G_{m}^{k}\right]\le\frac{k!}{q^{k}}\left(c\log r^{-1}\right)^{k}.\label{eq:BoundOnMomentsOfG}
\end{equation}
\end{lem}
\begin{proof}
To see \prettyref{eq:MeanOfG} note that from the construction \eqref{eq:DefOfGm}
of $G_{m}$ and \eqref{eq:Xiplus1sampled} of $X_{\cdot}^{i}$ we
have for $m\ge2$ 
\[
\mathbb{E}\left[G_{m}\right]=E_{\mu_{r}^{R}}\left[H_{B\left(y,r\right)}\right]+\sum_{j=1}^{\infty}\left(1-q\right)^{j}E_{\mu_{r}^{R}}\left[E_{\nu_{W_{R_{j}}}}\left[H_{B\left(y,r\right)}\right]\right].
\]
By \eqref{eq:RhoAndRhoTilde} the $P_{\mu_{r}^{R}}-$law of $W_{R_{j}}$
is $\mu_{R}^{r}$. Thus in fact
\[
\sum_{j=1}^{\infty}\left(1-q\right)^{j}E_{\mu_{r}^{R}}\left[E_{\nu_{W_{R_{j}}}}\left[H_{B\left(y,r\right)}\right]\right]=E_{\nu_{\mu_{R}^{r}}}\left[H_{B\left(y,r\right)}\right]\frac{1-q}{q},
\]
where $\nu_{\mu_{R}^{r}}\left(\cdot\right)$ denotes the measure $\int\mu_{R}^{r}\left(dz\right)\nu_{z}\left(\cdot\right)$.
Now by \eqref{eq:DefOfModifiedTransition} and \eqref{eq:RhoAndRhoTilde}
\[
\nu_{\mu_{R}^{r}}\left(\cdot\right)=\frac{P_{\mu_{R}^{r}}\left[W_{T_{B\left(y,R\right)}}\in\cdot\right]-q\mu_{r}^{R}\left(\cdot\right)}{1-q}=\frac{\mu_{r}^{R}\left(\cdot\right)-q\mu_{r}^{R}\left(\cdot\right)}{1-q}=\mu_{r}^{R}\left(dw\right).
\]
Thus \eqref{eq:MeanOfG} follows since for $m\ge2$
\[
\mathbb{E}\left[G_{m}\right]=E_{\mu_{r}^{R}}\left[H_{B\left(y,r\right)}\right]+E_{\mu_{r}^{R}}\left[H_{B\left(y,r\right)}\right]\frac{1-q}{q}=E_{\mu_{r}^{R}}\left[H_{B\left(y,r\right)}\right].
\]

To see \prettyref{eq:BoundOnMomentsOfG} note that
\[
\begin{array}{ccl}
\mathbb{E}\left[G_{m}^{k}\right] & = & \mathbb{E}\left[\left(\sum_{j=1}^{\infty}1_{\left\{ J_{m}+j\le J_{m+1}\right\} }H_{B\left(y,r\right)}\left(X_{\cdot}^{J_{m}+j}\right)\right)^{k}\right]\\
 & = & {\displaystyle \sum_{i_{1},i_{2},\ldots:\sum i_{j}=k}}\mathbb{E}\left[1_{\left\{ J_{m}+j\le J_{m+1}\mbox{ if }i_{j}\ne0\right\} }\prod_{j=1}^{\infty}\left(H_{B\left(y,r\right)}\left(X_{\cdot}^{J_{m}+j}\right)\right)^{i_{j}}\right]\\
 & = & {\displaystyle \sum_{i_{1},i_{2},\ldots:\sum i_{j}=k}}\left(1-q\right)^{\sup\left\{ j:i_{j}\ne0\right\} -1}\mathbb{E}\left[\prod_{j=1}^{\infty}\left(H_{B\left(y,r\right)}\left(X_{\cdot}^{J_{m}+j}\right)\right)^{i_{j}}\right].
\end{array}
\]
By repeated application of Khasminskii's lemma \eqref{eq:KhasminskiiLemma}
and the strong Markov property we have 
\[
\begin{array}{ccl}
\mathbb{E}\left[\prod_{j=1}^{\infty}\left(H_{B\left(y,r\right)}\left(X_{\cdot}^{J_{m}+j}\right)\right)^{i_{j}}\right] & \le & \left(\sup_{z\in\mathbb{T}}E_{z}\left[H_{B\left(y,r\right)}\right]\right)^{k}\prod_{j=1}^{\infty}i_{j}!\\
 & \overset{\eqref{eq:CrudeHittingTimeBound}}{\le} & \left(c\log r^{-1}\right)^{k}\prod_{j=1}^{\infty}i_{j}!.
\end{array}
\]
Thus
\[
\mathbb{E}\left[G_{m}^{k}\right]\le\left(c\log r^{-1}\right)^{k}{\displaystyle \sum_{i_{1},i_{2},\ldots:\sum i_{j}=k}}\left(1-q\right)^{\sup\left\{ j:i_{j}\ne0\right\} -1}\prod_{j=1}^{\infty}i_{j}!.
\]
Now if $E_{1},E_{2},\ldots,$ are independent standard exponential
random variables which are also independent of $J_{1}$, then since
$\mathbb{E}\left[E_{i}^{k}\right]=k!$,
\[
{\displaystyle \sum_{i_{1},i_{2},\ldots:\sum i_{j}=k}}\left(1-q\right)^{\sup\left\{ j:i_{j}\ne0\right\} -1}\prod_{j=1}^{\infty}i_{j}!=\mathbb{E}\left[\left(\sum_{J_{1}<i\le J_{2}}E_{i}\right)^{k}\right]=\frac{k!}{q^{k}},
\]
where the last equality holdss because $\sum_{J_{1}<i\le J_{2}}E_{i}$
is an exponential random variable with mean $q^{-1}$. Thus \eqref{eq:BoundOnMomentsOfG}
follows.
\end{proof}
We can now derive a large deviation control on sums of the $G_{i}$.
\begin{lem}
\label{lem:GLD}For all $\delta\in\left(0,c\right)$\textup{ and all
$m\ge1$
\begin{equation}
\begin{array}{l}
\mathbb{P}\left[\mathbb{E}\left[G_{2}\right]m\left(1-\delta\right)\le\sum_{i=1}^{m}G_{i}\le\mathbb{E}\left[G_{2}\right]m\left(1+\delta\right)\right]\\
\ge1-c\exp\left(-c\left(m-1\right)\delta^{2}\left(\frac{\mathbb{E}\left[G_{2}\right]}{c\log r^{-1}/q}\right)^{2}\right).
\end{array}\label{eq:GLD}
\end{equation}
}\end{lem}
\begin{proof}
For all $\lambda>0$
\[
\mathbb{E}\left[\exp\left(\lambda G_{2}\right)\right]\le1+\lambda\mathbb{E}\left[G_{2}\right]+\sum_{k\ge2}\frac{\left|\lambda\right|^{k}}{k!}\mathbb{E}\left[G_{2}^{k}\right].
\]
Using \eqref{eq:BoundOnMomentsOfG} this gives
\[
\mathbb{E}\left[\exp\left(\lambda G_{2}\right)\right]\le1+\lambda\mathbb{E}\left[G_{2}\right]+2\lambda^{2}\left(\frac{c\log r^{-1}}{q}\right)^{2}\le e^{\lambda\mathbb{E}\left[G_{2}\right]+2\lambda^{2}\left(c\log r^{-1}/q\right)^{2}},
\]
provided
\begin{equation}
\left|\lambda\right|\frac{c\log r^{-1}}{q}\le\frac{1}{2}.\label{eq:Requirement}
\end{equation}
Similarly (but more crudely) for such $\lambda$ we have that 
\[
\mathbb{E}\left[\exp\left(\lambda G_{1}\right)\right]\le\exp\left(c\frac{\lambda\log r^{-1}}{q}\right)\le c.
\]
Thus using an exponential Chebyshev bound, \eqref{eq:IndependenceOfGs}
and \eqref{eq:GsSameLaw} we have for all $\lambda>0$ as in \eqref{eq:Requirement}
\[
\mathbb{P}\left[\sum_{i=1}^{m}G_{i}\ge\mathbb{E}\left[G_{2}\right]m\left(1+\delta\right)\right]\le c\exp\left(-\left(m-1\right)\lambda\left\{ \delta\mathbb{E}\left[G_{2}\right]-2\lambda\left(c\frac{\log r^{-1}}{q}\right)^{2}\right\} \right),
\]
and (using also that $\lambda\mathbb{E}\left[G_{2}\right]\left(1-\delta\right)\le\lambda c\frac{\log r^{-1}}{q}\le c$)
\[
\mathbb{P}\left[\sum_{i=2}^{m}G_{i}\le\mathbb{E}\left[G_{2}\right]m\left(1-\delta\right)\right]\le c\exp\left(-\left(m-1\right)\lambda\left\{ \delta\mathbb{E}\left[G_{2}\right]-2\lambda\left(c\frac{\log r^{-1}}{q}\right)^{2}\right\} \right).
\]
Setting $\lambda=c\frac{\delta\mathbb{E}\left[G_{2}\right]}{\left(c\log r^{-1}/q\right)^{2}}$
for a small enough constant $c$ (which we may since then \eqref{eq:Requirement}
is satisifed by \eqref{eq:BoundOnMomentsOfG}) we get \eqref{eq:GLD}. 
\end{proof}
We can now use \prettyref{lem:GLD} to derive a large deviation control
on the sum $\sum_{i=1}^{n-1}\left(R_{i+1}-D_{i}\right)$. For this
we essentially speaking need to control the number of renewals that
take place in the first $n-1$ steps of the Markov chain $\left(X_{\cdot}^{i}\right)_{i\ge1}$.
\begin{prop}
\label{prop:OutInLD}($x,y\in\mathbb{T}$) If $n\ge1$ and $\delta\in\left(0,c\right)$
then with $\mu=E_{\mu_{r}^{R}}\left[H_{B\left(y,r\right)}\right]$
\begin{equation}
\begin{array}{l}
P_{x}\left[\mu\left(n-1\right)\left(1-\delta\right)\le\sum_{i=1}^{n-1}\left(R_{i+1}-D_{i}\right)\le\mu\left(n-1\right)\left(1+\delta\right)\right]\\
\ge1-\exp\left(-c\left(n-1\right)q^{2}\delta^{2}\left(\frac{\mu}{\log r^{-1}}\right)^{2}\right).
\end{array}\label{eq:OutInLd}
\end{equation}
\end{prop}
\begin{proof}
By \prettyref{lem:LawOFXiandLawOfExcOfBMCoincide} it suffices to
show that
\begin{equation}
\begin{array}{l}
\mathbb{P}\left[\mu\left(n-1\right)\left(1-\delta\right)\le\sum_{i=1}^{n-1}H_{B\left(y,r\right)}\left(X_{\cdot}^{i}\right)\le\mu\left(n-1\right)\left(1+\delta\right)\right]\\
\ge1-\exp\left(-cnq^{2}\delta^{2}\left(\frac{\mu}{\log r^{-1}}\right)^{2}\right).
\end{array}\label{eq:UpperBoundsInTermOfXs}
\end{equation}
Let
\[
m^{-}=\frac{\left(n-1\right)q}{1+\frac{\delta}{100}}\mbox{ and }m^{+}=\frac{\left(n-1\right)q}{1-\frac{\delta}{100}}.
\]
We have
\begin{equation}
\begin{array}{l}
\mathbb{P}\left[\sum_{m=1}^{m^{-}}G_{m}\le\sum_{i=1}^{n-1}H_{B\left(y,r\right)}\left(X_{\cdot}^{i}\right)\le\sum_{m=1}^{m^{+}}G_{m}\right]\\
\ge\mathbb{P}\left[J_{m^{-}}<n-1\le J_{m^{+}}\right]\\
\ge1-\mathbb{P}\left[\sum_{i=1}^{n-1}I_{i}\le m^{-}\right]-\mathbb{P}\left[\sum_{i=1}^{n-1}I_{i}\ge m^{+}\right]\ge1-ce^{-c\left(n-1\right)q^{2}\delta^{2}},
\end{array}\label{eq:asdasdasddsa}
\end{equation}
where the last inequality follows by Hoeffding's large deviation inequality
for the binomial distribution with parameters $m^{\pm}$ and $q$.
Thus the complement of the probability in \eqref{eq:UpperBoundsInTermOfXs}
is bounded above by
\begin{equation}
\mathbb{P}\left[\sum_{i=1}^{m^{+}}G_{i}\ge\mu\left(n-1\right)\left(1+\delta\right)\right]+\mathbb{P}\left[\sum_{i=1}^{m^{-}}G_{i}\le\mu\left(n-1\right)\left(1-\delta\right)\right]+ce^{-c\left(n-1\right)q^{2}\delta^{2}}.\label{eq:mplus}
\end{equation}
Now by \prettyref{eq:MeanOfG}
\[
E_{\mu_{r}^{R}}\left[H_{B\left(y,r\right)}\right]\left(n-1\right)\left(1+\delta\right)=\mathbb{E}\left[G_{2}\right]m^{+}\left(1-\frac{\delta}{100}\right)\ge\left(1+\delta\right)\mathbb{E}\left[G_{2}\right]m^{+}\left(1+\frac{\delta}{2}\right),
\]
and similarly
\[
E_{\mu_{r}^{R}}\left[H_{B\left(y,r\right)}\right]\left(n-1\right)\left(1-\delta\right)\le\mathbb{E}\left[G_{2}\right]m^{-}\left(1-\frac{\delta}{2}\right).
\]
Thus using \eqref{eq:GLD} the complement of the probability in \eqref{eq:UpperBoundsInTermOfXs}
is bounded above by
\[
\exp\left(-c\left(m_{-}-1\right)\delta^{2}\left(\frac{\mu}{c\log r^{-1}/q}\right)^{2}\right)+ce^{-c\left(n-1\right)q^{2}\delta^{2}},
\]
so \eqref{eq:UpperBoundsInTermOfXs} follows, since $\mu\le c\log r^{-1}$
by \prettyref{eq:CrudeHittingTimeBound} and $m_{-}\ge cn$.
\end{proof}
We can now combine \prettyref{lem:BoundOnFirstGuys}, \prettyref{lem:InOutLD}
and \prettyref{prop:OutInLD} to obtain a large deviation bound for
$D_{n}$.
\begin{prop}
\label{prop:GeneralConcentration}$\left(x,y\in\mathbb{T}\right)$
For all $0<r<R<\frac{1}{2}$, $n\ge2$ and $\delta\in\left(0,c\right)$
\begin{equation}
P_{x}\left[\frac{1}{\pi}\log\frac{R}{r}n\left(1-\delta\right)\le D_{n}\le\frac{1}{\pi}\log\frac{R}{r}n\left(1+\delta\right)\right]\ge1-ce^{-cn\delta^{2}R\left(1-\frac{r}{R}\right)^{6}/\left(\log r^{-1}\right)^{2}}.\label{eq:Concentration}
\end{equation}
\end{prop}
\begin{proof}
Using the decomposition \prettyref{eq:Decomposition} the complement
of the probability in \prettyref{eq:Concentration} is above by
\[
\begin{array}{ccl}
P_{x}\left[D_{1}\ge n\frac{\delta}{2}\frac{1}{\pi}\log\frac{R}{r}\right] & + & P_{x}\left[\left\{ n\left(1-\frac{\delta}{2}\right)\frac{R^{2}-r^{2}}{2}\le\sum_{i=2}^{n}\left(D_{i}-R_{i}\right)\le n\left(1+\frac{\delta}{2}\right)\frac{R^{2}-r^{2}}{2}\right\} ^{c}\right]\\
 & + & P_{x}\left[\left\{ n\left(1-\frac{\delta}{2}\right)\mu\le\sum_{i=1}^{n-1}\left(R_{i+1}-D_{i}\right)\le n\left(1+\frac{\delta}{2}\right)\mu\right\} ^{c}\right],
\end{array}
\]
with $\mu$ as in \prettyref{prop:OutInLD}, since by \prettyref{eq:OutInOutExpectation}
and \prettyref{eq:ExpectedTimeToLeaveBallT}, 
\[
\frac{1}{\pi}\log\frac{R}{r}=E_{\mu_{r}^{R}}\left[D_{1}\right]=E_{\mu_{r}^{R}}\left[R_{1}\right]+E_{\mu_{r}^{R}}\left[D_{1}-R_{1}\right]=\mu+\frac{R^{2}-r^{2}}{2}.
\]
Using \prettyref{lem:BoundOnFirstGuys}, \prettyref{lem:InOutLD}
and \prettyref{prop:OutInLD} this quantity is at most
\[
ce^{-cn\delta\log\frac{R}{r}/\left(\log r^{-1}\right)}+ce^{-cn\delta^{2}\left(\frac{R^{2}-r^{2}}{R}\right)^{2}}+\exp\left(-cnq^{2}\delta^{2}\left(\frac{\mu}{\log r^{-1}}\right)^{2}\right).
\]
By \prettyref{eq:qLowerBound} we have $q=\left(\frac{1-r/R}{1+r/R}\right)^{2}$,
and simple calculus shows that
\[
\mu=\frac{1}{\pi}\log\frac{R}{r}-\frac{R^{2}-r^{2}}{2}\ge c\left(1-\left(\frac{r}{R}\right)^{2}\right),
\]
so $q^{2}\mu^{2}\ge\left(1-r/R\right)^{6}/\left(1+r/R\right)^{2}\ge c\left(1-r/R\right)^{6}$.
A fortiori $\log\frac{R}{r}\ge c\left(1-\left(r/R\right)^{2}\right)\ge c\left(1-r/R\right)^{6}$,
and $\left(\frac{R^{2}-r^{2}}{R}\right)^{2}=R\left(1-\left(r/R\right)^{2}\right)\ge R\left(1-r/R\right)^{6}$,
so \prettyref{eq:Concentration} follows.
\end{proof}
Finally we may now use \prettyref{prop:GeneralConcentration} to prove
the main result of this section: \prettyref{prop:Concentration}.
For this we use a union bound over a ``packing'' of circles, similarly
to in the proof of \prettyref{prop:SmartMarkov}.
\begin{proof}[Proof of \prettyref{prop:Concentration}]
For $y\in F_{L}$ let $y_{l}$ denote the point in $F_{l}$ that
is closest to $y$ (breaking ties in some arbitary way). We have (cf.
\prettyref{eq:Distance})
\[
d\left(y,y_{\log L}\right)\overset{\eqref{eq:GridDefinition}}{\le}r_{\log L}\overset{\eqref{eq:DefOfRadii}}{\le}L^{-1}.
\]
Now setting 
\begin{equation}
r_{l}^{-}=\left(1-\frac{100}{L}\right)r_{l}\mbox{ and }r_{l}^{+}=\left(1-\frac{100}{L}\right)^{-1}r_{l}\mbox{ for }l\in\left\{ 0,1\right\} \label{eq:ModifiedRadiiConcentration}
\end{equation}
(as in \eqref{eq:ModifiedRadiiDef}) we have that
\[
r_{1}^{-}\le r_{1}-r_{\log L}\le r_{1}+r_{\log L}\le r_{1}^{+}\mbox{ and }r_{0}^{-}\le r_{0}-r_{\log L}\le r_{0}+r_{\log L}\le r_{0}^{+}.
\]
Thus for all $y\in F_{L}$ 
\begin{equation}
\begin{array}{l}
B\left(y_{\log L},r_{1}^{-}\right)\subset B\left(y,r_{1}\right)\subset B\left(y_{\log L},r_{1}^{+}\right)\\
\subset B\left(y_{\log L},r_{0}^{-}\right)\subset B\left(y,r_{0}\right)\subset B\left(y_{\log L},r_{0}^{+}\right).
\end{array}\label{eq:ConcentrationSandwich}
\end{equation}
Because of \prettyref{eq:ConcentrationSandwich}, each excursion from
$\partial B\left(y,r_{1}\right)$ to $\partial B\left(y,r_{0}\right)$
happens during an excursion from $\partial B\left(y_{\log L},r_{1}^{-}\right)$
to $B\left(y_{\log L},r_{0}^{+}\right)$. Thus for all $y\in F_{L}$
and all $s$ 
\[
D_{t_{s}}^{y,0}\le D_{\lfloor t_{s}\rfloor}\left(y_{\log L},r_{0}^{+},r_{1}^{-}\right).
\]
Also during each excursion from $\partial B\left(y,r_{1}\right)$
to $\partial B\left(y,r_{0}\right)$ at least one excursion from $\partial B\left(y_{\log L},r_{1}^{+}\right)$
to $B\left(y_{\log L},r_{0}^{-}\right)$ takes place. Thus we have
for all $y\in F_{L}$ and all $s$
\[
D_{\lfloor t_{s}\rfloor}\left(y_{\log L},r_{0}^{-},r_{1}^{+}\right)\le D_{t_{s}}^{y,0}.
\]
Therefore the required bounds \prettyref{eq:UpperBoundConcentration}
and \prettyref{eq:LowerBoundConcentration} follow from
\begin{eqnarray}
\lim_{L\to\infty}P_{x}\left[D_{\lfloor t_{s}\rfloor}\left(y,r_{0}^{+},r_{1}^{-}\right)>\frac{1}{\pi}t_{2s}\mbox{ for some }y\in F_{\log L}\right] & = & 0,\label{eq:UpperBoundConcentration-1}\\
\lim_{L\to\infty}P_{x}\left[D_{\lfloor t_{s}\rfloor}\left(y,r_{0}^{-},r_{1}^{+}\right)<\frac{1}{\pi}t_{-\frac{s}{2}}\mbox{ for some }y\in F_{\log L}\right] & = & 0.\label{eq:LowerBoundConcentration-1}
\end{eqnarray}
We use a union bound to obtain that the probability in \prettyref{eq:UpperBoundConcentration-1}
is bounded above by
\begin{equation}
\left|F_{\log L}\right|\times\sup_{x,y\in\mathbb{T}}P_{x}\left[D_{\lfloor t_{s}\rfloor}\left(y,r_{0}^{+},r_{1}^{-}\right)>\frac{1}{\pi}t_{2s}\right].\label{eq:UnionBound}
\end{equation}
If $\delta=\frac{s}{100}\frac{\log L}{L}$ since $t_{2s}/t_{s}\ge1+\frac{s\log}{2L}$
and $\log\frac{r_{0}^{+}}{r_{1}^{-}}=1+O\left(L^{-1}\right)$ we have
for $L\ge c$ 
\[
\frac{1}{\pi}t_{2s}\ge\frac{1}{\pi}\log\frac{r_{0}^{+}}{r_{1}^{-}}t_{s}\left(1+\delta\right).
\]
Thus by \prettyref{prop:GeneralConcentration} (note that $\delta\to0$,
so for $L$ large enough the proposition is applicable),
\[
\begin{array}{ccl}
{\displaystyle \sup_{y\in\mathbb{T}}}P_{x}\left[D_{\lfloor t_{s}\rfloor}\left(y,r_{0}^{+},r_{1}^{-}\right)>\frac{1}{\pi}t_{2s}\right] & \le & ce^{-c\delta^{2}\lfloor t_{s}\rfloor r_{0}^{+}\left(1-r_{1}^{+}/r_{0}^{-}\right)^{6}/\left(\log r_{1}^{+}\right)^{2}}\\
 & \overset{\eqref{eq:DefOfRadii},\eqref{eq:ModifiedRadiiConcentration}}{\le} & ce^{-c\delta^{2}\lfloor t_{s}\rfloor\left(\log L\right)^{-3/4}/\left(\log\log L\right)^{2}}\\
 & \overset{\eqref{eq:defofts}}{\le} & ce^{-cs^{2}\left(\log L\right)^{2}\frac{\left(\log L\right)^{-3/4}}{\log\log L}}\overset{\eqref{eq:DefOfRadii}}{\le}ce^{-cs^{2}\left(\log L\right)^{1.01}}.
\end{array}
\]
Going back to \prettyref{eq:UnionBound} we have by \prettyref{eq:SizeOFfl}
that the probability in \prettyref{eq:UpperBoundConcentration-1}
is bounded by
\[
c\left(\log L\right)^{3/2}e^{2\log L}\times ce^{-cs^{2}\left(\log L\right)^{1.01}}=o\left(1\right).
\]
Thus we have proved \eqref{eq:UpperBoundConcentration-1}, and therefore
\prettyref{eq:UpperBoundConcentration}. The claim \eqref{eq:LowerBoundConcentration-1}
(and therefore \prettyref{eq:LowerBoundConcentration}) follows similarly
by a union bound, \prettyref{eq:SizeOFfl} and \prettyref{prop:GeneralConcentration}.
\end{proof}
Having proven the concentration result \prettyref{prop:Concentration},
all three main propositions \ref{prop:UpperBound}-\ref{prop:Concentration}
that went in to the proof of the main result \prettyref{thm:MainResult}
have been demonstrated. Thus the proof of \prettyref{thm:MainResult}
is complete (except for the small proofs in the appendix). Let us
finish with a remark on the conjecture \prettyref{eq:DiscreteCorrection}
about the cover time of the discrete two dimensional torus.
\begin{rem}
\label{rem:EndRemark}In the proof of \prettyref{thm:MainResult}
we have used the rotational invariance of Brownian motion in balls
extensively. It is this invariance which gives us the exact formula
\prettyref{eq:BallExitProbForRadii} for the probability of going
``up a scale or down a scale'', and the characterisation of the
traversal process $T_{l}^{y,t},l\ge0,$ as a Galton-Watson process.
A lattice random walk has no such invariance property. But for balls
of large radius a discrete torus analogue of \prettyref{eq:BallExitProbForRadii}
still holds approximately, and therefore an analogue of our traversal
processes should behave roughly as a Galton-Watson process. Our argument
therefore provides a heuristic justification of \prettyref{eq:DiscreteCorrection}.
Since the discrete torus version of \prettyref{eq:BallExitProbForRadii}
comes with a quantitative error (see Proposition 1.6.7 and Excercise
1.6.8 \cite{LawlersLillaGrona}), it is concievable that it can also
be used to prove \prettyref{eq:DiscreteCorrection}.\end{rem}
\begin{acknowledgement*}
The authors thank Louis-Pierre Arguin, Alain-Sol Sznitman and Augusto
Teixeira for useful discussions, and Serguei Popov for suggesting
the use of renewals to prove \prettyref{prop:OutInLD}.
\end{acknowledgement*}

\section{Appendix}

In the appendix we collect some less important proofs. We first give
the proof of the large deviation bound \prettyref{lem:LDForBinSumOfGeo}
for sums of a binomial number of geometric random variables, which
was used to prove the upper bound \prettyref{prop:UpperBound}.
\begin{proof}[Proof of \prettyref{lem:LDForBinSumOfGeo}]
Note that
\begin{equation}
\mathbb{P}\left[\sum_{i=1}^{n}J_{i}G_{i}\le\theta\right]=\mathbb{P}\left[\sum_{i=1}^{J_{1}+\ldots+J_{n}}G_{i}\le\theta\right].\label{eq:Rust}
\end{equation}
Now (since a sum of geometrics is a negative binomial distribution)
we have
\[
\mathbb{P}\left[\sum_{i=1}^{m}G_{i}\le\theta\right]=\mathbb{P}\left[I_{1}+\ldots+I_{\theta}\ge m\right]\mbox{\,\ for }m\ge1,
\]
where $I_{1},I_{2},\ldots$ are iid Bernoulli random variables with
success probablity $p$, which can be taken to be independent of the
$J_{i}-s$. Thus (by conditioning on $J_{1}+\ldots+J_{n}$ in \prettyref{eq:Rust})
we have in fact 
\[
\mathbb{P}\left[\sum_{i=1}^{n}J_{i}G_{i}\le\theta\right]=\mathbb{P}\left[I_{1}+\ldots+I_{\theta}\ge J_{1}+\ldots+J_{n}\right].
\]
For any $\lambda>0$ this probability is bounded above by
\[
\begin{array}{l}
\mathbb{E}\left[\exp\left(\lambda\left(I_{1}+\ldots+I_{\theta}-J_{1}-\ldots J_{n}\right)\right)\right]\\
=\left(1+p\left(e^{\lambda}-1\right)\right)^{\theta}\left(1+q\left(e^{-\lambda}-1\right)\right)^{n}\le\exp\left(\theta p\left(e^{\lambda}-1\right)+qn\left(e^{-\lambda}-1\right)\right),
\end{array}
\]
where we have used that $1+x\le e^{x}$. Now \prettyref{eq:LDForBinSumOfGeoLowerBound}
follows by setting $\lambda=\frac{1}{2}\log\frac{qn}{\theta p}$.

\end{proof}
Next we  derive the characterisation \prettyref{lem:RayKnightCont}
of local times of continuous time random walk on $\left\{ 0,\ldots,L\right\} $
from the generalized second Ray-Knight theorem. Recall the definition
of $\tilde{\mathbb{P}}_{l}$ and $Y_{t}$ from above \prettyref{eq:ContTimeLocTimeDef}
and the definition of $L_{l}^{t}$ from \prettyref{eq:ContTimeLocTimeDef}.
\begin{proof}[Proof of \prettyref{lem:RayKnightCont}]
Let $\mathcal{L}=\left\{ 0,\ldots,L\right\} $. The generalized
second Ray-Knight theorem (see \cite{EHMRSARayKnightTheoremForsymmetricMarkovProcesses}
or Theorem 8.2.2 \cite{MarucsRosenMarkovProcGaussProcandLocTime})
implies that $\left(L_{l}^{\tau\left(t\right)}+\frac{1}{2}\eta_{x}^{2}\right)_{l\in\mathcal{L}}\overset{\mbox{law}}{=}\left(\frac{1}{2}\left(\eta_{l}+\sqrt{2t}\right)^{2}\right)_{l\in\mathcal{L}}$,
where $\eta_{l}$ is a centered Gaussian process on $\mathcal{L}$
with covariance $\mathbb{E}\left[\eta_{a}\eta_{b}\right]=\tilde{\mathbb{E}}_{b}\left[L_{a}^{H_{0}}\right]=a$
for $b\le a$, independent of $L_{l}^{\tau\left(t\right)}$. Thus
$\eta_{l},l\in\mathcal{L}$, is in fact Brownian motion at the integer
times $l\in\mathcal{L}$. This in turn implies that $\left(\frac{1}{2}\eta_{l}^{2}\right)_{l\in\mathcal{L}}$
has the $\mathbb{Q}_{0}^{1}-$law of $\left(\frac{1}{2}X_{l}\right)_{l\in\mathcal{L}}$
and $\left(\frac{1}{2}\left(\eta_{l}+\sqrt{2t}\right)^{2}\right)_{l\in\mathcal{L}}$
has the $\mathbb{Q}_{2t}^{1}-$law of $\left(\frac{1}{2}X_{l}\right)_{l\in\mathcal{L}}$
(recall \prettyref{eq:Bessel1andBMSame}). By the additivity property
\prettyref{eq:AdditiveBessel} of Bessel processes we thus have that
$\left(L_{l}^{\tau\left(t\right)}+\frac{1}{2}X_{l}^{1}\right)_{l\in\mathcal{L}}\overset{\mbox{law}}{=}\left(\frac{1}{2}X_{l}^{1}+\frac{1}{2}X_{l}^{2}\right)_{l\in\mathcal{L}}$
where $\left(X_{t}^{1}\right)_{t\ge0}$ has law $\mathbb{Q}_{0}^{1}$,
$X_{t}^{2}$ haw law $\mathbb{Q}_{2t}^{0}$. Now the claim follows
because we may ``cancel out'' $\frac{1}{2}X_{l}^{1}$ from this
equality in law, since all random variables involved are non-negative
(see (2.56) \cite{SznitmanTopicsInOccupationTimes}).
\end{proof}
Next we give the proof of \prettyref{lem:RayishKnightish}, which
describes the law of the local times $L_{l}^{D_{t}},l\in\left\{ 0,\ldots,L\right\} ,$
of continuous time random walk on $\left\{ 0,1,\ldots,L\right\} $
when conditioned on $L_{L}^{D_{t}}=0$. Recall the definition of $D_{t}$
from \prettyref{eq:DeparturesFromZero}. For the proof let us denote
by $\Gamma$ the state space of $\left(Y_{t}\right)_{t\ge0}$, that
is the space of all piecewise constant cadlag functions from $[0,\infty)$
to $\left\{ 0,\ldots,L\right\} $.
\begin{proof}[Proof of \prettyref{lem:RayishKnightish}]
Define the succesive returns to and departures from $\left\{ 0,\ldots,L\right\} \backslash\left\{ 1\right\} $
of $Y_{t}$ by $\tilde{D}_{0}=H_{1}$, 
\[
\tilde{R}_{n}=H_{\left\{ 0,2\right\} }\circ\theta_{\tilde{D}_{n-1}}+\tilde{D}_{n-1},n\ge1,\mbox{ and }\tilde{D}_{n}=H_{1}\circ\theta_{\tilde{R}_{n}}+\tilde{R}_{n},n\ge1.
\]
Collect the excursions of $Y_{t}$ into a marked point process $\mu$
on $[0,\infty)\times\Gamma$ defined by
\[
\mu=\sum_{i\ge1}\delta_{\left(L_{1}^{\tilde{R}_{i}},Y_{\left(\tilde{R}_{i}+\cdot\right)\wedge\tilde{D}_{i}}\right)}.
\]
The point process $\mu$ is a Poisson point process on $\mathbb{R}_{+}\times\Gamma$
of intensity 
\begin{equation}
\lambda\otimes\left(\frac{1}{2}\tilde{\mathbb{P}}_{0}\left[Y_{\cdot\wedge H_{1}}\in dw\right]+\frac{1}{2}\tilde{\mathbb{P}}_{2}\left[Y_{\cdot\wedge H_{1}}\in dw\right]\right),\label{eq:PPPIntensity}
\end{equation}
where $\lambda$ is Lebesgue-measure normalized so that $\lambda\left(\left[0,1\right]\right)=2$.
We can decompose this point process into
\[
\mu_{1}=1_{\mathbb{R}_{+}\times\left\{ Y_{0}=0\right\} }\mu\mbox{ and }\mu_{2}=1_{\mathbb{R}_{+}\times\left\{ Y_{0}=2,H_{L}<H_{1}\right\} }\mu\mbox{ and }\mu_{3}=1_{\mathbb{R}_{+}\times\left\{ Y_{0}=2,H_{1}<H_{L}\right\} }\mu,
\]
where $\mu_{1}$ collects the excursions that start in $0$, $\mu_{2}$
collects the excursions that start in $2$ and hit $L$, and $\mu_{3}$
has the excurions that start in $2$ and avoid $L$. Since we are
restricting $\mu$ to disjoint sets, $\mu_{1}$, $\mu_{2}$ and $\mu_{3}$
are independent Poisson point processes.

Let
\[
\mu_{1}=\sum_{i}\delta_{\left(S_{i},w_{i}\right)},
\]
for $S_{1}<S_{2}<\ldots$, so that $S_{t}$ is the local time at vertex
$1$ until the $t-$th jump to $0$. Note that (recall \prettyref{eq:DeparturesFromZero})
\[
L_{1}^{D_{t}}=S_{t},\mbox{ for }t\in\left\{ 1,2,\ldots\right\} 
\]
We have
\begin{equation}
L_{l}^{D_{t}}=\sum_{\left(s,w\right)\in\mu_{2}\cup\mu_{3}:s\le t}L_{l}^{\infty}\left(w_{i}\right)\mbox{\,\ for }l\in\left\{ 2,3,\ldots\right\} \label{eq:mu2plusmu3}
\end{equation}
where $L_{l}^{\infty}\left(w\right)$ is the local time at $l$ of
the path $w$, i.e. $L_{l}^{\infty}\left(w\right)=d_{l}^{-1}\int_{0}^{\infty}1_{\left\{ w_{s}=l\right\} }ds$
for $d_{l}$ as in \prettyref{eq:Normalization}. For any $u\ge0$
define the vector
\begin{equation}
V_{u}=\left(u,\sum_{\left(s,w\right)\in\mu_{2}:s\le t}L_{2}^{\infty}\left(w_{i}\right),\ldots,\sum_{\left(s,w\right)\in\mu_{2}:s\le t}L_{L}^{\infty}\left(w_{i}\right)\right)\in\mathbb{R}^{L},\label{eq:DefOfVu}
\end{equation}
By \prettyref{eq:mu2plusmu3} we have
\[
\left(L_{l}^{D_{t}}\right)_{l\in\left\{ 1,\ldots,L\right\} }=V_{S_{t}}\mbox{ on the event }\left\{ L_{L}^{D_{t}}=0\right\} =\left\{ \mu_{3}\left(\left[0,S_{t}\right]\times\Gamma\right)=0\right\} .
\]
Furthermore note that $L_{1}^{D_{t}}$ and $\left\{ L_{L}^{D_{t}}=0\right\} $
only depend on $\mu_{1}$ and $\mu_{3}$, while $V_{u}$ only depends
on $\mu_{2}$, which is independent of  $\mu_{1}$ and $\mu_{3}$.
Therefore
\begin{equation}
\begin{array}{c}
\tilde{\mathbb{P}}_{0}\left[\left(L_{l}^{D_{t}}\right)_{l\in\left\{ 1,\ldots,L\right\} }\in A|L_{L}^{D_{t}}=0\right]=\tilde{\mathbb{P}}_{0}\left[f\left(L_{1}^{D_{t}}\right)|L_{L}^{D_{t}}=0\right],\\
\mbox{where }f\left(u\right)=\tilde{\mathbb{P}}_{0}\left[V_{u}\in A\right].
\end{array}\label{eq:SnickSnack}
\end{equation}
We are thus interested in the law of $V_{u}$. Let $\tilde{Y}_{t}$
be continuous time random walk on $\left\{ 1,\ldots,L\right\} $ with
local times and inverse local time at vertex $1$ given by
\[
\tilde{L}_{l}^{u}=\frac{1}{1+1_{\left\{ 1<l<L\right\} }}\int_{0}^{u}1_{\left\{ \tilde{Y}_{s}=l\right\} }ds\mbox{ and }\tilde{\tau}\left(t\right)=\inf\left\{ s\ge0:\tilde{L}_{1}^{s}>u\right\} .
\]
Sampling $\tilde{Y}_{t},t\ge0,$ by ``stiching together'' the excursions
in the point processes $\mu_{2}$ and $\mu_{3}$ we see that
\begin{equation}
\left(\tilde{L}_{l}^{\tilde{\tau}\left(u\right)}\right)_{l\in\left\{ 2\ldots,L\right\} }\overset{\mbox{law}}{=}\left(\sum_{\left(s,w\right)\in\mu_{2}\cup\mu_{3}:s\le u}L_{l}^{\infty}\left(w\right)\right)_{l\in\left\{ 2,\ldots,L\right\} }.\label{eq:TildeLaw}
\end{equation}
So by \prettyref{lem:RayKnightCont} (with $\left\{ 1,\ldots,L\right\} $
in place of $\left\{ 0,\ldots,L\right\} $) we have that
\begin{equation}
\begin{array}{l}
\tilde{\mathbb{P}}_{0}\left[\left({\displaystyle \sum_{\left(s,w\right)\in\mu_{2}\cup\mu_{3}:s\le u}}L_{l}^{\infty}\left(w\right)\right)_{l\in\left\{ 2,\ldots,L\right\} }\in\cdot\right]\\
=\tilde{\mathbb{P}}_{0}\left[\left(\tilde{L}_{l}^{\tilde{\tau}\left(u\right)}\right)_{l\in\left\{ 2\ldots,L\right\} }\in\cdot\right]=\mathbb{Q}_{2u}^{0}\left[\left(\frac{1}{2}X_{l}\right)_{l\in\left\{ 1,\ldots,L-1\right\} }\in\cdot\right].
\end{array}\label{eq:RayKnightTilde}
\end{equation}
Now since 
\[
\left\{ \mu_{3}\left(\left[0,t\right]\times\Gamma\right)=0\right\} =\left\{ \sum_{\left(s,w\right)\in\mu_{2}\cup\mu_{3}:s\le t}L_{L}^{\infty}\left(w\right)=0\right\} ,
\]
and $V_{u}$ is independent of $\mu_{3}$ we have
\[
\begin{array}{ccl}
\tilde{\mathbb{P}}_{0}\left[V_{u}\in A\right]=\tilde{\mathbb{P}}_{0}\left[V_{u}\in A|\mu_{3}\left(\left[0,t\right]\times\Gamma\right)=0\right] & \overset{\eqref{eq:DefOfVu},\eqref{eq:TildeLaw}}{=} & \mathbb{P}_{0}\left[\left(\tilde{L}_{L}^{\tau\left(u\right)}\right)_{l\in\left\{ 1,\ldots,,L\right\} }|\tilde{L}_{L}^{\tau\left(u\right)}=0\right]\\
 & \overset{\eqref{eq:RayKnightTilde}}{=} & \mathbb{Q}_{2u}^{0}\left[\left(\frac{1}{2}X_{l}\right)_{l\in\left\{ 0,\ldots,L-1\right\} }\in A|X_{L-1}=0\right]\\
 & \overset{\eqref{eq:BesselBridgeDef}}{=} & \mathbb{Q}_{2u\to0}^{0,L-1}\left[\left(\frac{1}{2}X_{l}\right)_{l\in\left\{ 0,\ldots,L-1\right\} }\in A\right].
\end{array}
\]
Plugging this into \eqref{eq:SnickSnack} gives the claim.
\end{proof}
The same construction of $Y_{t}$ from the Poisson point processes
$\mu_{1},\mu_{2}$ and $\mu_{3}$ can be used to \prettyref{lem:LOneGivenLLisZero},
which gives a control on the law of $L_{1}^{D_{t}}$ conditioned on
$L_{L}^{D_{t}}=0$.
\begin{proof}[Proof of \prettyref{lem:LOneGivenLLisZero}]
We will first show that
\begin{equation}
\begin{array}{c}
\mbox{the }\tilde{\mathbb{P}}_{0}\left[\cdot|L_{L}^{D_{t}}=0\right]-\mbox{law of }\frac{L}{L-1}L_{L}^{D_{t}}\mbox{\,\ is that of a sum of }t\\
\mbox{independent standard exponential random variables}.
\end{array}\label{eq:LawOfLDt0}
\end{equation}
In the notation of the proof of \prettyref{lem:RayishKnightish}:
Since $L_{1}^{D_{t}}=S_{t}$ and $\left\{ L_{L}^{D_{t}}=0\right\} =\left\{ \mu_{3}\left(\left[0,S_{t}\right]\times\Gamma\right)=0\right\} $
we are interested in the law of $S_{t}$ given $\left\{ \mu_{3}\left(\left[0,S_{t}\right]\times\Gamma\right)=0\right\} $.
Since $\mu_{3}$ is independent of $S_{t}$ we have that
\[
\tilde{\mathbb{P}}_{0}\left[S_{t}=ds,\mu_{3}\left(\left[0,S_{t}\right]\times\Gamma\right)=0\right]=\tilde{\mathbb{P}}_{0}\left[S_{t}=ds,\tilde{\mathbb{P}}\left[\mu_{3}\left(\left[0,s\right]\times\Gamma\right)=0\right]\right].
\]
The intensity of $\mu_{3}$ is the $\lambda\otimes\frac{1}{2}\tilde{\mathbb{P}}_{2}\left[\cdot,H_{L}<H_{1}\right]$
(recall \prettyref{eq:PPPIntensity}), so that
\begin{eqnarray*}
 & \tilde{\mathbb{P}}_{0}\left[\mu_{3}\left(\left[0,s\right]\times\Gamma\right)=0\right]=e^{-s\tilde{\mathbb{P}}_{2}\left[H_{L}<H_{1}\right]}=e^{-\frac{s}{L-1}},\mbox{ and}\\
 & \tilde{\mathbb{P}}_{0}\left[S_{t}=ds,\mu_{3}\left(\left[0,S_{t}\right]\times\Gamma\right)=0\right]=e^{-\frac{s}{L-1}}\tilde{\mathbb{P}}_{0}\left[S_{t}=ds\right].
\end{eqnarray*}
The $\tilde{\mathbb{P}}_{0}-$law of $S_{t}$ is the gamma distribution
with shape $t$ and scale $1$. Thus
\[
\tilde{\mathbb{P}}_{0}\left[S_{t}=ds,\mu_{3}\left(\left[0,S_{t}\right]\times\Gamma\right)=0\right]=e^{-\frac{s}{L-1}}\frac{s^{t-1}e^{-s}}{\left(t-1\right)!}=e^{-\left(\frac{1}{L-1}+1\right)s}\frac{s^{t-1}}{\left(t-1\right)!},
\]
so that the $\tilde{\mathbb{P}}_{0}\left[\cdot|\mu_{3}\left(\left[0,S_{t}\right]\times\Gamma\right)=0\right]-$law
of $S_{t}$ is the gamma distribution with shape $t$ and scale $\left(1+1/\left(L-1\right)\right)^{-1}=\left(L-1\right)/L$.
This proves \prettyref{eq:LawOfLDt0}.

Since $t\le10L^{2}$ the probability in \prettyref{eq:LowerBoundOnL1Prob}
is bounded below by 
\begin{equation}
\tilde{\mathbb{P}}_{0}\left[\sqrt{t\frac{L-1}{L}}-100\le\sqrt{L_{1}^{D_{t}}}\le\sqrt{t\frac{L-1}{L}}+100|L_{L}^{D_{t}}=0\right].\label{eq:EvenSmallerGuy}
\end{equation}
By \prettyref{eq:LawOfLDt0} and the Central Limit Theorem the $\tilde{\mathbb{P}}_{0}\left[\cdot|L_{L}^{D_{t}}=0\right]-$law
of $\left(\frac{L}{L-1}L_{L}^{D_{t}}-t\right)/\sqrt{t}$ converges
to a normal random variable as $t\to\infty$, uniformly in $L$. This
implies that \prettyref{eq:EvenSmallerGuy} is bounded below, so \prettyref{eq:LowerBoundOnL1Prob}
follows.

Also for $k=1,2$
\begin{equation}
\left|\tilde{u}\left(l\right)-v\right|^{k}\le c\left\{ \left|\tilde{u}\left(l\right)-\sqrt{2t}\left(\frac{L-l}{L}\right)^{3/2}\right|^{k}+\left|\sqrt{2t}\left(\frac{L-l}{L}\right)^{3/2}-u\left(l\right)\right|^{k}+\left|u\left(l\right)-v\right|^{k}\right\} .\label{eq:KindOfTriangle}
\end{equation}
We have
\[
\left|\sqrt{2t}\left(\frac{L-l}{L}\right)^{3/2}-u\left(l\right)\right|=u\left(l\right)\left|\sqrt{\frac{L-1}{L}}-1\right|\le cu\left(l\right)L^{-1}\le c\sqrt{t}L^{-1}\le c.
\]
Also
\[
\begin{array}{c}
\left|\tilde{u}\left(l\right)-\sqrt{t}\left(\frac{L-l}{L}\right)^{3/2}\right|=\frac{L-l}{L-1}\left|\tilde{u}\left(1\right)-\sqrt{t}\sqrt{\frac{L-1}{L}}\frac{L-1}{L}\right|\le c+\left|\tilde{u}\left(1\right)-\sqrt{t}\sqrt{\frac{L-1}{L}}\right|,\\
\mbox{and }\left|\tilde{u}\left(1\right)-\sqrt{t}\sqrt{\frac{L-1}{L}}\right|\le c\left|\sqrt{\frac{L}{L-1}L_{1}^{D_{t}}}-\sqrt{t}\right|\le c\frac{\left|\frac{L}{L-1}L_{1}^{D_{t}}-t\right|}{\sqrt{t}}.
\end{array}
\]
Taking the expectation in \prettyref{eq:KindOfTriangle} and using
\[
\tilde{\mathbb{E}}_{0}\left[\left|\frac{L}{L-1}L_{1}^{D_{t}}-t\right||L_{L}^{D_{t}}=0\right]\overset{\eqref{eq:LawOfLDt0}}{\le}c\sqrt{t},
\]
(also Cauchy-Schwarz if $k=2$) we get \prettyref{eq:ExpectedDeviation}.
\end{proof}
We remains to prove \prettyref{lem:ConditionedLawOfTraversals}, giving
a large deviation bound for the number of traversals $\tilde{T}_{l}^{t}$
(recall \eqref{eq:ZTraversals}) given the continuous local times
$L_{l}^{D_{t}}$. For this we will need the following computation
of the conditional distribution of $\tilde{T}_{l}^{t}$ (which can
be seen as a special case of the results of Section 4 \cite{DingAsympOfCTviaGFFBoundedDegree}).
To prove it we use the following fact about the modified Bessel function
of the first kind $I_{1}\left(\cdot\right)$:
\begin{equation}
\sum_{m\ge1}\frac{z^{m}}{m!\left(m-1\right)!}=\sqrt{z}I_{1}\left(2\sqrt{z}\right)\mbox{ for all }z\in\mathbb{R}.\label{eq:BesselFuncExpression}
\end{equation}

\begin{lem}
\label{lem:ConditionedLawOfTraversals}For all $u_{0},u_{1},u_{2},\ldots,u_{L}\in[0,\infty)$
such that $u_{i}=0\implies u_{i+1}=0$, and any $l\in\left\{ 1,\ldots,L-1\right\} $
such that $u_{l+1}>0$ we have for $m\in\left\{ 1,2,\ldots\right\} $
\[
\mathbb{\tilde{P}}_{0}\left[\tilde{T}_{l}^{t}=m|L_{l}^{D_{t}}=u_{l},l=0,\ldots,L\right]=\frac{\left(u_{l}u_{l+1}\right)^{m}/\left(m!\cdot\left(m-1\right)!\right)}{\sqrt{u_{l}u_{l+1}}I_{1}\left(2\sqrt{u_{l}u_{l+1}}\right)}.
\]
\end{lem}
\begin{proof}
The law of $T_{1}$ under $\mathbb{G}_{a}$ can be written down explicitly
as 
\[
\mathbb{G}_{a}\left[T_{1}=b\right]={a+b-1 \choose a-1}\left(\frac{1}{2}\right)^{a+b}\mbox{ for }a\in\left\{ 1,2,\ldots\right\} ,b\in\left\{ 0,1,2,\ldots\right\} ,
\]
since there are ${a+b-1 \choose a-1}$ ways to write $b$ as a sum
of $a$ non-negative integers, and since the probability that a geometric
random variable with support $\left\{ 0,1,\ldots\right\} $ and mean
$1$ takes on the value $k$ is $\left(\frac{1}{2}\right)^{k+1}$.
By \prettyref{lem:RayKnightDisc-1} we therefore have for all $t=t_{0},t_{1},t_{2},\ldots,t_{L-1}\in\left\{ 0,1,2,\ldots\right\} $
such that $t_{i}=0\implies t_{i+1}=0$ 
\[
\mathbb{\tilde{P}}_{0}\left[\tilde{T}_{i}^{t}=t_{i},i=0,\ldots,L-1\right]=\prod_{i\in\left\{ 1,\ldots,L-1\right\} :t_{i-1}>0}{t_{i-1}+t_{i}-1 \choose t_{i-1}-1}\left(\frac{1}{2}\right)^{t_{i-1}+t_{i}}.
\]
Contidioned on the number of visits to each vertex the total holding
times at the vertices are independent and gamma distributed, so we
have for such $t_{i}$ and any $u_{0},u_{1},\ldots,u_{L}\in[0,\infty)$
such that $t_{l-1}=0\iff u_{l}=0$ that
\[
\begin{array}{l}
\tilde{\mathbb{P}}_{0}\left[\tilde{T}_{l}^{t}=t_{l},l=1,\ldots,L-1,L_{l}^{D_{t}}=u_{l},l=0,\ldots,L\right]\\
=\left\{ {\displaystyle \prod_{i\in\left\{ 1,\ldots,L-1\right\} :t_{i-1}>0}}{t_{i-1}+t_{i}-1 \choose t_{i-1}-1}\left(\frac{1}{2}\right)^{t_{i-1}+t_{i}}\right\} \\
\quad\times\left\{ \left(\frac{e^{-u_{1}}u_{1}^{t_{0}-1}}{\left(t_{0}-1\right)!}\right)\left({\displaystyle \prod_{i\in\left\{ 1,\ldots,L-1\right\} :t_{i-1}>0}}\frac{e^{-2u_{i}}\left(2u_{i}\right)^{t_{i-1}+t_{i}}}{u_{i}\left(t_{i-1}+t_{i}-1\right)!}\right)\left(\frac{e^{-u_{L}}u_{L}^{t_{L-1}-1}}{\left(t_{L-1}-1\right)!}\right)\right\} ,
\end{array}
\]
where the quantity in the last paranethesis is interpreted as $1$
if $t_{l-1}=0$ or $u_{L}=0$. Exploiting two cancellations the right-hand
side equals
\[
\begin{array}{l}
\left\{ {\displaystyle \prod_{i\in\left\{ 1,\ldots,L-1\right\} :t_{i-1}>0}}\frac{1}{\left(t_{i-1}-1\right)!t_{i}!}\right\} \\
\times\left\{ \left(\frac{e^{-u_{1}}u_{1}^{t_{0}-1}}{\left(t_{0}-1\right)!}\right)\left({\displaystyle \prod_{i\in\left\{ 1,\ldots,L-1\right\} :t_{i-1}>0}}\frac{e^{-2u_{i}}u_{i}^{t_{i-1}+t_{i}}}{u_{i}}\right)\left(\frac{e^{-u_{L}}u_{L}^{t_{L-1}-1}}{\left(t_{L-1}-1\right)!}\right)\right\} .
\end{array}
\]
Considering only the terms that depend on $t_{l}$ we have that if
$u_{0},u_{1},\ldots,u_{l+1}>0$ 
\[
\tilde{\mathbb{P}}_{0}\left[\tilde{T}_{l}^{t}=m|L_{l}^{D_{t}}=u_{l},l=0,\ldots,L\right]=\frac{1}{\tilde{Z}}\frac{\left(u_{l}u_{l+1}\right)^{m}}{\left(m-1\right)!m!},m\ge1,l\in\left\{ 1,\ldots,L-1\right\} ,
\]
for a normalizing constant $\tilde{Z}$ depending only on $t,u_{0},\ldots,u_{L}$.
Using \prettyref{eq:BesselFuncExpression} we can identify the constant
as
\[
\tilde{Z}=\sum_{m\ge1}\frac{\left(u_{l}u_{l+1}\right)^{m}}{\left(m-1\right)!m!}=\sqrt{u_{l}u_{l+1}}I_{1}\left(2\sqrt{u_{l}u_{l+1}}\right).
\]

\end{proof}
We now prove the large deviation result \prettyref{lem:LDConditionedTraverals}
for the traversal process $\tilde{T}_{l}^{t}$ conditioned on $L_{l}^{D_{t}},l=0,\ldots,L$.
\begin{proof}[Proof of \prettyref{lem:LDConditionedTraverals}]
Denote $\tilde{\mathbb{P}}_{0}\left[\cdot|\sigma\left(L_{l}^{D_{t}}:l=0,\ldots,L\right)\right]$
by $\tilde{\mathbb{Q}}$. By \prettyref{lem:ConditionedLawOfTraversals},
\[
\tilde{\mathbb{Q}}\left[\exp\left(\lambda\tilde{T}_{l}^{t}\right)\right]=\sum_{m\ge1}\frac{\left(e^{\lambda}\mu^{2}\right)^{m}/\left(m!\cdot\left(m-1\right)!\right)}{\mu I_{1}\left(2\mu\right)}\overset{\eqref{eq:BesselFuncExpression}}{=}\frac{e^{\lambda/2}\mu I_{1}\left(2e^{\lambda/2}\mu\right)}{\mu I_{1}\left(2\mu\right)}\mbox{ for }\lambda\in\mathbb{R}.
\]
Thus for all $\lambda>0$ 
\[
\tilde{\mathbb{Q}}\left[\tilde{T}_{l}^{t}\ge\mu+\theta\right]\le e^{\lambda/2}\frac{I_{1}\left(2e^{\lambda/2}\mu\right)}{I_{1}\left(2\mu\right)}\exp\left(-\lambda\left(\mu+\theta\right)\right).
\]
Using the standard estimate $I_{1}\left(z\right)=\frac{e^{z}}{\sqrt{2\pi z}}\left(1+O\left(z^{-1}\right)\right)$
we have that
\[
I_{1}\left(2e^{\lambda/2}\mu\right)/I_{1}\left(2\mu\right)\le ce^{\lambda/4}e^{2\left(e^{\lambda/2}-1\right)\mu},
\]
so that for all $\lambda>0$
\[
\tilde{\mathbb{Q}}\left[\tilde{T}_{l}^{t}\ge\mu+\theta\right]\le ce^{c\lambda}\exp\left(2\left\{ e^{\lambda/2}-1\right\} \mu-\lambda\left\{ \mu+\theta\right\} \right)\le ce^{c\lambda}\exp\left(c\lambda^{2}\mu-\lambda\theta\right).
\]
Setting $\lambda=c\theta/\mu$ for a small enough $c$ the right-hand
side is bounded above by $ce^{c\theta/\mu-c\theta^{2}/\mu}$, giving
one half of \eqref{eq:LargeDeviationConditionedTraversal}. By estimating
$\tilde{\mathbb{Q}}\left[\exp\left(-\lambda\tilde{T}_{l}^{t}\right)\right]$
one can similarly show that $\tilde{\mathbb{Q}}\left[\tilde{T}_{l}^{t}\le\mu-\theta\right]\le ce^{c\theta/\mu-c\theta^{2}/\mu}$,
giving the other half.
\end{proof}

\end{document}